\documentclass{amsart}

\usepackage{amsthm,amssymb,latexsym,amsmath}
\usepackage[all]{xy}
\usepackage[english,russian]{babel}
\usepackage[utf8x]{inputenc}
\usepackage{xcolor}
\usepackage[colorlinks,linkcolor=blue,citecolor=blue,filecolor=blue,urlcolor=blue]{hyperref}
\usepackage{graphicx}
\graphicspath{{pictures/}}

\newcommand{\Coh}{\operatorname{Coh}}

\DeclareMathOperator{\Hom}{Hom}

\DeclareMathOperator{\Aut}{Aut}
\DeclareMathOperator{\coker}{coker}
\DeclareMathOperator{\im}{im}

\DeclareMathOperator{\rk}{{rk}}

\newcommand{\into}{\hookrightarrow}
\newcommand{\onto}{\twoheadrightarrow}

\newlength{\rrrr}

\newcommand{\intoo}[1]{\:
	\xymatrix@1{\ar@{^(->}[r]^{#1}&}\:}
\newcommand{\ontoo}[1]{\:
	\xymatrix@1{\ar@{->>}[r]^{#1}&}\:}

\def\Gr{\mathop{\mathrm{Gr}}\nolimits}

\def\id{\mathop{\mathrm{id}}\nolimits}

\def\R{\ensuremath{\mathbb{R}}}
\def\Z{\ensuremath{\mathbb{Z}}}

\def\FF{\ensuremath{\mathcal F}}

\def\TT{\ensuremath{\mathcal T}}

\makeatletter

\newtheorem{theorem}{Theorem}[section]

\newtheorem{proposition}[theorem]{Proposition}
\newtheorem{lemma}[theorem]{Lemma}

\newtheorem{corollary}[theorem]{Corollary}

\newtheorem{remark}[theorem]{Remark}

\usepackage{float}

\title[Rank Two Sheaves With Maximal Third Chern Class on Fano 3-fold $X_5$]{Rank Two Sheaves With Maximal Third Chern Class on the Fano Threefold of Index 2 and Degree 5}
\author{Danil A. Vassiliev}
\thanks{This work was supported by the Russian Science Foundation under grant no. 25-11-00214, https://rscf.ru/en/project/25-11-00214/.}

\begin{document}
	\maketitle
	\selectlanguage{english}
	\begin{abstract}
	We obtain a complete classification of rank two semistable sheaves with maximal third Chern class on the Fano threefold $X_5$ of index 2 and degree 5. Also, we describe moduli spaces of such sheaves in the general case of big second Chern class. These moduli spaces are smooth rational varieties. The study uses the theory of tilt-stability and Bridgeland stability conditions.
	\end{abstract}
	
	\section{Introduction}	
	In this work we finish the study started in \cite{Vass}. Namely, we obtain a complete description of Gieseker semistable rank two sheaves with maximal third Chern class $c_3$ on the Fano variety $X_5$ of index 2 and degree 5. Also, we obtain a description of moduli spaces of such sheaves in the general case (for $c_2$ big enough). These moduli spaces turn out to be smooth rational varieties.
	
	In general, moduli spaces of semistable sheaves with fixed Chern classes are difficult to describe. For example, the decomposition of the moduli scheme of semistable rank 2 sheaves on $\mathbb P^3$ with Chern classes $c_1=0,c_2=k,c_3=0$ into irreducible components is only known for $k=1, 2$, while for $k=3$ this scheme contains at least 11 irreducible components \cite[abstract]{Lavrov}. Benjamin Schmidt noted that in some cases the theory of stability conditions on derived categories allows to solve this problem. Namely, in \cite{Sch18} he classified rank two semistable sheaves on $\mathbb P^3$ with maximal $c_3$, and then in \cite{Sch23} generalized this result to sheaves of rank up to four. Also he described moduli spaces of such sheaves, which are always irreducible. Moreover, for any rank from one to four, any $c_1$, any $c_2\gg 0$ and maximal $c_3$ the moduli spaces are smooth and rational. Indeed, they are described either as locally trivial fibrations over a smooth rational base with a smooth rational fiber, or as a blow up of a smooth rational variety in a smooth locus.
	
    In a joint work with A.~S.~Tikhomirov \cite{Fano} we applied methods of Schmidt to describe rank two semistable sheaves with maximal $c_3$ on a smooth quadric threefold $X_2$. Moduli spaces of such sheaves are also smooth and rational in the general case. In both cases of $\mathbb P^3$ and $X_2$ the argument is based on the existence of full strong exceptional collections in derived categories $\mathrm D^b(\mathbb P^3)$ and $\mathrm D^b(X_2)$. Also the argument uses the \textit{generalized Bogomolov--Gieseker inequality}, which is proven for all Fano threefolds with Picard number one \cite{Li}. There are only two types of such threefolds besides $\mathbb P^3$ and $X_2$, which admit a full exceptional collection. Namely, these are the Fano threefold $X_5$ of index two and degree 5 and Fano threefolds $V_{22}$ of index one and degree 22 (see the proof of \cite[Lemma 3.5]{NVdB}). In \cite{Vass} and in the present work we consider the case of the variety $X_5$, which from now on will be denoted as $X$.
	
	The variety $X=X_5$ is a smooth Fano threefold of index 2 and degree 5, which determines it uniquely up to isomorphism \cite{Исковских}. Its Picard group is isomorphic to $\mathbb Z$. The variety $X$ can be constructed as a linear section of Grassmannian $\mathrm{Gr}(2,5)\subset\mathbb P^9$ (embedded by Pl\"ucker) by a general linear subspace of codimension 3. Denote by $\mathcal U$ the restriction to $X$ of the tautological rank 2 subbundle on $\mathrm{Gr}(2,5)$ and by $\mathcal Q$ the restriction to $X$ of the tautological rank 3 quotient bundle on $\mathrm{Gr}(2,5)$. The class of a divisor corresponding to the ample generator of the Picard group of $X$ will be denoted by $H$. By abuse of notation a hyperplane section of $X$ will also be denoted by $H$. 
	
	We denote by $L$ a line on $X$, and by $C$ a conic on $X$. If $L\subset H$, we denote by $\mathcal I_{L,H}$ the ideal sheaf of $L$ on $H$, which is included into a short exact sequence $0\to\mathcal I_{L,H}\to\mathcal O_H\to\mathcal O_L\to 0$, and analogously for $C\subset H$ we define the ideal sheaf $\mathcal I_{C,H}$. We also use the derived dualizing functor $\mathbb D(E):=\mathbf{R}\mathcal{H}om(E,\mathcal O_X)[1]$ for $E\in\mathrm D^b(X)$. In particular, if $E\cong \mathcal I_{L,H}$ or $E\cong \mathcal I_{C,H}$, then $\mathbb D(E)\cong\mathcal Ext^1(E,\mathcal O_X)$ is also a sheaf.
	
	The main result of the present article, partially proven in \cite{Vass}, is the classification of rank 2 semistable sheaves on $X$ with maximal $c_3$. Note that the cohomology groups $H^{2i}(X,\mathbb Z)$ are isomorphic to $\mathbb Z$ and the Chern classes of objects can be considered as integers. Let us define the following error term:
	
	\begin{equation}\label{error}
		\varepsilon(d)=
		\begin{cases}
			0, & \text{if } \mathrm{frac}(d)\in\{\frac 12,\frac 35\}\\
			\frac{2}{25}, & \text{if } \mathrm{frac}(d)\in\{0,\frac{1}{10},\frac 15,\frac{9}{10}\}\\
			\frac{3}{25}, & \text{if } \mathrm{frac}(d)\in\{\frac{3}{10},\frac 25,\frac{7}{10},\frac 45\},
		\end{cases}
	\end{equation}
	where $\mathrm{frac}(d)$ denotes the fractional part of a number $d\in\mathbb R$.
		
	\begin{theorem}\label{intro}Let $E$ be a Gieseker semistable sheaf of rank 2 on $X$ with Chern classes $c_1,c_2,c_3$. \\
	(1) If $c_1=-1$, then $c_2\ge 2$. \\
	(1.1) If $c_2=2$, then $c_3\le 0$. In case of equality we have $E\cong\mathcal U$. \\
	(1.2) If $c_2=3$, then $c_3\le 1$. In case of equality $E$ is included into an exact triple $0\to \mathcal U(-1)\to\mathcal O_X(-1)^{\oplus 4}\to E\to 0.$ \\
	(1.3) If $c_2=4$, then $c_3\le 2$. In case of equality $E$ is included into an exact sequence
	$0\to\mathcal O_X(-2)\to\mathcal O_X(-1)^{\oplus 3}\to E\to\mathcal O_L(-2)\to 0.$ \\
	(1.4) If $c_2=5$, then $c_3\le 5$. In case of equality $E$ is included into an exact triple $0\to \mathcal O_X(-2)\to\mathcal O_X(-1)^{\oplus 3}\to E\to 0.$ \\
	(1.5) If $c_2\ge 6$, then $c_3\le\frac{c_2^2}{5}-4c_2-10\varepsilon(\frac 12-\frac{c_2}{5}).$
	In case of equality $E$ is included into one of the following exact triples:\\
	(1.5.1) $0\to \mathcal O_X(-1)^{\oplus 2}\to E\to\mathcal I_{C,H}(\frac{2-c_2}{5})\to 0,$ if $c_2 \equiv 2 \pmod 5$;\\
	(1.5.2) $0\to \mathcal O_X(-1)^{\oplus 2}\to E\to\mathcal I_{L,H}(\frac{1-c_2}{5})\to 0,$ if $c_2 \equiv 1 \pmod 5$;\\
	(1.5.3) $0\to \mathcal O_X(-1)^{\oplus 2}\to E\to\mathcal O_H(-\frac{c_2}{5})\to 0,$ if $c_2 \equiv 0 \pmod 5$;\\
	(1.5.4) $0\to \mathcal O_X(-1)^{\oplus 2}\to E\to \mathbb D(\mathcal I_{L,H})(\frac{-6-c_2}{5})\to 0,$ if $c_2 \equiv 4 \pmod 5$;\\
	(1.5.5) $0\to \mathcal O_X(-1)^{\oplus 2}\to E\to \mathbb D(\mathcal I_{C,H})(\frac{-7-c_2}{5})\to 0,$ if $c_2 \equiv 3 \pmod 5$.\\
	(2) If $c_1=0$, then $c_2\ge 0$. \\
	(2.1) If $c_2=0$, then $c_3\le 0$. In case of equality $E\cong\mathcal O_X^{\oplus 2}$. \\
	(2.2) If $c_2=1$, then $c_3\le -2$, and in case of equality $E$ is included into an exact triple $0\to E\to\mathcal O_X^{\oplus 2}\to\mathcal O_L(1)\to 0$. \\
	(2.3) If $c_2=2$, then $c_3\le 0$. In case of equality $E$ is included into an exact triple $0\to \mathcal Q(-1)^{\oplus 2}\to\mathcal U^{\oplus 4}\to E\to 0.$ \\
	(2.4) If $c_2=3$, then $c_3\le 0$. In case of equality $E$ is the cohomology sheaf of a monad $\mathcal Q(-1)^{\oplus 3}\to\mathcal U^{\oplus 6}\to\mathcal O_X.$ \\
	(2.5) If $c_2=4$, then $c_3\le 4$. In case of equality $E$ is included into an exact triple $0\to\mathcal O_X(-1)^{\oplus 2}\to\mathcal U^{\oplus 2}\to E\to 0.$ \\
	(2.6) If $c_2=6$, then $c_3\le 8$. In case of equality $E$ is included into one of the following exact triples:
	$0\to\mathcal U\to E\to \mathbb D(\mathcal I_{L,H})(-2)\to 0, 0\to\mathcal U(-1)^{\oplus 2}\to\mathcal O_X(-1)^{\oplus 6}\to E\to 0$. \\
	(2.7) If $c_2=8$, then $c_3\le 14$. In case of equality $E$ is included into one of the following exact triples: $0\to F\to E\to \mathcal O_H(-1)\to 0, 0\to\mathcal U\to E\to \mathcal I_{L,H}(-1)\to 0, 0\to\mathcal O_X(-2)\to\mathcal Q(-1)\to E\to 0,$ where $F$ is the cokernel of a monomorphism of sheaves $\mathcal U(-1)\hookrightarrow\mathcal O_X(-1)^{\oplus 4}$ (from item (1.2) of this theorem). \\
	(2.8) If $c_2\in\{5,7\}$ or $c_2\ge 9$, then $c_3\le\frac{c_2^2+c_2+4}{5}-10\varepsilon(-\frac{c_2}{5})$. In case of equality $E$ is included into one of the following exact triples:\\
	(2.8.1) $0\to\mathcal U\to E\to \mathbb D(\mathcal I_{C,H})(\frac{-5-c_2}{5})\to 0$, if $c_2 \equiv 0 \pmod 5$; \\
	(2.8.2) $0\to\mathcal U\to E\to \mathcal I_{C,H}(\frac{4-c_2}{5})\to 0$, if $c_2 \equiv 4 \pmod 5$; \\
	(2.8.3) $0\to\mathcal U\to E\to \mathcal I_{L,H}(\frac{3-c_2}{5})\to 0$, if $c_2 \equiv 3 \pmod 5$; \\
	(2.8.4) $0\to\mathcal U\to E\to \mathcal O_H(\frac{2-c_2}{5})\to 0$, if $c_2 \equiv 2 \pmod 5$; \\
	(2.8.5) $0\to\mathcal U\to E\to \mathbb D(\mathcal I_{L,H})(\frac{-4-c_2}{5})\to 0$, if $c_2 \equiv 1 \pmod 5$.
	\end{theorem}
	
	Also, we describe moduli spaces of these rank 2 sheaves with maximal $c_3$ for the general case when $c_2$ is big enough. In this case the sheaves $E$ are constructed as extensions, in which the right term is one of sheaves $\mathcal I_{C,H},\mathcal I_{L,H},\mathcal O_H,\mathbb D(\mathcal I_{L,H})$ or $\mathbb D(\mathcal I_{C,H})$, twisted by $\mathcal O_X(m)$ for some $m\in\mathbb Z$. In order to construct the moduli spaces of the rank 2 sheaves $E$, we firstly need to construct moduli spaces of these torsion sheaves. We denote by $\mathcal M_X(v)$ the Gieseker--Maruyama moduli scheme of semistable sheaves on $X$ with Chern character $v$. We prove the following theorem.
	
	\begin{theorem}\label{torsion thm}
		\begin{enumerate}
			\item If $v=\mathrm{ch}(\mathcal I_{C,H})$, then $\mathcal M_X(v)=B_1=\mathrm{Gr}(2,5)$.
			\item If $v=\mathrm{ch}(\mathcal I_{L,H})$, then $\mathcal M_X(v)=B_2=\mathbb F$, which is a certain locally trivial $\mathbb P^4$-fibration over $\mathbb P^2$ defined by (\ref{bbF}).
			\item If $v=\mathrm{ch}(\mathcal O_H)$, then $\mathcal M_X(v)=B_3=\mathbb P^6$.
			\item If $v=\mathrm{ch}(\mathbb D(\mathcal I_{L,H}))$, then $\mathcal M_X(v)=B_4=B_2=\mathbb F$.
			\item If $v=\mathrm{ch}(\mathbb D(\mathcal I_{C,H}))$, then $\mathcal M_X(v)=B_5=B_1=\mathrm{Gr}(2,5)$.
		\end{enumerate}
	\end{theorem} 
	
	Now we come to the description of moduli spaces of rank 2 semistable sheaves on $X$. If such a sheaf $E$ has $c_1=-1,c_2\ge 6$ and maximal $c_3=\frac{c_2^2}{5}-10\varepsilon(\frac 12-\frac{c_2}{5})$, then by Theorem \ref{intro} it can be included into an exact sequence of the form
	\begin{equation}\label{extension2}0\to\mathcal O_X(-1)^{\oplus 2}\to E\to G_i\to 0,\quad 1\le i\le 5,
	\end{equation}
	where $G_1=\mathcal I_{C,H}(m)$ with $m\le -1$, $G_2=\mathcal I_{L,H}(m)$ with $m\le -1$, $G_3=\mathcal O_H(m)$ with $m\le -2$, $G_4=\mathbb D(\mathcal I_{L,H})(m)$ with $m\le -3$, $G_5=\mathbb D(\mathcal I_{C,H})(m)$ with $m\le -3$. Using the results of \cite{BPS} and \cite{Lan} we show that there are locally free sheaves $\mathcal A_i$ over $B_i$, such that the fiber of $\mathcal A_i$ over a point $p\in B_i$ is naturally isomorphic to the space $\mathrm{Ext}^1(G_i,\mathcal O_X(-1))$ for the sheaf $G_i$ corresponding to the point $p$ (see (\ref{A}) for a precise definition). We obtain the following theorem.
	
	\begin{theorem}
		\label{c=-1}(i) The Gieseker--Maruyama moduli scheme $\mathcal M_X(v)$ of Gieseker semistable sheaves $E$ on $X$ with Chern character defined from exact triple (\ref{extension2}) is isomorphic to the Grassmanization $\Pi:\mathcal Gr(\mathcal A_i^\vee,2)\to B_i$ of two-dimensional quotient spaces of fibers of $\mathcal A_i^\vee$. In particular, it is a smooth rational variety. 
		\\
		(ii) The dimension of $\mathcal M_X(v)$ equals $5m^2-9m+6$ for $i=1$, $5m^2-7m+4$ for $i=2$, $5m^2-5m+4$ for $i=3$, $5m^2+7m+4$ for $i=4$ and $5m^2+9m+6$ for $i=5$.
	\end{theorem}
	
	Analogously we deal with the case of even determinant. If a Gieseker semistable rank 2 sheaf $E$ on $X$ has $c_1=0,c_2\in\{5, 7\}\cup\mathbb Z_{\ge 9}$ and maximal $c_3=\frac{c_2^2+c_2+4}{5}-10\varepsilon(-\frac{c_2}{5})$, then by Theorem \ref{intro} it can be included into an exact sequence of the form
	\begin{equation}\label{extension3}0\to\mathcal U\to E\to G'_i\to 0,\quad 1\le i\le 5,
	\end{equation}
	where $G'_1=\mathcal I_{C,H}(m)$ with $m\le -1$, $G'_2=\mathcal I_{L,H}(m)$ with $m\le -2$, $G'_3=\mathcal O_H(m)$ with $m\le -1$, $G'_4=\mathbb D(\mathcal I_{L,H})(m)$ with $m\le -3$, $G'_5=\mathbb D(\mathcal I_{C,H})(m)$ with $m\le -2$ (here conditions on $m$ are different from those in (\ref{extension2}), therefore we use the notation $G'_i$ instead of $G_i$). Now we can form locally free sheaves $\widetilde{\mathcal A}_i$ over $B_i$ such that the fiber of $\widetilde{\mathcal A}_i$ over an arbitrary point $p\in B_i$ is naturally isomorphic to the space $\mathrm{Ext}^1(G'_i,\mathcal U)$ for the sheaf $G'_i$ corresponding to the point $p$ (see (\ref{tilde A})). We get the following theorem.
	
	\begin{theorem}
		\label{c=0}(i) The Gieseker--Maruyama moduli scheme $\mathcal M_X(v)$ of Gieseker semistable sheaves $E$ on $X$ with Chern character defined from exact triple (\ref{extension3}) is isomorphic to the projectivization $\widetilde{\Pi}:\mathbb P(\widetilde{\mathcal A}_i^\vee)\to B_i$. In particular, it is a smooth rational variety. 
	\\	
	(ii) The dimension of $\mathcal M_X(v)$ equals $5m^2-14m+14$ for $i=1$, $5m^2-12m+11$ for $i=2$, $5m^2-10m+10$ for $i=3$, $5m^2+2m+4$ for $i=4$ and $5m^2+4m+5$ for $i=5$.
	\end{theorem}
	
	Finally, we prove a conjecture that we posed in \cite[Conjecture 4.5]{Vass}. It states that, additionally to Theorem \ref{intro}, a general sheaf $E$ with maximal $c_3$ from the case (1.5) (resp. (2.10)) of this theorem corresponds to an extension \ref{extension2} (resp. \ref{extension3}), in which the sheaf $G_i$ (resp. $G'_i$) is a line bundle on a smooth hyperplane section $S\subset X$. Since $S$ is a del Pezzo surface of degree 5, we have $\mathrm{Pic}(S)\cong\mathbb Z^5$, and these line bundles can be given explicitly.
	
	In the proof of our main theorem \ref{intro} we follow Schmidt (\cite{Sch18}, \cite{Sch23}). Namely, we apply the theory of tilt-stability and Bridgeland stability conditions. In this theory we replace the category of coherent sheaves on $X$ by another heart of a bounded t-structure on the derived category $\mathrm D^b(X)$. We mostly work with a more simple notion of tilt-stability, in which this heart is the category $\mathrm{Coh}^\beta(X)$, consisting of certain two-term complexes. Its definition depends on a parameter $\beta\in\mathbb R$. Tilt-stability conditions, also called $\nu_{\alpha,\beta}$-stability conditions, depend also on a positive real parameter $\alpha$. We obtain a locally finite wall-and-chamber structure in the upper $(\beta,\alpha)$-half-plane, such that the set of $\nu_{\alpha,\beta}$-semistable objects of $\mathrm{Coh}^\beta(X)$ does not change when $(\beta,\alpha)$ varies within a chamber. 
	It turns out that Gieseker semistable sheaves $E$ of positive rank belong to the category $\mathrm{Coh}^\beta(X)$ and are $\nu_{\alpha,\beta}$-semistable for $\beta<\mu(E),\alpha\gg 0$ (see Proposition \ref{2-stability}). Any wall $W$, at which the stability of an object $E$ changes, is induced by a short exact sequence
	$$0\to F\to E\to G\to 0,$$
	in $\mathrm{Coh}^\beta(X)$, such that $\nu_{\alpha,\beta}(E)=\nu_{\alpha,\beta}(F)$ for $(\beta,\alpha)\in W$. In most cases from Theorem \ref{intro} the short exact sequences, in which $E$ is included, are obtained from the classification of walls, at which $E$ can be destabilized. In the general case when $c_2$ is big enough, there is only one wall for $E$, and generalized Bogomolov--Gieseker inequality (Proposition \ref{BMT}) implies that there are no tilt-semistable objects below this wall. In several remaining special cases with low $c_2$ one also needs to use the full strong exceptional collection in $\mathrm D^b(X)$ to describe semistable sheaves.
	
	The structure of the paper is as follows. In Section \ref{derived category} we recall some of the notions and results from the theory of stability conditions on derived categories that we will use, as well some facts about the variety $X$. Also, in Lemmas \ref{subobject} and \ref{quotient} we provide a justification and a slight generalization of a technique used by Schmidt in \cite{Sch23}. In Section \ref{r=1} we give a characterization of ideal sheaves of lines and conics on $X$ as tilt-semistable objects with maximal third Chern character $\mathrm{ch}_3$. Then in Section \ref{rank 0} we describe rank zero objects with $c_1=1$ and maximal $\mathrm{ch}_3$. In Section \ref{rank 2} we prove Theorem \ref{main} on rank 2 tilt-semistable objects with maximal $\mathrm{ch}_3$ and deduce our main Theorem \ref{intro} as a corollary. Then in Section \ref{special} we consider the question of existence of semistable sheaves with maximal $c_3$ from Theorem \ref{intro} in some special cases. In Section \ref{torsion} we prove Theorem \ref{torsion thm} on moduli spaces of torsion sheaves. In Sections \ref{section -1} and \ref{section 0} we prove Theorems \ref{c=-1} and \ref{c=0} on moduli spaces of rank two sheaves on $X$ with maximal $c_3$ in general cases and deduce that bounds from Theorem \ref{intro} are exact. Finally, in Section \ref{conjecture} we obtain a proof of a conjecture that we posed in \cite{Vass}.
	
	\textbf{Acknowledgments.} The author thanks God for the strength to complete this work. Also the author thanks Alexander S. Tikhomirov for encouragement and attention to the work, and for his help with proving Proposition \ref{conj proof}.
	
	\textbf{Notation.}
	\begin{center}
		\begin{tabular}{ r l }
			$\mathbb C$ & base field\ \\
			$X=X_5$ & smooth section of the Grassmannian $\mathrm{Gr}(2,5)$ embedded by\\
			& Pl\"ucker into the space $\mathbb P^9$, by a linear subspace $\mathbb P^6$\\
			$\mathcal U$ & restriction to $X$ of tautological subbundle on $\mathrm{Gr}(2,5)$\\
			$\mathcal Q$ & restriction to $X$ of tautological quotient bundle on $\mathrm{Gr}(2,5)$\\
			$H$ & positive generator of the Picard group $\mathrm{Pic}\ X\simeq\Z$ --\\ &
			class of hyperplane section of $X\hookrightarrow
			\mathbb P^6$,\\
			$\mathrm{Coh}(X)$ & category of coherent sheaves on $X$\\
			$\mathrm D^b(X)$ & bounded derived category of coherent sheaves on
			$X$\\
			$\mathcal H^{i}(E)$ & $i$th cohomology sheaf of the complex $E\in\mathrm
			D^b(X)$ \\
			$H^{i}(E)$ & $i$th hypercohomology group $\mathbf R^i\Gamma(E)$ of $E\in\mathrm
			D^b(X)$\\
			$\mathrm{ch}(E)$ & Chern character of the object $E \in\mathrm D^b(X)$ \\
			$\mathrm{ch}_{\le m}(E)$ & $\mathrm{ch}_0(E)+ \ldots+ \mathrm{ch}_m(E)$ \\
			$\mathbb D(E)$ & derived dual $\mathbf{R}\mathcal{H}om(E,\mathcal O_X)[1]$ of $E\in\mathrm
			D^b(X)$ \\
			$\mathrm{hom}(E,F)$ & $\mathrm{dim} \Hom(E,F)$ for $E,F\in\mathrm D^b(X)$ \\
			$\mathrm{ext}^i(E,F)$ & $\mathrm{dim} \Hom(E,F[i])$ for $E,F\in\mathrm D^b(X),i\in\mathbb Z$ \\
			$\mathcal M_X(v)$ & moduli scheme of Gieseker semistable sheaves on $X$ with Chern character $v$
		\end{tabular}
	\end{center}
	
	\section{General Notions}
	\label{derived category}
	For convenience of the reader we include here some of the general notions and results, which we already cited in our earlier articles.
	
	Recall that $X=X_5$ is the unique smooth Fano threefold of Picard number one, index 2 and degree 5. The variety $X$ can be constructed as a codimension 3 linear section of Grassmannian $\mathrm{Gr}(2,5)$ embedded by Pl\"ucker in $\mathbb P^9$. Cohomology groups of $X$ are 
	$$H^*(X,\mathbb Z)=H^0(X,\mathbb Z)\oplus H^2(X,\mathbb Z)\oplus H^4(X,\mathbb Z)\oplus H^6(X,\mathbb Z)=$$
	$$=\Z[X]\oplus\Z[H]\oplus\Z[L]\oplus\Z[P],$$
	where $H$ denotes the class of a hyperplane section of $X$, $L$ is the class of a projective line on $X$ and $P$ is the class of a point. We have $H^2=5L,H^3=5P$.
	
	The \textit{slope} of a coherent sheaf $E\in\Coh(X)$ is defined as $\mu(E)=\frac{
		H^2\cdot\mathrm{ch}_1(E)}{H^3\cdot\mathrm{ch}_0(E)}$, division by 0 is understood as giving $+\infty$. A coherent sheaf $E$ is called
	$\mu$-\textit{stable} (resp. $\mu$\textit{-semistable}) or \textit{slope stable} (resp. \textit{slope semistable}) if for any proper
	subsheaf $0\ne F\hookrightarrow E$ the inequality $\mu(F)<\mu(E/F)$ (resp. $\mu(F)\le\mu(E/F)$) holds.
	
	Following \cite{BMT} and \cite{Sch18}, we obtain a new heart of a bounded t-structure on $\mathrm{D}^b(X)$ using the process of tilting. 
	For an arbitrary real number $\beta$ we have the following torsion pair in the category of coherent sheaves:
	$$\mathcal T_\beta=\{E\in\Coh(X)\colon\mathrm{any}\; \mathrm{quotient}\;E\to G\;\mathrm{satisfies}\;\mu(G)>\beta\},$$
	$$\mathcal F_\beta=\{E\in\Coh(X)\colon\mathrm{any}\;\mathrm{subsheaf}\;0\neq F\to E\;\mathrm{satisfies}\;\mu(F)\le\beta\}.$$
	
	Slope stability has the Harder-Narasimhan property, meaning that for any coherent sheaf $E$ there is a filtration
	$$0=E_0\subset E_1\subset\ldots\subset E_n=E,$$
	such that each $F_i:=E_i/E_{i-1}$ is slope semistable and 
	$\mu(F_i)>\mu(F_{i+1})$ for all $i$. 
	Denote $\mu_{\min}(E):=\mu(F_n),\mu_{\max}(E):=\mu(F_1)$. 
	The category $\mathcal T_\beta$ consists precisely of those $E\in\Coh(X)$ for which $\mu_{\min}(E)>\beta$. Respectively, $\mathcal F_\beta$ consists of those $E\in\Coh(X)$ for which $\mu_{\max}(E)\le\beta$ (\cite{BMT}).
	
	The new heart $\Coh^\beta(X)$ is defined as the extension closure $\langle\mathcal F_\beta[1],\mathcal T_\beta\rangle$. Objects of $\Coh^\beta(X)$ can be identified with two-term complexes $E^{-1}\overset{d}{\to} E^0$, such that $\ker d\in\mathcal F_\beta,\coker d\in\mathcal T_\beta$ (\cite[Section 2.3]{BMT}).
	
	Recall the notion of tilt-stability. The twisted Chern character is defined as $\mathrm{ch}^\beta=e^{-\beta H}\cdot\mathrm{ch}$. Explicitly, for $E\in \mathrm D^b(X)$ with $\mathrm{ch}(E)=r+cH+dH^2+eH^3$ we have
	$$\mathrm{ch}_0^\beta(E)=r,\; \mathrm{ch}_1^\beta(E)=(c-\beta r)H,\; \mathrm{ch}_2^\beta(E)=(d-\beta c+\frac{\beta^2}{2} r)H^2,$$
	$$\mathrm{ch}_3^\beta(E)=(e-\beta d+\frac{\beta^2}{2}c-\frac{\beta^3}{6}r)H^3.$$
	
	Let $\alpha>0$ be a positive real number. Following Schmidt (\cite{Sch14,Sch19}), for an object $E\in\Coh^\beta(X)$ we define a central charge function
	$$Z^{\mathrm{tilt}}_{\alpha,\beta}(E)=Z^{\mathrm{tilt}}_{\alpha,\beta}(\mathrm{ch}_0(E),\mathrm{ch}_1(E),\mathrm{ch}_2(E))=-H\cdot\mathrm{ch}_2^\beta(E)+\frac{\alpha^2}{2}H^3\cdot\mathrm{ch}_0^\beta(E)+i(H^2\cdot\mathrm{ch}_1^\beta(E)),$$
	which is a $\mathbb Z$-linear map $Z^{\mathrm{tilt}}_{\alpha,\beta}:K(\mathrm D^b(X))=K(\Coh^\beta(X))\to\mathbb C$.
	
	The tilt-slope of $E$ is defined by
	$$\nu_{\alpha,\beta}(E)=-\frac{\Re(Z^{\mathrm{tilt}}_{\alpha,\beta}(E))}{\Im(Z^{\mathrm{tilt}}_{\alpha,\beta}(E))},$$
	where division by 0 is interpreted as giving $+\infty$. The object $E$ is called \textit{tilt-(semi)stable} (or $\nu_{\alpha,\beta}$-\textit{(semi)stable}) if for any non-trivial proper subobject $F\hookrightarrow E$ in $\mathrm{Coh}^\beta(X)$ the inequality $\nu_{\alpha,\beta}(F)<(\le)\,\nu_{\alpha,\beta}(E/F)$ holds.
	
	The \textit{Bogomolov--Gieseker inequality} holds for tilt-stability (see \cite[Theorem 2.1]{MS18}):
	\begin{proposition}\label{BG} For any $\nu_{\alpha,\beta}$-semistable object $E\in\Coh^\beta(X)$ or a slope semistable sheaf $E\in\mathrm{Coh}(X)$ we have
		\begin{multline}\overline{\Delta}_H(E):=
			(H^2\cdot\mathrm{ch}_1(E))^2-2(H^3\cdot\mathrm{ch}_0(E))(H\cdot\mathrm{ch}_2(E))=\\=(H^2\cdot\mathrm{ch}_1^\beta(E))^2-2(H^3\cdot\mathrm{ch}_0^\beta(E))(H\cdot\mathrm{ch}_2^\beta(E))\ge 0.
		\end{multline}
	\end{proposition}
	
	In order to construct Bridgeland stability conditions on $X$ we make another tilt (as in \cite{BMT}). Let
	$$\mathcal T'_{\alpha,\beta}=\{E\in\Coh^\beta(X)\colon\text{any quotient}\, E\onto G\;\mathrm{satisfies}\;\nu_{\alpha,\beta}(G)>0\},$$
	$$\mathcal F'_{\alpha,\beta}=\{E\in\Coh^\beta(X)\colon\text{any non-trivial subobject}\, F\into E\;\mathrm{satisfies}\;\nu_{\alpha,\beta}(F)\le 0\},$$
	and set $\mathcal A^{\alpha,\beta}(X)=\langle \mathcal F'_{\alpha,\beta}[1],\mathcal T'_{\alpha,\beta}\rangle$. For any $s\in\mathbb R$ define
	$$Z_{\alpha,\beta,s}:=-\mathrm{ch}_3^\beta+s\alpha^2H^2\cdot\mathrm{ch}_1^\beta+i (\alpha H\cdot\mathrm{ch}_2^\beta-\frac{\alpha^3H^3}{2}\cdot\mathrm{ch}_0^\beta),$$
	$$\lambda_{\alpha,\beta,s}:=-\frac{\Re(Z_{\alpha,\beta,s})}{\Im(Z_{\alpha,\beta,s})},$$
	division by 0 is again interpreted as giving $+\infty$.
	
	An object $E \in \mathcal A^{\alpha, \beta}(X)$ is called \textit{$\lambda_{\alpha, \beta, s}$-(semi)stable}, if for any non-trivial subobject $F \hookrightarrow E$ we have $\lambda_{\alpha, \beta, s}(F) < (\le)\  \lambda_{\alpha, \beta, s}(E)$.	
	
	It is known that for $s>\frac 16$ $(\mathcal A^{\alpha,\beta}(X),Z_{\alpha,\beta,s})$ is a Bridgeland stability condition \cite[Corollary 0.2]{Li}. This result follows from the fact that the following \textit{generalized Bogomolov-Gieseker inequality} holds for $X$ (\cite[Theorem 0.1]{Li}, \cite[Theorem 5.4]{BMS16}):
	\begin{proposition}\label{BMT} Assume that $E$ is $\nu_{\alpha,\beta}$-semistable. Then
		$$Q_{\alpha,\beta}(E)=\alpha^2\overline{\Delta}_H(E)+4(H\cdot\mathrm{ch}_2^\beta)^2-6(H^2\cdot\mathrm{ch}_1^\beta(E))\mathrm{ch}_3^\beta(E)\ge 0.$$ \end{proposition}
	
	Fix $v\in\Lambda=\mathbb Z^2\oplus\frac{1}{10}\mathbb Z$. There is a locally finite wall-and-chamber structure in the upper half-plane $\mathbb H:=\{(\beta,\alpha)\in\mathbb R^2\ |\ \alpha>0\}$ such that for an object $E\in\mathrm{Coh}^\beta(X)$ with $\mathrm{ch}_{\le 2}(E)=v$ tilt-stability of $E$ does not change when $(\beta,\alpha)$ vary within a chamber (\cite[Proposition B.5]{BMS16}). A \textit{numerical wall} with respect to $v\in\Lambda$ is a non-trivial proper subset $W$ of $\mathbb H$ given by an equation of the form $\nu_{\alpha,\beta}(v)=\nu_{\alpha,\beta}(w)$ for another element $w\in\Lambda$. We denote this numerical wall by $W(v,w)$. A numerical wall $W$ is called an \textit{actual wall} (or simply a \textit{wall}) if the set of semistable objects with class $v$ changes at $W$. The structure of walls in tilt-stability is well understood. We have the following Proposition.
	
	\begin{proposition}[Numerical properties of walls]\label{numerical} Let $v=(v_0,v_1,v_2)\in\Lambda$ be a fixed class with $\overline{\Delta}_H(v)\ge 0$. All numerical walls in the following statements are with respect to $v$.
		\begin{enumerate} 
			\item The numerical wall $W(v,w)$ is given by
			$$x(\alpha^2+\beta^2)+y\beta+z=0,$$
			where $x=v_0w_1-v_1w_0,y=2(v_2w_0-v_0w_2),z=2(v_1w_2-v_2w_1)$. In particular, numerical walls are either semicircles with center on the $\beta$-axis or rays parallel to the $\alpha$-axis. 
			\item If $v_0\neq 0$, there is exactly one numerical vertical wall given by $\beta = v_1/v_0$. If $v_0 = 0$, there are no vertical walls.
			\item The curve $\nu_{\alpha, \beta}(v) = 0$ intersects all semicircular walls at their highest point.
			\item If $v_0\neq 0$, then the curve $\nu_{\alpha, \beta}(v) = 0$ is a hyperbola, which may be degenerate if $\overline{\Delta}_H(v)=0$. Its asymptotes are the lines $\beta-\alpha=v_1/v_0$ and $\beta+\alpha=v_1/v_0$. The semicircular numerical walls are nested along one of the two branches of the hyperbola. 
			\item If there exist semicircular actual walls, then there is a largest such wall.
			\item If $\overline{\Delta}_H(v)>0$, then the set of $(\beta,\alpha)\in\mathbb H$ for which $Q_{\alpha,\beta}(v_0,v_1H,v_2H^2,v_3H^3)=0$ is the semicircular numerical wall $W((v_0,v_1,v_2),(5v_1,10v_2,15v_3))$. The locus where $Q_{\alpha,\beta}(v_0,v_1H,v_2H^2,v_3H^3)<0$ is the semidisk bounded by\linebreak $Q_{\alpha,\beta}(v_0,v_1H,v_2H^2,v_3H^3)=0$ and $\beta$-axis.
		\end{enumerate}
	\end{proposition}
	
	If an object $E\in\mathrm{Coh}^\beta(X),$ $\mathrm{ch}_{\le 2}(E)=v$ is destabilized at the wall $W$ (that is, $W$ belongs to the boundary of the set of those $(\beta,\alpha)$, for which $E$ is $\nu_{\alpha,\beta}$-semistable), 
	then there exists a short exact sequence $0\to F\to E\to G\to 0$ in $\mathrm{Coh}^\beta(X)$ for $(\beta,\alpha)\in W$ such that $$W=\{(\beta,\alpha)\in\mathbb H\ |\ \nu_{\alpha,\beta}(E)=\nu_{\alpha,\beta}(F)\}$$ (see \cite[Proposition B.5]{BMS16}). In this situation we also say that $E$ \textit{is destabilized by the exact triple} $0\to F\to E\to G\to 0$. We denote a wall given by such an equation by $W(E,F)$. Note that the object $E$, and hence objects $F$ and $G$, are $\nu_{\alpha,\beta}$-semistable for $(\beta,\alpha)\in W(E,F)$. 
	
	If $W(E,F)$ is semicircular, we denote its radius by $\rho(E,F)$ 
	and the $\beta$-coordinate of its center by $s(E,F)$. The numerical wall with equation $Q_{\alpha,\beta}(E)=0$ will be denoted by $W_Q(E)$. If this wall is semicircular, then its radius and the $\beta$-coordinate of its center will be denoted by $\rho_Q(E)$ and $s_Q(E)$, respectively. 
	We will also use the properties of walls described in the following Proposition (see \cite[Proposition 2.4]{Vass}).
	
	\begin{proposition}[Properties of actual walls]\label{actual}Let $0\to F\to E\to G\to 0$ be an exact sequence in $\mathrm{Coh}^\beta(X)$ defining a wall $W$.
		\begin{enumerate}
			\item If $W$ is semicircular and $\mathrm{ch}_0(F)>\mathrm{ch}_0(E)\ge 0$, then
			$$\rho(E,F)^2\le\frac{\overline{\Delta}_H(E)}{4H^3\cdot\mathrm{ch}_0(F)(H^3\cdot(\mathrm{ch}_0(F)-\mathrm{ch}_0(E))}.$$
			\item If $W$ is semicircular, then $\overline{\Delta}_H(F)+\overline{\Delta}_H(G)<\overline{\Delta}_H(E).$
			\item For any point $(\beta,\alpha)\in W$, we have $0\le\mathrm{ch}_1^\beta(F)\le\mathrm{ch}_1^\beta(E)$. In case of equality on either end, the wall is vertical. In particular, if $H^2\cdot\mathrm{ch}_1^{\beta_0}(E)>0$ has the minimal positive value among objects of $\mathrm{Coh}^{\beta_0}(X)$, then there is no actual wall that intersects the ray $\beta=\beta_0$.
			\item If $\mathrm{ch}_0(E)>0$, then either $\mathrm{ch}_0(F)>0$ and $\mu(F)\le\mu(E)$ or $\mathrm{ch}_0(G)>0$ and $\mu(G)\le\mu(E)$. 
			\item Suppose that $W$ is semicircular. If $\mathrm{ch}_0(E)>0, \mathrm{ch}_0(F)>0$ and $\mu(F)\le\mu(E)$, then $\beta_-(E)<\beta_-(F)\le\mu(F)<\mu(E)$. Respectively, if $\mathrm{ch}_0(E)>0, \mathrm{ch}_0(G)>0$ and $\mu(G)\le\mu(E)$, then $\beta_-(E)<\beta_-(G)\le\mu(G)<\mu(E)$. 
		\end{enumerate}
	\end{proposition}
	
	We have the following lemma (\cite[Lemma 3.5]{Sch19}) that refines the last statement of Proposition \ref{actual}~(3).
	
	\begin{lemma}\label{minimal} If an object $E\in\mathrm{Coh}^\beta(X)$ is $\nu_{\alpha,\beta}$-semistable for $\alpha\gg 0$, then one of the following holds:\\
		(a) $\mathcal H^{-1}(E)=0$ and $\mathcal H^0(E)$ is a $\mu$-semistable pure sheaf of dimension $\ge 2$.\\
		(b) $\mathcal H^{-1}(E)=0$ and $\dim \mathcal H^0(E)\le 1$.\\
		(c) $\mathcal H^{-1}(E)$ is a torsion-free $\mu$-semistable sheaf and $\dim \mathcal H^0(E)\le 1$. Also, if $\nolinebreak{\beta>\mu(E)}$, then there are no non-zero morphisms from sheaves of dimension $\le 1$ to $E$.\\
		Moreover, if $H^2\cdot\mathrm{ch}_1^\beta(E)$ is either zero or has the minimal positive value among objects of $\mathrm{Coh}^\beta(X)$, then $E$ is $\nu_{\alpha,\beta}$-semistable if and only if it satisfies one of the three properties above.
	\end{lemma}
	
	We will also use the following relation between tilt-stability and the more classical notions of stability. Let $f, g \in \R[m]$ be polynomials. If $\deg(f) < \deg(g)$, then we set
	$f > g$. If $\deg(f) = \deg(g)$ and $a$, $b$ are the leading coefficients in $f$,
	$g$ respectively, then we put $f < (\leq)\ g$ if $\frac{f(m)}{a} < (\leq)\
	\frac{g(m)}{b}$ for all $m \gg 0$.
	For an arbitrary sheaf $E\in\Coh(X)$ we define the numbers $a_i(E)$ for $i \in \{0,
	1, 2, 3\}$ through the Hilbert polynomial $P(E, m):=\chi(E(m))=a_3(E)m^3+a_2(E) m^2+
	a_1(E) m + a_0(E)$. In addition, we set $P_2(E,m):=a_3(E)m^2+a_2(E)m+a_1(E)$.
	The sheaf $E \in \Coh(X)$ is called \emph{Gieseker (semi)stable} or simply \emph{(semi)stable}, if for any
	its proper subsheaf $0\ne F \into E$ the inequality $P(F, m) < (\leq)\ P(E/F, m)$ holds. Respectively, $E$ is called \emph{2-(semi)stable}
	if for any subsheaf $0\ne F \into E$ the inequality $P_2(F, m) < (\leq)\ P_2(E/F, m)$ holds.
	Stability, 2-stability and $\mu$-stability of a sheaf are related by the following implications:
	\begin{equation*}
		\xymatrix{
			\text{$\mu$-stability} \ar@{=>}[r] & \text{$2$-stability} \ar@{=>}[r] & \text{stability} \ar@{=>}[d] \\
			\text{$\mu$-semistability} & \text{$2$-semistability} \ar@{=>}[l] &\text{semistability} \ar@{=>}[l]
		}
	\end{equation*}
	
	\begin{proposition}[{\cite[Proposition 2.7]{Vass}}]\label{2-stability}
		An object $E\in\mathrm D^b(X)$ belongs to $\mathrm{Coh}^\beta(X)$ and is $\nu_{\alpha,\beta}$-(semi)stable for $\beta<\mu(E)<+\infty$ 
		and $\alpha\gg0$ iff $E$ is a 2-(semi)stable sheaf of positive rank.
	\end{proposition}
	
	In our arguments we also use the following simple property of tilt-stability.
	
	\begin{lemma}\label{tilt-stable}If $E\in\mathrm{Coh}^\beta(X)$ is $\nu_{\alpha,\beta}$-semistable for all points $(\beta,\alpha)$ in an open set $U\subset\mathbb H$, then either $E$ is $\nu_{\alpha,\beta}$-stable for all $(\beta,\alpha)$ in $U$, or $E$ is destabilized by a subobject $F\hookrightarrow E$, such that $\mathrm{ch}_{\le 2}(F)$ is proportional to $\mathrm{ch}_{\le 2}(E/F)$. If moreover the vector $\mathrm{ch}_{\le 2}(E)\in\Lambda$ is primitive, then $E$ is tilt-stable.
	\end{lemma}
	
	\begin{proof} Assume for the converse that there is such a point $(\beta_0,\alpha_0)$ in the interior of $U$, such that $E$ is strictly $\nu_{\alpha_0,\beta_0}$-semistable. If $(\beta_0,\alpha_0)$ lies on a semicircular wall $W$ for $E$, then $E$ is not tilt-semistable at least on one side of $W$, contradicting our assumptions. Hence, there is a subobject $F\hookrightarrow E$, such that $\nu_{\alpha,\beta}(F)=\nu_{\alpha,\beta}(E/F)$ for all $(\beta,\alpha)\in\mathbb H$. In particular,
	\begin{equation}\label{F E}\frac{H\cdot\mathrm{ch}_2^{\beta_0}(F)-\frac{\alpha^2}{2}H^3\cdot\mathrm{ch}_0^{\beta_0}(F)}{H^2\cdot\mathrm{ch}_1^{\beta_0}(F)}=\frac{H\cdot\mathrm{ch}_2^{\beta_0}(E/F)-\frac{\alpha^2}{2}H^3\cdot\mathrm{ch}_0^{\beta_0}(E/F)}{H^2\cdot\mathrm{ch}_1^{\beta_0}(E/F)}\end{equation}
	for all $\alpha>0$. Since $E$ is tilt-semistable in a neighborhood of $(\beta_0,\alpha_0)$, we have $\mathrm{ch}_1^{\beta_0}(E)\neq 0$, and (\ref{F E}) implies that $\mathrm{ch}_{\le 2}^{\beta_0}(F)$ is proportional to $\mathrm{ch}_{\le 2}^{\beta_0}(E/F)$ with a nonzero real coefficient. Hence, $\mathrm{ch}_{\le 2}(F)$ is proportional to $\mathrm{ch}_{\le 2}(E/F)$. Since $F$ and $E/F$ belong to $\mathrm{Coh}^{\beta_0}(X)$, the coefficient of proportionality cannot be negative. And if $\mathrm{ch}_{\le 2}(E)=\mathrm{ch}_{\le 2}(F)+\mathrm{ch}_{\le 2}(E/F)$ is primitive, then it cannot be positive, hence in this case $E$ is tilt-stable.
	\end{proof}
	
	For a tilt-semistable object $E$ with $\mathrm{ch}_0(E)\neq 0$, following \cite{BMSZ}, we define
	$$\beta_-(E)=\mu(E)-\sqrt{\frac{\overline{\Delta}_H(E)}{(H^3\cdot\mathrm{ch}_0(E))^2}},$$
	$$\beta_+(E)=\mu(E)+\sqrt{\frac{\overline{\Delta}_H(E)}{(H^3\cdot\mathrm{ch}_0(E))^2}}.$$
	
	The numbers $\beta_-(E)$ and $\beta_+(E)$ are the two solutions of the equation $\nu_{0,\beta}(E)=0$. We will use the following theorem (already used in \cite[Theorem 2.8]{Vass}) to bound the rank of destabilizing subobjects.
	
	\begin{theorem}\label{rank bound}Let $E$ be a tilt-stable object on $X$ with $\mathrm{ch}_0(E)\neq 0$, which is not a shift of a line bundle or a twisted ideal sheaf of points. If there is an integer $n\in\mathbb Z$ such that $\beta_-(E),\beta_+(E)\in[n,n+1)$, then
		$$\overline{\Delta}_H(E)\ge\mathrm{ch}_0(E)^2.$$
	\end{theorem}
	
	The following lemma (\cite[Lemma 2.9]{Vass}) will allow us to classify exact triples which may destabilize a given tilt-semistable object using Theorem \ref{rank bound}.
	
	\begin{lemma}\label{stable}Let $E$ be a tilt-semistable object with $\mathrm{ch}_0(E)>0$. If $E$ is destabilized at a semicircular wall $W$, then $W$ is induced by a tilt-stable subobject or quotient $F$ with $\mathrm{ch}_0(F)>0,\mu(F)<\mu(E).$
	\end{lemma}
	
	In this article we mostly work with tilt-stability. The following lemma (\cite[Lemma 2.10]{Vass}) will allow us to use the properties of Bridgeland stability.
	
	\begin{lemma}\label{nu-lambda}
		Let $E\in\mathrm{Coh}^{\beta_0}(X)$ be a $\nu_{\alpha_0,\beta_0}$-semistable object such that $\nu_{\alpha_0,\beta_0}(E)=0$ and fix some $s>0$. Then $\lambda_{{\alpha_0},{\beta_0},s}(E[1])=+\infty$ and $E[1]$ is $\lambda_{{\alpha_0},{\beta_0},s}
		$-semistable. If moreover $E$ is $\nu_{\alpha_0,\beta_0}$-stable, then there is a neighbourhood $U$ of $(\beta_0,\alpha_0)$ such that such that for all $(\beta,\alpha)\in U$ with $\nu_{\alpha,\beta}(E)
		>0$ the object $E$ is $\lambda_{\alpha,\beta,s}$-semistable.
	\end{lemma}
	
	Denote by $\mathcal U$ and $\mathcal Q$ the restrictions to $X$ of tautological subbundle and tautological quotient bundle on the Grassmannian $\Gr(2,5)$, respectively, so that we have a short exact sequence
	\begin{equation}\label{tautological}0\to\mathcal U\to \mathcal O_{X}^{\oplus 5}\to\mathcal Q\to 0.\end{equation}
	
	It is known that $\mathcal O_X,\mathcal U$ and $\mathcal Q$ are arithmetically Cohen--Macaulay (ACM) sheaves (see \cite[Proposition 5.2]{Faenzi}, \cite[Proposition 5.7]{Faenzi}), that is, $H^i(\mathcal O_X(j))=H^i(\mathcal U(j))=H^i(\mathcal Q(j))=0$ for $0<i<3$ and any $j\in\mathbb Z$.
	
	In $\mathrm{D}^b(X)$ we have a full strong exceptional collection $(\mathcal O_X(-1),\mathcal Q(-1),\mathcal U,\mathcal O_X)$ (see \cite{Vass} for details). Chern characters of $\mathcal U$ and $\mathcal Q(-1)$ are calculated in \cite[Lemma 2.11]{Vass}: $\mathrm{ch}(\mathcal Q(-1))=3-2H+\frac 25H^2+\frac{1}{15}H^3,\mathrm{ch}(\mathcal U)=2-H+\frac{1}{10}H^2+\frac{1}{30}H^3$.
	
	Since $(\mathcal O_X(-1),\mathcal Q(-1),\mathcal U,\mathcal O_X)$ is a full strong exceptional collection, it follows that the extension closure $\mathfrak C:=\langle\mathcal O_{X}(-1)[3],\mathcal Q(-1)[2],\mathcal U[1],\mathcal O_{X}\rangle$ is the heart of a bounded t-structure on $\mathrm{D}^b(X)$ \cite[Lemma 3.14]{Curves}.
	
	Define the following region in the upper half-plane:
	\begin{equation}\label{D}D:=\{(\beta,\alpha)\in\mathbb H\ |\ \beta< -\frac 12,\alpha<\beta+1\}.\end{equation}
	
	For any	$\gamma\in\mathbb R$ we define a torsion pair
		\begin{equation}\label{T'' F''}
				\begin{split}
						& \TT''_{\gamma}=\{E\in\mathcal A^{\alpha,\beta}(X)\ |\ \text{any quotient}\ E\onto G
						\,\text{satisfies}\,\lambda_{\alpha,\beta, s}(G)>\gamma \}, \\
						& \FF''_{\gamma} = \{E \in \mathcal A^{\alpha, \beta}(X)\ |\ \text{any
								subobject}\ 0\ne F\into E\,\text{satisfies}\, \lambda_{\alpha,\beta, s}(F)\leq\gamma\}.
					\end{split}
			\end{equation}
	
	We have the following proposition (\cite[Corollary 2.15]{Vass}).
	
	\begin{proposition}\label{heart}Let $(\beta,\alpha)\in D$, and let $ \TT''_{\gamma}$, $\FF''_{\gamma}$ be the torsion pair on the category $\mathcal A^{\alpha,\beta}(X)$ defined in (\ref{T'' F''}). There are $s>\frac 16$ and $\gamma\in\R$ such that
		$\langle \mathcal T''_{\gamma},\mathcal F''_{\gamma}[1]\rangle =\mathfrak C.$
	\end{proposition}
	
	\begin{remark}\label{complexes}
		Any object $E\in\mathfrak C$ is isomorphic to a complex of the form
		$$\mathcal O_{X}(-1)^{\oplus a}\to\mathcal Q(-1)^{\oplus b}\to\mathcal U^{\oplus c}\to\mathcal O_{X}^{\oplus d},$$
		concentrated in degrees $-3,-2,-1$ and $0$. The numbers $a,b,c,d$ are uniquely determined by the Chern character of $E$.
	\end{remark}
	
	The following lemma provides a justification of a technique used by Schmidt in \cite{Sch23} (see e.~g. \cite[Lemma 6.9 (i)]{Sch23}). This lemma and the following one can be applied to any smooth projective variety $X$ of Picard rank 1.
	
	\begin{lemma}\label{subobject}Suppose that $E\in\mathrm{Coh}^\beta(X)$ is tilt-semistable with $\mathrm{ch}_0(E)\ge 0$ and for some $a\in\mathbb Z$ we have $\mathrm{hom}(\mathcal O_X(a),E)=n>0,a<\mu(E)$. Suppose also that $W(\mathcal O_X(a),E)$ is nonempty and lies in the half-plane $\beta<a$, then for $\beta'<a$ and close enough to $a$ the evaluation morphism $s_{a,E}:\mathcal O_X(a)^{\oplus n}\to E$ is a monomorphism in the category $\mathrm{Coh}^{\beta'}(X)$.\end{lemma}
	
	\begin{proof}The existence of a nonzero morphism $\mathcal O_X(a)\to E$ implies that $E$ is tilt-unstable below $W(E,\mathcal O_X(a))$. Since the case in which $E$ is tilt-semistable only at the vertical wall is excluded by the assumption $\mathrm{ch}_0(E)\ge 0$, it follows that $E$ is tilt-semistable at $W(E,\mathcal O_X(a))$ or at some wall lying above $W(\mathcal O_X(a),E)$. In particular, $E\in\mathrm{Coh}^{\beta'}(X)$ for $\beta'<a$ and close enough to $a$, since $W(\mathcal O_X(a),E)$ contains points with such $\beta$-coordinates (note that the top point of $W(\mathcal O_X(a),E)$ has to lie on the degenerate hyperbola $\alpha^2-(\beta-a)^2=0$ given by $\nu_{\alpha,\beta}(\mathcal O_X(a))=0$). From now on we assume that $E\in\mathrm{Coh}^{\beta'}(X)$.
		
		We have an exact sequence
		\begin{equation}\label{cohomology}0\to\mathcal H^{-1}(E)[1]\to E\to\mathcal H^0(E)\to 0\end{equation}
		in the category $\mathrm{Coh}^{\beta'}(X)$. Applying the functor $\mathrm{Hom}(\mathcal O_X(a),-)$ and considering evaluation maps, we get a commutative diagram of the form
		$$\xymatrix{
			0 \ar[r] & \mathcal O_X(a)^{\oplus m} \ar[r]\ar[d]^f & \mathcal O_X(a)^{\oplus n}\ar[r]\ar[d]^{s_{a,E}} & \mathcal O_X(a)^{\oplus k}\ar[d]^g & \\
			0\ar[r] & \mathcal H^{-1}(E)[1]\ar[r] & E\ar[r] & \mathcal H^0(E)\ar[r] & 0 }$$
		
		for some $m,k\ge 0$. A simple diagram chase shows that in order to prove that $s_{a,E}$ is a monomorphism, it suffices to prove that $f$ and $g$ are monomorphisms. Therefore, we reduce to cases when $E$ is a sheaf or a sheaf shifted by 1 (here we can drop assumptions that $E$ is tilt-semistable and has $\mathrm{ch}_0(E)\ge 0$).
		
		Suppose that $E$ is a sheaf, then slope stability of $\mathcal O_X(a)$ and the fact that the map $s_{a,E}:\mathcal O_X(a)^{\oplus n}\to E$ corresponds to $n$ linearly independent morphisms $\mathcal O_X(a)\to E$ imply that the kernel $K$ of $s_{a,E}$ in the category $\mathrm{Coh}(X)$ has $\mu_{\max}(K)<a$. The kernel of $s_{a,E}$ in the category $\mathrm{Coh}^{\beta'}(X)$ belongs to $\mathcal T_{\beta'}$ and factors through $K$, hence for $\mu_{\max}(K)\le\beta'<a$ this kernel is zero.
		
		Suppose now that $E\cong F[1]$ for a sheaf $F$, then $F\in\mathcal F_{\beta'}$. The morphism $s_{a,E}$ can be included into an exact triangle $F\to L\to \mathcal O_X(a)^{\oplus n}\overset{s_{a,E}}\to F[1]$. We have $\mu_{\max}(L)\le a$ (since $L$ is an extension of $\mathcal O_X(a)^{\oplus n}$ by $F$). If $G\into L$ is a slope semistable subsheaf of slope $a$, then $G\into \mathcal O_X(a)^{\oplus n}$, so $G\cong\mathcal O_X(a)^{\oplus l}$ for some $l>0$. In this case one can show that $n$ morphisms defining the map $\mathcal O_X(a)^{\oplus n}\to F[1]$ are linearly dependent, so by contradiction $\mu_{\max}(L)<a$. The kernel of $s_{a,E}$ in $\mathrm{Coh}^{\beta'}(X)$ factors through $L$ (and belongs to $\mathcal T_{\beta'}$), hence for all $\mu_{\max}(L)\le\beta'<a$ the kernel is zero.
	\end{proof}
	
	We can give the following generalization of this technique.
	
	\begin{lemma}\label{quotient}Suppose that $E\in\mathrm{Coh}^\beta(X)$ is tilt-semistable with $\mathrm{ch}_0(E)\ge 0$ and for some $a\in\mathbb Z$ we have $\mathrm{hom}(E,\mathcal O_X(a)[1])=n>0,a<\mu(E)$. Suppose also that $W(\mathcal O_X(a)[1],E)$ is nonempty and lies in the half-plane $\beta>a$, then for $\beta'>a$ and close enough to $a$ the coevaluation morphism $s_{E,a}:E\to\mathcal O_X(a)^{\oplus n}[1]$ is an epimorphism in the category $\mathrm{Coh}^{\beta'}(X)$.\end{lemma}
	
	\begin{proof} The argument is completely analogous to the one in Lemma \ref{subobject}.
	\end{proof}
	
	We use the following Proposition for describing semistable objects of negative rank. Recall that $\mathbb D(E):=\mathbf{R}\mathcal{H}om(E,\mathcal O_X)[1]$ is the derived dual of $E$.
	
	\begin{proposition}[{\cite[Proposition 5.1.3]{BMT}}]\label{duality} For any $\nu_{\alpha,\beta}$-semistable object $E\in\mathrm{Coh}^\beta(X)$ with $\nu_{\alpha,\beta}(E)\neq +\infty$ there exists a $\nu_{\alpha,-\beta}$-semistable object $\widetilde{E}\in\mathrm{Coh}^{-\beta}(X)$ and a sheaf $T$ supported in dimension $\le 0$ together with a distinguished triangle
		\begin{equation}\widetilde{E}\to\mathbb D(E)\to T[-1]\to\widetilde{E}[1].\end{equation}
	\end{proposition}
	
	\section{Objects with Maximal Third Chern Character}
	\label{bounds}
	\subsection{Rank one objects.}\label{r=1} Firstly we give a description of ideal sheaves of lines and conics on $X$ in terms of tilt-stability.
	
	\begin{lemma}\label{line}If $E\in\mathrm{Coh}^\beta(X)$ is tilt-semistable and $\mathrm{ch}(E)=1-\frac{H^2}{5}+eH^3$, then $e\le 0$, and in case of equality we have $E\cong\mathcal I_L$ for a line $L\subset X$.
	\end{lemma}
	
	\begin{proof}If $E$ is not $\nu_{\alpha,\beta}$-semistable for some $\beta<0,\alpha>0$, then there is a semicircular wall $W$, which should be induced by a tilt-stable subobject or quotient $F$ with $\mathrm{ch}_{\le 2}(F)=s+xH+yH^2,s>0,\mu(F)<\mu(E)$ (by Lemma \ref{stable}). This wall should intersect the ray $\beta=\beta_-(E)=-\sqrt{2/5}$ (by Proposition \ref{numerical} (2, 3)). Since $0<\mathrm{ch}_1^{\beta_-(E)}(F)<\mathrm{ch}_1^{\beta_-(E)}(E)$ (cf. Proposition \ref{actual}~(3)), we get
		\begin{equation}\label{x1}0<x+\sqrt{2/5}s<\sqrt{2/5}.
		\end{equation}
		We have $-1<\beta_-(E)<\beta_-(F)\le\mu(F)<\mu(E)=0$ (cf. Proposition \ref{actual}~(5)). The formula (\ref{x1}) implies that $\frac xs<\frac{\sqrt{2/5}}{s}-\sqrt{\frac 25}$. Therefore, for $s\ge 2$ we get $\beta_+(F)=2\mu(F)-\beta_-(F)<\frac{\sqrt{8/5}}{s}-\sqrt{\frac 85}+\sqrt{\frac 25}\le 0$. So, in this case $\beta_-(F),\beta_+(F)\in[-1,0)$. For $s\neq 1$ we can apply Theorem \ref{rank bound} to $F$ and get that $s^2\le\overline{\Delta}_H(F)<\overline{\Delta}_H(E)=10$. Hence, $s\in\{1,2,3\}$. A calculation shows, that for $s\in\{1,3\}$ there are no integer $x$ satisfying (\ref{x1}). For $s=2$ we get $x=-1$. But in this case $\overline{\Delta}_H(F)=25-100y,y\in\frac 12+\frac 15\mathbb Z$ (since Chern classes are integral), and there are no such $y$, satisfying $0\le\overline{\Delta}_H(F)<\overline{\Delta}_H(E)$. It follows that $E$ is tilt-semistable for all $\beta<0,\alpha>0$, hence $E$ is a 2-semistable sheaf (by Proposition \ref{2-stability}).	
		
		In particular, $E$ is a torsion-free sheaf of rank one, hence there is a short exact sequence
		$$0\to E\to E^{\vee\vee}\to T\to 0,$$
		in which $E^{\vee\vee}$ is a line bundle and $\mathrm{dim}\ T\le 1$. Since $c_1(E)=0$ and $\mathrm{Pic}\ X\cong\mathbb Z$, we have $E^{\vee\vee}\cong\mathcal O_X$ and $T\cong\mathcal O_Z$ for a subscheme $Z\subset X$. Since $\mathrm{ch}_2(\mathcal O_Z)=\frac{H^2}{5}$, $Z$ is a line $L$ together with a finite number of (possibly embedded) points. The number $e$ is maximal when there are no points and $E\cong\mathcal I_L$. Grothendieck-Riemann-Roch shows that in this case $e=0$.
	\end{proof}
	
	\begin{lemma}\label{conic}If $E\in\mathrm{Coh}^\beta(X)$ is tilt-semistable and $\mathrm{ch}(E)=1-\frac 25H^2+eH^3$, then $e\le\frac 15$, and in case of equality we have that either $E\cong\mathcal I_C$ for a conic $C\subset X$, or $E\cong\mathcal U\oplus(\mathcal O_X(-1)[1]).$
	\end{lemma}
	
	\begin{proof}We can assume that $e=\frac 15$, since the bound $e\le\frac 15$ is already proven in \cite[Lemma 4.2]{Vass}. Note also that $H^2\cdot\mathrm{ch}_1^{-1}(E)$ has the minimal value, hence there are no walls at $\beta=-1$ by Proposition \ref{actual}~(3).
		
		A calculation shows that $W(E,\mathcal O_X(1))$ lies on the wrong side of the vertical wall and $\nu_{\alpha,\beta}(\mathcal O_X(1))>\nu_{\alpha,\beta}(E)$ for all $\beta<0,\alpha>0$. Hence, there are no nonzero morphisms $\mathcal O_X(1)\to E$, and by Serre duality $\mathrm{Ext}^3(E,\mathcal O_X(-1))\cong\Hom(\mathcal O_X(1),E)^\vee=0$. Riemann--Roch gives $\mathrm{hom}(E,\mathcal O_X(-1)[1])\ge -\chi(E,\mathcal O_X(-1))=\chi(\mathcal O_X(1),E)=1$. A morphism $E\to\mathcal O_X(-1)[1]$ destabilizes $E$ below $W(E,\mathcal O_X(-1)[1])$, hence $E$ is unstable below this wall. The wall $W(E,\mathcal O_X(-1)[1])$ is given by $\alpha^2+(\beta-\frac{9}{10})^2=(\frac{1}{10})^2$, and any wall lying above $W(E,\mathcal O_X(-1)[1])$ would intersect the ray $\beta=-1$. Hence, $E$ is tilt-semistable at $W(E,\mathcal O_X(-1)[1])$. By Lemma \ref{quotient} we have an exact sequence
		\begin{equation}\label{G E}0\to G\to E\to\mathcal O_X(-1)[1]\to 0,\end{equation}
		where $G$ is tilt-semistable with $\mathrm{ch}(G)=2-H+\frac{H^2}{10}+\frac{H^3}{30}$. By \cite[Lemma 3.6]{Vass} $G\cong\mathcal U$. If $E$ is a sheaf, then we get an exact triple
		\begin{equation}\label{U E}0\to\mathcal O_X(-1)\to\mathcal U\to E\to 0.
		\end{equation}
		Serre correspondence (\cite[Theorem 1.1]{SVB}) implies that $E$ is an ideal sheaf of a locally complete intersection curve $C\subset X$. We find $\mathrm{ch}(\mathcal O_C)=\frac 25H^2-\frac{H^3}{5}$, so $C$ has degree 2. Since $\mathrm{ch}_3(\mathcal O_C)\neq 0$, Lemma \ref{line} implies that $C$ cannot be the union of two skew lines, so $C$ is a conic. The exact triple (\ref{U E}) was already studied in \cite[Example 3.2]{AC}.
		
		Finally, if $E$ is not a sheaf, then the long exact sequence of cohomology sheaves associated with (\ref{G E}) (with $G\cong\mathcal U$) has the form
		$$0\to\mathcal H^{-1}(E)\to\mathcal O_X(-1)\overset{f}{\to}\mathcal U\to\mathcal H^0(E)\to 0.$$
		Since $\mathcal U$ is a torsion free sheaf and $\mathcal H^{-1}(E)\neq 0$, we should have $f=0$. So, $E$ as a cone of the zero morphism is isomorphic to $\mathcal U\oplus(\mathcal O_X(-1)[1])$.
	\end{proof}
	
	\subsection{Rank zero objects.}\label{rank 0} In this section we deal with objects supported on hyperplane sections $H\subset X$.
	
	\begin{lemma}\label{0 1 -1/2} If $E\in\mathrm{Coh}^\beta(X)$ is tilt-semistable and $\mathrm{ch}(E)=H-\frac{H^2}{2}+eH^3$, then $e\le\frac 16$. If $e=\frac 16$ and $E$ is tilt-semistable above $W(E,\mathcal O_X)$, then $E\cong\mathcal O_H$.
	\end{lemma}
	
	\begin{proof}Note that $H^2\cdot\mathrm{ch}_1(E)$ has the least positive value, hence there is no wall intersecting the ray $\beta=0$. Therefore $Q_{0,0}(E)=1-6e\ge 0$, and we get the bound $e\le\frac 16$. Assume that $e=\frac 16$. The numerical wall $W(E,\mathcal O_X(-2)[1])$ intersects the ray $\beta=0$ and a nonzero morphism $E\to\mathcal O_X(-2)[1]$ would destabilize $E$ below $W(E,\mathcal O_X(-2)[1])$. It follows that $\mathrm{Ext}^2(\mathcal O_X,E)\cong\Hom(E,\mathcal O_X(-2)[1])^\vee=0$. Grothendieck--Riemann-Roch Theorem (applied to the map from $X$ to a point) gives $\dim \Hom(\mathcal O_X,E)\ge\chi(E)=1$. By Lemma \ref{subobject} we obtain a monomorphism $F=\mathcal O_X\hookrightarrow E$. Since any wall lying above $W(E,\mathcal O_X)$ would intersect the ray $\beta=0$, and $\mathcal O_X\hookrightarrow E$ destabilizes $E$ below $W(E,\mathcal O_X)$, it follows that $E$ is tilt-semistable (at least) at $W(E,\mathcal O_X)$; hence $E/F$ is tilt-semistable. We have $\mathrm{ch}(E/F)=-1+H-\frac{H^2}{2}+\frac{H^3}{6}$, hence by \cite[Corollary 3.3]{Vass} $E/F\cong\mathcal O_X(-1)[1]$. If $E$ is tilt-semistable above $W(E,\mathcal O_X)$, then the morphism $f:\mathcal O_X(-1)\to\mathcal O_X$ corresponding to this extension is nonzero, and $E$, which is a cone of $f$, is isomorphic to $\mathcal O_H$.
	\end{proof}
	
	\begin{lemma}\label{0 1 -7/10} If $E\in\mathrm{Coh}^\beta(X)$ is tilt-semistable and $\mathrm{ch}(E)=H-\frac{7}{10}H^2+eH^3$, then $e\le\frac 16$. If $e=\frac 16$ and if moreover $E$ is tilt-semistable above $W(E,\mathcal O_X(-1)[1])$, then $E\cong\mathcal I_{L,H}$ is an ideal sheaf of a line $L$ on a hyperplane section $H\subset X$, included into an exact triple
		\begin{equation}\label{triple4.1}0\to\mathcal O_X(-1)\oplus\mathcal U\to\mathcal Q^\vee\to \mathcal I_{L,H}\to 0.\end{equation} 
	\end{lemma}
	
	\begin{proof}The curve $\nu_{\alpha,\beta}(E)=0$ is the ray $\beta=-\frac{7}{10}$, hence the highest point of any wall for $E$ lies on this ray (by Proposition \ref{numerical}~(3)). Note also that $H^2\cdot\mathrm{ch}_1^{-1}(E)$	has the least positive value, hence there is no wall intersecting the ray $\beta=-1$. We have $H^0(E)=0$ (since $W(\mathcal O_X,E)$ intersects the ray $\beta=-1$, while a nonzero morphism $\mathcal O_X\to E$ would destabilize $E$ below $W(\mathcal O_X,E)$) and $H^2(E)\cong\mathrm{Hom}(E,\mathcal O_X(-2)[1])^\vee=0$ (since $W(\mathcal O_X(-2)[1],E)$ intersects the ray $\beta=-1$). Riemann--Roch Theorem gives $\chi(E)=5e-\frac 56\le 0$, hence $e\le\frac 16$. 
		
	Assume that $e=\frac 16$. If there is a semicircular wall $W$ for $E$, induced by a subobject or quotient $F$ with $\mathrm{ch}_{\le 2}(F)=s+xH+yH^2$, then, without loss of generality, we can assume that $s>0$. Suppose that $s\ge 2$. By Proposition \ref{actual}~(1) in this case $\rho(E,F)^2\le\frac{1}{16}$, hence $E$ is tilt-semistable at some point with $\beta=-\frac{7}{10},\alpha\le\frac 14$. Lemma \ref{nu-lambda} with Proposition \ref{heart} imply that in this case $E[1]\in\mathfrak C$. However, a direct calculation shows that equation $\mathrm{ch}(E[1])=a\cdot\mathrm{ch}(\mathcal O_X(-1)[3])+b\cdot\mathrm{ch}(\mathcal Q(-1)[2])+c\cdot\mathrm{ch}(\mathcal U[1])+d\cdot\mathrm{ch}(\mathcal O_X)$ has no solution with non-negative $a,b,c,d$, hence this case is impossible.
		
	Therefore, $E$ is destabilized by a subobject or quotient $F$ with $s=1$. The condition $0<\mathrm{ch}_1^{-\frac{7}{10}}(F)<\mathrm{ch}_1^{-\frac{7}{10}}(E)$ can be rewritten as $0<x+\frac{7}{10}<1$, so $x=0$. We have $y\in\frac 15\mathbb Z$, and the conditions $\overline{\Delta}_H(F)\ge 0,\overline{\Delta}_H(E/F)\ge 0,\overline{\Delta}_H(F)+\overline{\Delta}_H(E/F)<\overline{\Delta}_H(E)=25$ imply that $y=-\frac 15$. Lemma \ref{line} implies that $\mathrm{ch}_3(F)$ is maximized by $F\cong\mathcal I_L$. Then we have $\mathrm{ch}(E/F)=-1+H-\frac{H^2}{2}+\frac{H^3}{6}$, and \cite[Corollary 3.3]{Vass} gives $E/F\cong\mathcal O_X(-1)[1]$. If moreover $E$ is tilt-semistable above $W(E,\mathcal O_X(-1)[1])$, then $\mathcal I_L$ is indeed a subobject and $\mathcal O_X(-1)[1]$ is indeed a quotient. That is, we have an exact triple 
	\begin{equation}\label{triple4.11}0\to\mathcal I_L\to E\to\mathcal O_X(-1)[1]\to 0.\end{equation} Since $E$ is tilt-semistable for $\alpha\gg 0$, Proposition \ref{2-stability} implies that $E$ is a sheaf. Rotating the exact triangle corresponding to (\ref{triple4.11}), we obtain an exact sequence of coherent sheaves
	\begin{equation}\label{triple4.2}0\to\mathcal O_X(-1)\overset{i}{\to}\mathcal I_L\overset{j}{\to} E\to 0.\end{equation}
	Using an exact triple $0\to\mathcal I_L\to\mathcal O_X\to\mathcal O_L\to 0$, the monomorphism $i$ from (\ref{triple4.2}) and the induced monomorphism $i':\mathcal O_X(-1)\to\mathcal O_X$, we obtain (e.~g. by the Snake Lemma) an exact triple $0\to E\to\mathcal O_H\to\mathcal O_L\to 0$. Hence, $E\cong\mathcal I_{L,H}$.
	
	By \cite[Lemma 4.2]{K} we have an exact triple of coherent sheaves
	\begin{equation}\label{I_L}0\to\mathcal U\to\mathcal Q^\vee\overset{f}{\to}\mathcal I_L\to 0.\end{equation} 
	Together with (\ref{triple4.2}), it gives the following commutative diagram with exact rows:
	\begin{equation*}\xymatrix{
		0\ar[r] & 0\ar[r]\ar[d] & \mathcal Q^\vee\ar[r]^\cong\ar[d]^f & \mathcal Q^\vee\ar[r]\ar[d]^{j\circ f}\ar[r]\ar[d] & 0 \\
		0 \ar[r] & \mathcal O_X(-1) \ar[r] & \mathcal I_L\ar[r]^j & \mathcal I_{L,H}\ar[r] & 0.}\end{equation*}
	An application of the Snake Lemma and (\ref{I_L}) yields that $\ker(j\circ f)\cong\mathcal O_X(-1)\oplus\mathcal U$, so we obtain the exact triple (\ref{triple4.1}).
	\end{proof}
	
	\begin{lemma}\label{0 1 -9/10} If $E\in\mathrm{Coh}^\beta(X)$ is tilt-semistable and $\mathrm{ch}(E)=H-\frac{9}{10}H^2+eH^3$, then $e\le\frac{11}{30}$. If $e=\frac{11}{30}$ and if moreover $E$ is tilt-semistable above $W(E,\mathcal O_X(-1)[1])$, then $E\cong\mathcal I_{C,H}$ is an ideal sheaf of a conic $C$ on a hyperplane section $H\subset X$, included into an exact sequence
	\begin{equation}\label{triple15}0\to\mathcal O_X(-1)^{\oplus 2}\to\mathcal U\to\mathcal I_{C,H}\to 0.\end{equation} 
	\end{lemma}
	
	\begin{proof} Note that $H^2\cdot\mathrm{ch}_1^{-1}(E)$	has the least positive value, hence there is no wall intersecting the ray $\beta=-1$. We have $H^0(E)=0$ (since $W(\mathcal O_X,E)$ intersects this ray) and $H^2(E)\cong\mathrm{Hom}(E,\mathcal O_X(-2)[1])^\vee=0$ (since $W(\mathcal O_X(-2)[1],E)$ intersects the ray $\beta=-1$). Riemann--Roch Theorem gives $\chi(E)=5e-\frac{11}{6}\le 0$, hence $e\le\frac{11}{30}$. 
		
	Assume that $e=\frac{11}{30}$. We have $\mathrm{Ext}^3(E,\mathcal O_X(-1))\cong\mathrm{Hom}(\mathcal O_X(-1),E(-2))^\vee\cong\mathrm{Hom}(\mathcal O_X(1),E)^\vee=0$, since $W(\mathcal O_X(1),E)$ intersects the ray $\beta=-1$. Therefore $\mathrm{hom}(E,\mathcal O_X(-1)[1])=\mathrm{ext}^1(E,\mathcal O_X(-1))\ge-\chi(E,\mathcal O_X(-1))=\chi(\mathcal O_X(1),E)=2$ (by Riemann--Roch and Serre duality). Since there exists a nonzero morphism $E\to\mathcal O_X(-1)[1]$, the object $E$ is unstable below $W(E,\mathcal O_X(-1)[1])$, while above $W(E,\mathcal O_X(-1)[1])$ there are no walls, because such walls would intersect the ray $\beta=-1$. Therefore, $E$ is tilt-semistable at $W(E,\mathcal O_X(-1)[1])$. In view of Lemma \ref{quotient} for $-1<\beta\ll 0$ we have an exact sequence
	\begin{equation}\label{triple7}0\to F\to E\to\mathcal O_X(-1)^{\oplus 2}[1]\to 0\end{equation}
	in the category $\mathrm{Coh}^\beta(X)$ with a tilt-semistable object $F$. Here $\mathrm{ch}(F)=2-H+\frac{H^2}{10}+\frac{H^3}{30}$, hence by \cite[Lemma 3.6]{Vass} $F\cong\mathcal U$.
		
	Assume that $E$ is tilt-semistable above $W(E,\mathcal O_X(-1)[1])$. Since this is the unique wall for $E$, tilt-semistability of $E$ above this wall implies that $E$ is a sheaf (by Lemma \ref{minimal}). Also, exact triple (\ref{triple7}) implies that there is an epimorphism $E\to\mathcal O_X(-1)[1]$. Denote the kernel of this epimorphism by $F'$, then $F'$ is a tilt-semistable object with $\mathrm{ch}(F')=1-\frac 25H^2+\frac 15H^3$. Lemma \ref{conic} implies that either $F'\cong\mathcal I_C$ for a conic $C\subset X$, or $F'\cong\mathcal U\oplus(\mathcal O_X(-1)[1]).$ However, in the second case $F'$ cannot be a subobject of a sheaf $E$, hence $F'\cong\mathcal I_C$, and $E$ is the cokernel of a monomorphism $\mathcal O_X(-1)\to\mathcal I_C$. An argument analogous to the one given in Lemma \ref{0 1 -7/10} implies that $E\cong\mathcal I_{C,H}$. Finally, rotating the exact triangle corresponding to (\ref{triple7}) (where $F\cong\mathcal U$), we obtain (\ref{triple15}).
	\end{proof}
	
	\begin{lemma}\label{0 1 -1/10} If $E\in\mathrm{Coh}^\beta(X)$ is tilt-semistable and $\mathrm{ch}(E)=H-\frac{H^2}{10}+eH^3$, then $e\le-\frac{1}{30}$. If $e=-\frac{1}{30}$ and if moreover $E$ is tilt-semistable above $W(E,\mathcal O_X)$, then $E$ is a coherent sheaf $\mathbb D(\mathcal I_{C,H})(-1)$, included into an exact sequence
		\begin{equation}\label{triple24}0\to\mathcal U\to\mathcal O_X^{\oplus 2}\to \mathbb D(\mathcal I_{C,H})(-1)\to 0.
		\end{equation}
	\end{lemma}
	
	\begin{proof}Proposition \ref{duality} implies that there is a distinguished triangle $\widetilde{E}\to\mathbb D(E)(-1)\to T[-1]\to\widetilde{E}[1]$ with a tilt-semistable object $\widetilde{E}$ and a dimension zero sheaf $T$ of length $t\ge 0$. We calculate $\mathrm{ch}(\widetilde{E})=H-\frac{9}{10}H^2+(e+t+\frac 25)H^3$. Lemma \ref{0 1 -9/10} implies that $e\le-\frac{1}{30}$, and in case of equality $t=0$ (so $\mathbb D(E)(-1)\cong\widetilde{E}$ is tilt-semistable). If moreover $E$ is tilt-semistable above $W(E,\mathcal O_X)$, then $\mathbb D(E)(-1)$ is tilt-semistable above $W(\mathbb D(E)(-1),\mathcal O_X(-1)[1])$, so $\mathbb D(E)(-1)\cong\mathcal I_{C,H}$ and $E\cong\mathbb D(\mathcal I_{C,H})(-1)$. In this case $E$ is tilt-semistable for $\alpha\gg 0$ (there are no walls above $W(E,\mathcal O_X)$, because any such wall would intersect the ray $\beta=0$ and $H^2\cdot\mathrm{ch}_1(E)$ has the minimal positive value). Hence, $E$ is a sheaf by Lemma \ref{minimal}. 
	Twisting (\ref{triple15}) by $\mathcal O_X(1)$ and applying the functor $\mathbb D(-)$, we obtain an exact triangle
	\begin{equation}\label{triple25.2}\mathbb D(\mathcal I_{C,H})(-1)\to\mathcal U[1]\to\mathcal O_X^{\oplus 2}[1]\to\mathbb D(\mathcal I_{C,H})(-1)[1].
	\end{equation}
	Rotating this exact triangle twice and using the fact that $\mathbb D(\mathcal I_{C,H})(-1)$ is a sheaf, we obtain (\ref{triple24}). 
	\end{proof}
	
	\begin{lemma}\label{0 1 -3/10} If $E\in\mathrm{Coh}^\beta(X)$ is tilt-semistable and $\mathrm{ch}(E)=H-\frac{3}{10}H^2+eH^3$, then $e\le-\frac{1}{30}$. If $e=-\frac{1}{30}$ and $E$ is tilt-semistable above $W(E,\mathcal O_X)$, then $E$ is a coherent sheaf $\mathbb D(\mathcal I_{L,H})(-1)$, included into an exact triple
	\begin{equation}\label{triple25}0\to\mathcal Q(-1)\to\mathcal O_X\oplus\mathcal U\to \mathbb D(\mathcal I_{L,H})(-1)\to 0.
	\end{equation}
	\end{lemma}
	
	\begin{proof}Again we use a distinguished triangle $\widetilde{E}\to\mathbb D(E)(-1)\to T[-1]\to\widetilde{E}[1]$ with a tilt-semistable $\widetilde{E}$ and a sheaf $T$ of length $t\ge 0$. We find $\mathrm{ch}(\widetilde{E})=H-\frac{7}{10}H^2+(e+t+\frac 15)H^3$. Lemma \ref{0 1 -7/10} implies that $e\le-\frac{1}{30}$, and in case of equality $t=0$. Also in case of equality, if $E$ is tilt-semistable above $W(E,\mathcal O_X)$, then $\mathbb D(E)(-1)$ is tilt-semistable above $W(\mathbb D(E)(-1),\mathcal O_X(-1)[1])$, so $\mathbb D(E)(-1)\cong\mathcal I_{L,H}$ and $E\cong\mathbb D(\mathcal I_{L,H})(-1)$. In this case $E$ is a sheaf by Lemma \ref{minimal}.
    Twisting (\ref{triple4.1}) by $\mathcal O_X(1)$ and applying the functor $\mathbb D(-)$, we obtain an exact triangle
	\begin{equation}\label{triple25.2}\mathbb D(\mathcal I_{L,H})(-1)\to\mathcal Q(-1)[1]\to\mathcal O_X[1]\oplus\mathcal U[1]\to\mathbb D(\mathcal I_{L,H})(-1)[1].
	\end{equation}
	Rotating this exact triangle twice, we obtain (\ref{triple25}).
	\end{proof}
	
	Now we get the following strengthening of \cite[Lemma 4.1]{Vass}. Here $\varepsilon(d)$ is the error term defined by (\ref{error}).
	
	\begin{corollary}\label{0 1}If $E\in\mathrm{Coh}^\beta(X)$ is tilt-semistable and $\mathrm{ch}(E)=H+dH^2+eH^3$, then $e\le\frac{d^2}{2}+\frac{1}{24}-\varepsilon(d)$. In case of equality $E$ is a twisted object with maximal $\mathrm{ch}_3$ from Lemmas \ref{0 1 -1/2}--\ref{0 1 -3/10}.
	\end{corollary}
	
	\begin{proof} Since tilt-stability is invariant with respect to a twist by a line bundle, the statement follows from a straightforward calculation of Chern characters of twisted objects $E(n),n\in\mathbb Z$, where $E$ is an object with maximal $\mathrm{ch}_3$ from Lemmas given above.
	\end{proof}
	
	\subsection{Rank two objects.}\label{rank 2}
	Now we can formulate our main result about rank 2 tilt-semistable objects. 
	
	\begin{theorem}\label{main}Let $E\in\mathrm{Coh}^\beta(X)$ be a tilt-semistable object with $\mathrm{ch}(E)=2+cH+dH^2+eH^3$.\\
		(1) If $c=-1$, then $d\le \frac{1}{10}$. \\
		(1.1) If $d=\frac{1}{10}$, then $e\le \frac{1}{30}$. In case of equality we have $E\cong\mathcal U$.\\
		(1.2) If $d=-\frac{1}{10}$, then $e\le\frac{7}{30}$. In case of equality $E$ is included into an exact triple
		$0\to\mathcal O_X(-1)^{\oplus 4}\to E\to \mathcal U(-1)[1]\to 0.$\\
		(1.3) If $d=-\frac{3}{10}$, then $e\le\frac{13}{30}$.
		If $e=\frac{13}{30}$ and $E$ is tilt-stable, then $E$ is a coherent sheaf, included into an exact sequence
		$0\to \mathcal O_X(-2)\to\mathcal O_X(-1)^{\oplus 3}\to E\to\mathcal O_L(-2)\to 0.$\\
		(1.4) If $d=-\frac 12$, then $e\le\frac 56$, and in case of equality $E$ is destabilized by an exact triple
		$0\to \mathcal O_X(-1)^{\oplus 3}\to E\to\mathcal O_X(-2)[1]\to 0.$\\
		(1.5) If $d\le-\frac{7}{10}$, then $e\le\frac{d^2}{2}-d+\frac{5}{24}-\varepsilon(d).$ In case of equality $E$ is destabilized by one of the following exact triples:\\
		(1.5.1) $0\to \mathcal O_X(-1)^{\oplus 2}\to E\to\mathcal I_{C,H}(d-\frac{1}{10})\to 0,$ if $d\in\frac{1}{10}+\mathbb Z$;\\
		(1.5.2) $0\to \mathcal O_X(-1)^{\oplus 2}\to E\to\mathcal I_{L,H}(d-\frac{3}{10})\to 0,$ if $d\in\frac{3}{10}+\mathbb Z$;\\
		(1.5.3) $0\to \mathcal O_X(-1)^{\oplus 2}\to E\to\mathcal O_H(d-\frac 12)\to 0,$ if $d\in\frac 12+\mathbb Z$;\\
		(1.5.4) $0\to \mathcal O_X(-1)^{\oplus 2}\to E\to \mathbb D(\mathcal I_{L,H})(d-\frac{17}{10})\to 0,$ if $d\in\frac{7}{10}+\mathbb Z$;\\
		(1.5.5) $0\to \mathcal O_X(-1)^{\oplus 2}\to E\to \mathbb D(\mathcal I_{C,H})(d-\frac{19}{10})\to 0,$ if $d\in\frac{9}{10}+\mathbb Z$.\\
		(2) If $c=0$, then $d\le 0$.\\
		(2.1) If $d=0$, then $e\le 0$. In case of equality $E\cong\mathcal O_X^{\oplus 2}$.\\
		(2.2) If $d=-\frac 15$, then $e\le-\frac{1}{5}$. In case of equality $E$ is a coherent sheaf, included into an exact triple $0\to E\to\mathcal O_X^{\oplus 2}\to\mathcal O_L(1)\to 0.$ \\
		(2.3) If $d=-\frac 25$, then $e\le 0$. In case of equality $E$ is included into an exact triple $0\to\mathcal U^{\oplus 4}\to E \to \mathcal Q(-1)^{\oplus 2}[1]\to 0.$ \\
		(2.4) If $d=-\frac 35$, then $e\le 0$. In case of equality $E$ is the cohomology sheaf of a monad $\mathcal Q(-1)^{\oplus 3}\to\mathcal U^{\oplus 6}\to\mathcal O_X.$ \\
		(2.5) If $d=-\frac 45$, then $e\le\frac 25$. In case of equality $E$ is destabilized by an exact triple $0\to\mathcal U^{\oplus 2}\to E\to \mathcal O_X(-1)^{\oplus 2}[1]\to 0.$ \\
		(2.6) If $d=-\frac 65$, then $e\le\frac 45$. In case of equality $E$ is destabilized by one of the following exact triples:
		$0\to\mathcal U\to E\to \mathbb D(\mathcal I_{L,H})(-2)\to 0, 0\to \mathbb D(\mathcal I_{L,H})(-2)\to E\to\mathcal U\to 0, 0\to\mathcal O_X(-1)^{\oplus 6}\to E\to\mathcal U(-1)^{\oplus 2}[1]\to 0$. \\
		(2.7) If $d=-\frac 85$, then $e\le\frac 75$. In case of equality $E$ is destabilized by one of the following exact triples: $0\to F\to E\to \mathcal O_H(-1)\to 0, 0\to \mathcal O_H(-1)\to E\to F\to 0, 0\to\mathcal U\to E\to \mathcal I_{L,H}(-1)\to 0, 0\to\mathcal I_{L,H}(-1)\to E\to\mathcal U\to 0, 0\to\mathcal Q(-1)\to E\to \mathcal O_X(-2)[1]\to 0,$ where $F$ is the cokernel of a monomorphism of sheaves $\mathcal U(-1)\hookrightarrow\mathcal O_X(-1)^{\oplus 4}$ (from item (1.2) of this theorem). \\
		(2.8) If $d\in\{-1,-\frac 75\}$ or $d\le-\frac 95$, then $e\le\frac{d^2}{2}-\frac{d}{10}+\frac{2}{25}-\varepsilon(d)$. In case of equality $E$ is destabilized by one of the following exact triples:\\
		(2.8.1) $0\to\mathcal U\to E\to \mathbb D(\mathcal I_{C,H})(d-1)\to 0$, if $d\in\mathbb Z$; \\
		(2.8.2) $0\to\mathcal U\to E\to \mathcal I_{C,H}(d+\frac 45)\to 0$, if $d\in\frac 15+\mathbb Z$; \\
		(2.8.3) $0\to\mathcal U\to E\to \mathcal I_{L,H}(d+\frac 35)\to 0$, if $d\in\frac 25+\mathbb Z$; \\
		(2.8.4) $0\to\mathcal U\to E\to \mathcal O_H(d+\frac 25)\to 0$, if $d\in\frac 35+\mathbb Z$; \\
		(2.8.5) $0\to\mathcal U\to E\to \mathbb D(\mathcal I_{L,H})(d-\frac 45)\to 0$, if $d\in\frac 45+\mathbb Z$.
	\end{theorem}
	
	We continue to consider special cases as we started in \cite{Vass}.
	
	\begin{lemma}\label{2 -1 -3/10} If $E\in\mathrm{Coh}^\beta(X)$ is tilt-semistable and $\mathrm{ch}(E)=2-H-\frac{3}{10}H^2+eH^3$, then $e\le\frac{13}{30}$. If $e=\frac{13}{30}$ and $E$ is tilt-stable, then $E$ is a coherent sheaf, included into an exact sequence
		\begin{equation}\label{sequence}0\to\mathcal O_X(-2)\to\mathcal O_X(-1)^{\oplus 3}\to E\to\mathcal O_L(-2)\to 0.\end{equation}
	\end{lemma}
	
	\begin{proof}Note that $H^2\cdot\mathrm{ch}_1^{-1}(E)$ has the least positive value, hence there are no walls at $\beta=-1$ (Proposition \ref{actual}~(3)). The generalized Bogomolov-Gieseker inequality (Proposition \ref{BMT}) implies that $Q_{0,-1}(E)=79-150e\ge 0$, so $e\le\frac{79}{150}$. Since Chern classes are integral, we get $e\in\frac{7}{30}+\frac{1}{10}\mathbb Z$, hence $e\le\frac{13}{30}$. Assume that $e=\frac{13}{30}$. 
		
	One can see that the numerical wall $W(E,\mathcal O_X(-3)[1])$ intersects the ray $\beta=-1$. Therefore, $\mathrm{Ext}^2(\mathcal O_X(-1),E)\cong\Hom(E,\mathcal O_X(-3)[1])=0$. Grothendieck-Riemann-Roch Theorem gives that $\mathrm{hom}(\mathcal O_X(-1),E)\ge\chi(\mathcal O_X(-1),E)=3$. Therefore, we get an exact triple
	\begin{equation}\label{G}0\to\mathcal O_X(-1)^{\oplus 3}\to E\to G\to 0\end{equation}
	with a tilt-semistable object $G$. The subobject $\mathcal O_X(-1)^{\oplus 3}\hookrightarrow E$ destabilizes $E$ below $W(E,\mathcal O_X(-1))$, and any wall lying above $W(E,\mathcal O_X(-1))$ would intersect the ray $\beta=-1$. Therefore, if $E$ is tilt-stable, then it should be tilt-stable for $\beta=-1$ and any $\alpha>0$. In particular, $E$ is a 2-stable sheaf.
		
	We calculate $\mathrm{ch}(G)=-1+2H-\frac 95H^2+\frac{14}{15}H^3$, therefore $\mathrm{ch}(\mathbb D(G)(-2))=1-\frac{H^2}{5}$. By Proposition \ref{duality} we have a distinguished triangle
	$$\widetilde{G}\to \mathbb D(G)(-2)\to T[-1]\to \widetilde{G}[1]$$
	with a tilt-semistable object $\widetilde{G}$ and a sheaf $T$ with $\dim T\le 0$. Therefore, $\frac{\mathrm{ch}_3(\widetilde G)}{H^3}=\frac{\mathrm{ch}_3(T)}{H^3}.$ We have $\frac{\mathrm{ch}_3(T)}{H^3}\ge 0$ and, by Lemma \ref{line}, $\frac{\mathrm{ch}_3(\widetilde G)}{H^3}\le 0$, hence $T=0$ and $\mathbb D(G)(-2)\cong\mathcal I_L$. 
	In other words, $G\cong\mathbb D(\mathcal I_L(2))$. 
    Then $\mathcal H^{-1}(G)\cong\mathcal Hom(\mathcal I_L(2),\mathcal O_X),\linebreak\mathcal H^0(G)\cong\mathcal Ext^1(\mathcal I_L(2),\mathcal O_X)$. We have $\mathcal Hom(\mathcal I_L(2),\mathcal O_X)\cong\mathcal O_X(-2)$ and \linebreak$\mathcal Ext^1(\mathcal I_L(2),\mathcal O_X)\cong\mathcal Ext^2(\mathcal O_L(2),\mathcal O_X)\cong(\Lambda^2(\mathcal N_{L/X}))(-2)$. Note that $\mathcal N_{L/X}$ is a rank 2 vector bundle on $L$, hence $\Lambda^2(\mathcal N_{L/X})$ is a line bundle on $L$. Since $\mathrm{ch}(\mathcal H^0(G))=\mathrm{ch}(G)+\mathrm{ch}(\mathcal H^{-1}(G))=\frac{H^2}{5}-\frac 25H^3$ and $\mathrm{ch}(\mathcal O_L)=\frac{H^2}{5}$, we get \linebreak$(\Lambda^2(\mathcal N_{L/X}))(-2)\cong\mathcal O_L(-2)$. The long exact sequence of cohomology sheaves associated with the exact triple (\ref{G}) gives the exact sequence (\ref{sequence}).
	\end{proof}
	
	For the next step let us give two auxiliary lemmas.
	
	\begin{lemma}\label{3 -2 3/5}There are no tilt-semistable objects $E\in\mathrm{Coh}^\beta(X)$ with $\mathrm{ch}_{\le 2}(E)=3-2H+\frac 35H^2$.
	\end{lemma}
	
	\begin{proof}We have proved in \cite[Lemma 3.5]{Vass} that such an object $E$ should be a 2-stable sheaf. If $E$ is reflexive, then $\mathrm{ch}_{\le 2}(E^\vee)=3+2H+\frac 35H^2$ (by \cite[Lemma 2.17]{Vass}) and $E^\vee$ is slope stable (it follows from the facts that the dual of slope semistable sheaf is always slope semistable and that the rank and the first Chern class of $E^\vee$ are coprime). However, in this case $\mathrm{ch}_{\le 2}(E^\vee(-1))=3-H+\frac{H^2}{10}$, and in \cite[Corollary 3.11]{Vass} we proved that there are no tilt-semistable objects with such a Chern character. Therefore, $E$ is not reflexive and we can consider an exact sequence $0\to E\to E^{\vee\vee}\to T\to 0,\dim T\le 1$. If $\dim T=0$, then $E^{\vee\vee}$ is tilt-semistable with $\mathrm{ch}_{\le 2}(E^{\vee\vee})=3-H+\frac{H^2}{10}$, and we again get a contradiction. And if $\dim T=1$, then $\frac{\mathrm{ch}_2(T)}{H^2}\ge\frac 15$ and we get a contradiction with Bogomolov--Gieseker inequality for $E^{\vee\vee}$.
	\end{proof}
	
	\begin{lemma}\label{4 -2 2/5} There are no tilt-semistable objects $F\in\mathrm{Coh}^\beta(X)$ with $\mathrm{ch}_{\le 2}(F)=4-2H+\frac 25H^2$.
	\end{lemma}
	
	\begin{proof}Suppose for the contrary that there is such an object with $\mathrm{ch}_3(F)=eH^3$. We have $\beta_-(F)=\frac{-1-\sqrt{1/5}}{2}\in(-1,-\frac 12)$. Suppose that there is a semicircular wall for $F$, induced by a tilt-stable subobject of quotient $G$ with $\mathrm{ch}_{\le 2}(G)=s+xH+yH^2,\frac xs<-\frac 12,s>0$. We should have $0<\mathrm{ch}_1^{\beta_-(F)}(G)<\mathrm{ch}_1^{\beta_-(F)}(F)$, that is,
		\begin{equation}\label{x3}0<x+\frac{1+\sqrt{1/5}}{2}s<\sqrt{\frac 45}.
		\end{equation}
		
		We have $-1<\beta_-(F)<\beta_-(G)\le\mu(G)<\mu(F)=-\frac 12<0$. Also $\beta_+(G)=2\mu(G)-\beta_-(G)<0$. If $s\neq 1$, we can apply Theorem \ref{rank bound} to $G$ and get that $s^2\le\overline{\Delta}_H(G)<\overline{\Delta}_H(F)=20$. Hence $1\le s\le 4$.
		
		If $s=1$, then (\ref{x3}) implies $x=0$, but the property $\frac xs<-\frac 12$ does not hold. If $s=2$, then $x=-1$, and the property $\frac xs<-\frac 12$ does not hold too. If $s=3$, then (\ref{x3}) gives $x=-2$. In this case we have $\overline{\Delta}_H(G)=100-150y\in[0,20)$ and $y\in\frac 15\mathbb Z$, hence $y=\frac 35$. However, in this case the inequality $\overline{\Delta}_H(G)+\overline{\Delta}_H(F/G)<\overline{\Delta}_H(F)$ does not hold (by abuse of notation we assume that $G$ is a subobject; if it is a quotient, then we can consider the kernel of $F\twoheadrightarrow G$ instead of $F/G$). And if $s=4$, then there are no integer $x$ satisfying (\ref{x3}). We have proved that there are no walls for $F$, in particular, $F$ is a 2-semistable sheaf.
		
		Tilt-semistability of $E$ implies that $H^0(F)=0$, and the absence of semicircular walls gives $H^2(F)\cong\mathrm{Hom}(F,\mathcal O_X(-2)[1])^\vee=0$. Grothendieck-Riemann-Roch Theorem gives $\chi(F)=5e+\frac 23\le 0$, hence $e\le-\frac{2}{15}$. 
		
		Suppose that $F$ is reflexive, then $\mathrm{ch}(F^\vee)=4+2H+\frac 25H^3+e'H^3$ for some $e'$ (by \cite[Lemma 2.17]{Vass}). The sheaf $F^\vee$ is slope semistable as a dual of a slope semistable sheaf. If $F^\vee$ is not slope stable, then it is an extension between two slope stable sheaves $G',G''$ with $\mathrm{ch}_{\le 1}(G')=\mathrm{ch}_{\le 1}(G'')=2+H$. However, Bogomolov--Gieseker Inequality together with integrality of Chern classes imply that $\frac{\mathrm{ch}_2(G')}{H^2}\le\frac{1}{10},\frac{\mathrm{ch}_2(G'')}{H^2}\le\frac{1}{10}$, contradicting $\mathrm{ch}_2(F^\vee)=\frac 25H^2$. Hence, $F^\vee$ is slope stable and tilt-stable.
		
		Note that $\mathrm{ch}(F^\vee(-1))=4-2H+\frac 25H^2+(e'-\frac{1}{15})H^3$ and $F^\vee(-1)$ is tilt-stable too. 
		The same arguments as for $F$ imply that $e'-\frac{1}{15}\le-\frac{2}{15}$, hence $e'\le-\frac{1}{15}$. We get a contradiction with the fact that $\mathrm{ch}_3(F)+\mathrm{ch}_3(F^\vee)\ge 0$ for a reflexive sheaf $F$ (\cite[Lemma 2.17]{Vass}). Hence, $F$ is not reflexive and we can consider an exact triple
		$$0\to F\to F^{\vee\vee}\to T\to 0,$$
		where $\dim T\le 1$. If $\dim T=0$, then $F^{\vee\vee}$ is tilt-semistable and reflexive with $\mathrm{ch}_{\le 2}(F^{\vee\vee})=4-2H+\frac 25H^2$, and we again get a contradiction. And if $\dim T=1$, then $\frac{\mathrm{ch}_2(T)}{H^2}\ge\frac 15$ and we get a contradiction with Bogomolov--Gieseker Inequality for $F^{\vee\vee}$.
	\end{proof}
	
	\begin{lemma}\label{2 0 -3/5} If $E\in\mathrm{Coh}^\beta(X)$ is tilt-semistable and $\mathrm{ch}(E)=2-\frac 35H^2+eH^3$, then $e\le 0$, and in case of equality $E$ is the cohomology sheaf of a monad
		\begin{equation}\label{3-term}
			\mathcal Q(-1)^{\oplus 3}\to\mathcal U^{\oplus 6}\to\mathcal O_X.
		\end{equation}
	\end{lemma}
	
	\begin{proof} We have $\beta_-(E)=-\sqrt{\frac 35}\in(-1,-\frac 12)$. The hyperbola $\nu_{\alpha,\beta}(E)=0$ intersects the line $\alpha=\beta+1$ at the point with $\alpha=\frac 15$, hence either $E$ is $\nu_{\alpha,\beta}$-semistable at some point $(\beta,\alpha)\in D$ with $\nu_{\alpha,\beta}(E)=0$ (so $E[1]\in\mathfrak C$ by Lemma \ref{nu-lambda} and Proposition \ref{heart}), or $E$ is tilt-semistable above some semicircular wall of radius $\rho\ge\frac 15$. Assume that there is such a wall, induced by a subobject or quotient $F$ with $\mathrm{ch}_{\le 2}(F)=s+xH+yH^2,s>0,x<0$. If $s>2$, then Proposition \ref{actual}~(1) implies that $\frac{60}{100s(s-2)}\ge\frac{1}{25}$, that is, $s\le 5$. Hence we can assume that $1\le s\le 5$.
		
		We have $0<\mathrm{ch}_1^{\beta_-(E)}(F)<\mathrm{ch}_1^{\beta_-(E)}(E)$, i.e.
		\begin{equation}\label{x4}0<x+\frac{\sqrt{15}}{5}s<\frac{2\sqrt{15}}{5}.
		\end{equation}
		
		If $s=1$, then (\ref{x4}) implies that $x=0$, contradicting our assumptions.
		
		If $s=2$, then (\ref{x4}) gives $x=-1$. We should have $\overline{\Delta}_H(F)=25-100y\in[0,60)$ and $y\in\frac 12+\frac 15\mathbb Z$, hence $y\in\{-\frac{3}{10},-\frac{1}{10},\frac{1}{10}\}$. Since $\overline{\Delta}_H(F)+\overline{\Delta}_H(E/F)=\overline{\Delta}_H(F)+25<\overline{\Delta}_H(E)=60$, only the case $y=\frac{1}{10}$ remains possible. However, in this case the equation defining $W(E,F)$ defines an empty set.
		
		If $s=3$, then (\ref{x4}) implies that $x\in\{-2,-1\}$. Assume that $x=-2$, then $\overline{\Delta}_H(F)=100-150y\in[0,60)$ and $y\in\frac 15\mathbb Z$, hence $y\in\{\frac 25,\frac 35\}$. The condition $\overline{\Delta}_H(F)+\overline{\Delta}_H(E/F)<\overline{\Delta}_H(E)$ rules out the first case, and Lemma \ref{3 -2 3/5} prohibits the second. Assume now that $x=-1$, then $\overline{\Delta}_H(F)=25-150y\in[0,60)$ and $y\in\frac 12+\frac 15\mathbb Z$, hence $y\in\{-\frac{1}{10},\frac{1}{10}\}$. The condition $\overline{\Delta}_H(E/F)\ge 0$ implies that only the case $y=-\frac{1}{10}$ remains possible. However, in this case $W(E,F)$ is on the wrong side of the vertical wall.
		
		If $s=4$, then (\ref{x4}) gives $x\in\{-3,-2\}$. Let $x=-3$, then $\overline{\Delta}_H(F)=225-200y\in[0,60)$ and $y\in\frac 12+\frac 15\mathbb Z$, so $y\in\{\frac{9}{10},\frac{11}{10}\}$. We have $\overline{\Delta}_H(E/F)=165-100y$, and the property $\overline{\Delta}_H(F)+\overline{\Delta}_H(E/F)<\overline{\Delta}_H(E)$ in both cases is not satisfied. Let now $x=-2$, then $\overline{\Delta}_H(F)=100-200y\in[0,60)$ and $y\in\frac 15\mathbb Z$, hence $y=\frac 25$. However, by Lemma \ref{4 -2 2/5} there are no tilt-semistable objects $F$ with such a Chern character.
		
		If $s=5$, then (\ref{x4}) implies that $x=-3$. In this case $\overline{\Delta}_H(F)=225-250y\in[0,60)$ and $y\in\frac 12+\frac 15\mathbb Z$, hence $y\in\{\frac{7}{10},\frac{9}{10}\}$. The condition $\overline{\Delta}_H(F)+\overline{\Delta}_H(E/F)<\overline{\Delta}_H(E)$ is satisfied only for $y=\frac{9}{10}$. However, in this case $W(E,F)$ is empty.
		
		We have proved that $E[1]\in\mathfrak C$. Solving the equation $\mathrm{ch}(E[1])=a\cdot\mathrm{ch}(\mathcal O_X(-1)[3])+b\cdot\mathrm{ch}(\mathcal Q(-1)[2])+c\cdot\mathrm{ch}(\mathcal U[1])+d\cdot\mathrm{ch}(\mathcal O_X)$ we find $a=-5e$, but $a$ should be a non-negative integer, hence $e\le 0$. For $e=0$ we find $b=3,c=6,d=1$, and Remark \ref{complexes} gives (\ref{3-term}).
	\end{proof}
	
	\begin{remark}\label{instantons}Since $\mathcal U$ and $\mathcal Q$ are ACM sheaves, it follows that for cohomology sheaves $E$ of monads (\ref{3-term}) we have $H^1(E(-1))=H^1(E(-2))=H^2(E(-1))=H^2(E)=0$. Also, the sheaves $E$ have vanishing odd Chern classes. If such a sheaf $E$ is slope semistable, then it is an instanton sheaf in the terminology of \cite{Higher rank}. According to \cite[Main Theorem 1]{Higher rank} there indeed exist such sheaves, which are moreover locally free and slope stable.
	\end{remark}
	
	\begin{lemma}\label{2 -1 -1/2} If $E\in\mathrm{Coh}^\beta(X)$ is tilt-semistable and $\mathrm{ch}(E)=2-H-\frac{H^2}{2}+eH^3$, then $e\le\frac 56$, and in case of equality $E$ is destabilized by an exact triple
		\begin{equation}\label{sequence2}0\to \mathcal O_X(-1)^{\oplus 3}\to E\to\mathcal O_X(-2)[1]\to 0.\end{equation}
	\end{lemma}
	\begin{proof} Since $H^2\cdot\mathrm{ch}_1^{-1}(E)$ has the least positive value, there are no semicircular walls intersecting the ray $\beta=-1$. Hence $Q_{0,-1}(E)=125-150e\ge 0$, and $e\le\frac 56$. Assume that $e=\frac 56$. A calculation shows that the numerical wall $W(E,\mathcal O_X(-3)[1])$ intersects the ray $\beta=-1$. It follows that $\mathrm{Ext}^2(\mathcal O_X(-1),E)\cong\mathrm{Hom}(E,\mathcal O_X(-3)[1])^\vee=0$. Then we have $\mathrm{hom}(\mathcal O_X(-1),E)\ge\chi(\mathcal O_X(-1),E)=3$, giving an exact triple
		$$0\to \mathcal O_X(-1)^{\oplus 3}\to E\to G\to 0.$$
		
		The object $G$ is tilt-semistable and $\mathrm{ch}(G)=-1+2H-2H^2+\frac 43H^3$. By Proposition \ref{duality} there is a distinguished triangle
		$$\widetilde{G}\to \mathbb D(G)\to T[-1]\to \widetilde{G}[1],$$
		in which $\widetilde{G}$ is a tilt-semistable object and $T$ is a sheaf with $\dim T\le 0$. We find $\mathrm{ch}(\widetilde{G})=1+2H+2H^2+\frac 43H^3+\mathrm{ch}_3(T)$, hence, by \cite[Proposition 3.1]{Vass}, we should have $T=0$ and $\widetilde{G}\cong\mathbb D(G)\cong\mathcal O_X(2)$. So, $G\cong\mathcal O_X(-2)[1]$.
	\end{proof}
	
	\begin{remark}\label{inequalities}Assume that $E\in\mathrm{Coh}^\beta(X)$ is tilt-semistable and $\mathrm{ch}_{\le 2}(E)=2+dH^2$. Assume that there is a semicircular wall for $E$ intersecting the ray $\beta=-1$ and induced by a subobject or quotient $F$ with $\mathrm{ch}_{\le 2}(F)=s+xH+yH^2$. We should have $0<\mathrm{ch}_1^{-1}(F)<\mathrm{ch}_1^{-1}(E)=2H$, hence $\mathrm{ch}_1^{-1}(F)=H$, that is, $x=1-s$. Note that in the case $s=1$ we have $x=0$, and $W(E,F)$ is not semicircular. Without loss of generality we can assume that $s>1$, and then we have conditions
		\begin{equation}\label{F positive}\overline{\Delta}_H(F)\ge 0\Leftrightarrow y\le\frac s2-1+\frac{1}{2s},
		\end{equation}
		\begin{equation}\label{wall at beta=-1}(\exists \alpha>0\colon\nu_{\alpha,-1}(F)=\nu_{\alpha,-1}(E))\Leftrightarrow y>\frac s2+\frac{d-1}{2}.
		\end{equation}
		And if $s>2$, then also
		\begin{equation}\label{E/F positive}\overline{\Delta}_H(E/F)\ge 0\Leftrightarrow y\le\frac s2+d+\frac{1}{2(s-2)}.
		\end{equation}
		As usual, if $F$ is a quotient of $E$, then we denote by $E/F$ the kernel of the map $E\twoheadrightarrow F$.
	\end{remark}
	
	\begin{lemma}
		\label{2 0 -4/5} If $E\in\mathrm{Coh}^\beta(X)$ is tilt-semistable and $\mathrm{ch}(E)=2-\frac 45H^2+eH^3$, then $e\le\frac 25$. In case of equality $E$ is destabilized by an exact triple
		\begin{equation}\label{triple3}0\to\mathcal U^{\oplus 2}\to E\to \mathcal O_X(-1)^{\oplus 2}[1]\to 0.
		\end{equation}
	\end{lemma}
	
	\begin{proof}Assume that $e\ge\frac 25$. Suppose that there is a semicircular wall intersecting the ray $\beta=-1$ and induced by a subobject or quotient $F$ with $\mathrm{ch}(F)=s+xH+yH^2+e'H^3, s>1$. As in Remark \ref{inequalities}, $x=1-s$.
		
		If $s=2$, then $x=-1,y\in\frac 12+\frac 15\mathbb Z$. Conditions (\ref{F positive}), (\ref{wall at beta=-1}) give $y=-\frac{1}{10}$. However, in this case the condition $\beta_-(F)>\beta_-(E)$ from Proposition \ref{actual}~(5) is not satisfied.
		
		If $s\ge 3$ is odd, then $x$ is even, hence $y\in\frac 15\mathbb Z$. Conditions (\ref{F positive}), (\ref{wall at beta=-1}) imply that $y\in(\frac s2-\frac{9}{10},\frac s2-1+\frac{1}{2s}]$. However, in this case $\frac s2-\frac{9}{10}\in\frac 15\mathbb Z$ and the length of this semi-interval is equal to $\frac{1}{2s}-\frac{1}{10}<\frac 15$, hence it does not contain any $y\in\frac 15\mathbb Z$. Similarly, if $s\ge 4$ is even, then $y\in\frac 12+\frac 15\mathbb Z,\frac s2-\frac{9}{10}\in\frac 12+\frac 15\mathbb Z$, and again there are no suitable $y$.
		
		A calculation shows that $W(E,\mathcal O_X(1))$ lies on the wrong side of the vertical wall and $\nu_{\alpha,\beta}(\mathcal O_X(1))>\nu_{\alpha,\beta}(E)$ for all $\beta<0,\alpha>0$. Hence, there are no nonzero morphisms $\mathcal O_X(1)\to E$, and by Serre duality $\mathrm{Ext}^3(E,\mathcal O_X(-1))\cong\Hom(\mathcal O_X(1),E)^\vee=0$. Riemann--Roch gives $\mathrm{hom}(E,\mathcal O_X(-1)[1])\ge -\chi(E,\mathcal O_X(-1))=\chi(\mathcal O_X(1),E)=5e\ge 2$. A morphism $E\to\mathcal O_X(-1)[1]$ destabilizes $E$ below $W(E,\mathcal O_X(-1)[1])$, hence $E$ is unstable below this wall. The wall $W(E,\mathcal O_X(-1)[1])$ is given by $\alpha^2+(\beta-\frac{9}{10})^2=(\frac{1}{10})^2$, and any wall lying above $W(E,\mathcal O_X(-1)[1])$ would intersect the ray $\beta=-1$. Hence, $E$ is tilt-semistable at $W(E,\mathcal O_X(-1)[1])$. By Lemma \ref{quotient} we have an exact triple
		$$0\to G\to E\to\mathcal O_X(-1)^{\oplus 2}[1]\to 0,$$
		where $G$ is tilt-semistable with $\mathrm{ch}(G)=4-2H+\frac{H^2}{5}+(e-\frac 13)H^3$. 
		
		The numerical wall $W(E,\mathcal U(-2)[1])$ intersects the ray $\beta=-1$. Since a nonzero morphism $E\to\mathcal U(-2)[1]$ would destabilize $E$ below $W(E,\mathcal U(-2)[1])$, it follows that $\mathrm{Ext}^2(\mathcal U,E)\cong\mathrm{Hom}(E,\mathcal U(-2)[1])^\vee=0$. Riemann--Roch gives $\mathrm{hom}(\mathcal U,E)\ge\chi(\mathcal U,E)=10e-2\ge 2$. In particular, there is a nonzero morphism $f:\mathcal U\to E$. Since $\Hom(\mathcal U,\mathcal O_X(-1)[1])=0$, $f$ factors through a morphism $g:\mathcal U\to G$. We have $\nu_{\alpha, \beta}(\mathcal U)=\nu_{\alpha, \beta}(G)$ for all $(\beta,\alpha)\in\mathbb H$ and $\mathcal U$ is tilt-stable, hence $g$ is a monomorphism. The object $G/\im g$ is tilt-semistable and has $\mathrm{ch}(G/\im g)=2-H+\frac{H^2}{10}+(e-\frac{11}{30})H^3$, hence by \cite[Lemma 3.6]{Vass} $e\le\frac 25$. In case of equality $G/\im g\cong\mathcal U$, so $G\cong\mathcal U^{\oplus 2}$.
	\end{proof}
	
	\begin{lemma}\label{2 -1 -7/10} Let $E\in\mathrm{Coh}^\beta(X)$ be a tilt-semistable object with $\mathrm{ch}(E)=2-H+dH^2+eH^3$.
		\begin{enumerate}
			\item If $d=-\frac{7}{10}$, then $e\le\frac{31}{30}$. In case of equality $E$ is included into an exact triple
			$$0\to \mathcal O_X(-1)^{\oplus 2}\to E\to\mathcal I_{L,H}(-1)\to 0.$$
			\item If $d=-\frac{9}{10}$, then $e\le\frac{43}{30}$. In case of equality $E$ is included into an exact triple
			$$0\to \mathcal O_X(-1)^{\oplus 2}\to E\to\mathcal I_{C,H}(-1)\to 0.$$
		\end{enumerate}	
	\end{lemma}
	
	\begin{proof} 
		\begin{enumerate}
			\item 
			Since $H^2\cdot\mathrm{ch}_1^{-1}(E)$ has the minimal value, there are no walls at $\beta=-1$. We should have $Q_{0,-1}(E)=179-150e\ge 0$, so $e\le\frac{179}{150}$. But $\chi(E)=5e-\frac{25}{6}\in\mathbb Z$, so $e\in\frac 56+\frac 15\mathbb Z$, and we get a bound $e\le\frac{31}{30}$. Assume that $e=\frac{31}{30}$.
			
			A calculation shows that $W(E,\mathcal O_X(-3)[1])$ intersects the ray $\beta=-1$. Therefore, $\mathrm{Ext}^2(\mathcal O_X(-1),E)\cong\Hom(E,\mathcal O_X(-3)[1])^\vee=0$. Grothendieck--Riemann--Roch gives $\Hom(\mathcal O_X(-1),E)\ge\chi(\mathcal O_X(-1),E)=2$. Lemma \ref{subobject} implies that we have an exact triple
			\begin{equation}\label{triple5}0\to\mathcal O_X(-1)^{\oplus 2}\to E\to G\to 0.\end{equation}
			
			We have $\mathrm{ch}(G(1))=H-\frac{7}{10}H^2+\frac{H^3}{6}$. The exact triple (\ref{triple5}), twisted by $\mathcal O_X(1)$, implies that $G(1)$ is tilt-semistable at the wall $W(\mathcal O_X,G(1))$, which is located above $W(\mathcal O_X(-1)[1],G(1))$. Hence, by Lemma \ref{0 1 -7/10} $G(1)\cong\mathcal I_{L,H}$.
			
			\item Here also $H^2\cdot\mathrm{ch}_1^{-1}(E)$ has the minimal value, so $Q_{0,-1}(E)=241-150e\ge 0$. We have $\chi(E)=5e-\frac{31}{6}\in\mathbb Z$, so $e\in\frac 56+\frac 15\mathbb Z$, and we get a bound $e\le\frac{43}{30}$. Assume that $e=\frac{43}{30}$. Again $W(E,\mathcal O_X(-3)[1])$ intersects the ray $\beta=-1$. Hence, $\mathrm{Ext}^2(\mathcal O_X(-1),E)\cong\Hom(E,\mathcal O_X(-3)[1])^\vee=0$. A calculation gives $\Hom(\mathcal O_X(-1),E)\ge\chi(\mathcal O_X(-1),E)=2$. We again get an exact triple of the form (\ref{triple5}). Now $\mathrm{ch}(G(1))=H-\frac{9}{10}H^2+\frac{11}{30}H^3$. The object $G(1)$ is tilt-semistable at the wall $W(\mathcal O_X,G(1))$, which is located above $W(\mathcal O_X(-1)[1],G(1))$. Hence, by Lemma \ref{0 1 -9/10} $G(1)\cong\mathcal I_{C,H}$.
		\end{enumerate}
	\end{proof}
	
	\begin{lemma}
		\label{2 0 -1} If $E\in\mathrm{Coh}^\beta(X)$ is tilt-semistable and $\mathrm{ch}(E)=2-H^2+eH^3$, then $e\le\frac 35$. In case of equality the object $E$ is destabilized by an exact triple
		\begin{equation}\label{triple8}0\to\mathcal U\to E\to \mathbb D(\mathcal I_{C,H})(-2)\to 0.\end{equation}
	\end{lemma}
	
	\begin{proof}Assume that $e\ge\frac 35$. Any semicircular wall for $E$ has to intersect the ray $\beta=\beta_-(E)=-1$. Suppose that there is such a wall, induced by a subobject or quotient $F$ with $\mathrm{ch}(F)=s+xH+yH^2+e'H^3,s>1$. As in Remark \ref{inequalities}, $x=1-s$.
		
		If $s=2$, then $x=-1$ and $y\in\frac 12+\frac 15\mathbb Z$. The condition $\overline{\Delta}_H(F)\ge 0$ implies that $y\le\frac{1}{10}$. If $y\le-\frac{1}{10}$, then $\beta_-(F)<\beta_-(E)$ and the condition from Proposition \ref{actual}~(5) is not satisfied. Hence, $y=\frac{1}{10}$. By Lemma \cite[Lemma 3.6]{Vass} $e'\le\frac{1}{30}$, and in case of equality $F\cong\mathcal U$. We find $\mathrm{ch}(E/F)=H-\frac{11}{10}H^2+e''H^3$; by Lemma \ref{0 1 -1/10} applied to $E/F(1)$ we get $e''\le\frac{17}{30}$. So, $e\le\frac 35$ in this case. The object $E/F(1)$ is tilt-semistable along the numerical wall $W(E(1),\mathcal U(1))$, which lies above $W(E/F(1),\mathcal O_X)$, hence $E/F(1)\cong \mathbb D(\mathcal I_{C,H})(-1)$. Using the long exact sequence of $\mathrm{Ext}$ groups associated with (\ref{triple24}) and the vanishing $\mathrm{Ext}^2(\mathcal U,\mathcal U(-1))=0$ (which follows from the fact that $\mathcal U\otimes\mathcal U$ is an ACM sheaf, see \cite[proof of Proposition 5.6]{Faenzi}), one can prove that $\mathrm{Ext}^1(\mathcal U,\mathbb D(\mathcal I_{C,H})(-2))=0$. Hence, $F\cong\mathcal U$ is indeed a subobject, and we get the exact triple (\ref{triple8}). We also calculate $\rho(E,\mathcal U)^2=\frac{21}{100}$.
		
		If $s=3$, then $x=-2$ and $y\in\frac 15\mathbb Z$. Bogomolov--Gieseker inequality for $F$ together with Lemma \ref{3 -2 3/5} imply that $y\le\frac 25$. However, in this case the condition $\beta_-(E)<\beta_-(F)$ is not satisfied. 
		
		If $s\ge 4$, then Proposition \ref{actual}~(1) implies that $\rho(E,F)^2\le\frac 18$. A nonzero morphism $E\to\mathcal U(-2)[1]$ would destabilize $E$ below $W(E,\mathcal U(-2)[1])$, but $\rho(E,\mathcal U(-2)[1])^2=\frac{2541}{2500}$ is bigger than $\frac{21}{100}$ and $\frac 18$, hence this is impossible. Therefore, $\mathrm{Ext}^2(\mathcal U,E)\cong\mathrm{Hom}(E,\mathcal U(-2)[1])^\vee=0$. Riemann--Roch gives $\dim \Hom(\mathcal U,E)\ge\chi(\mathcal U,E)=10e-5\ge 1$. It means that there is a nonzero morphism $\mathcal U\to E$, contradicting tilt-semistability of $E$ below $W(E,\mathcal U)$. In particular, any tilt-semistable $E$ with $e\ge\frac 35$ should be destabilized at some wall. Since all numerical walls $W(E,F)$ for $s\ge 4$ lie below $W(E,\mathcal U)$, it follows that they are not actual walls. Therefore, any $E$ with maximal $\mathrm{ch}_3(E)$ is destabilized by an exact triple (\ref{triple8}).
	\end{proof}
	
	\begin{lemma}\label{2 0 -6/5} If $E\in\mathrm{Coh}^\beta(X)$ is tilt-semistable and $\mathrm{ch}(E)=2-\frac 65H^2+eH^3$, then $e\le\frac 45$. In case of equality $E$ is destabilized by one of exact triples
		\begin{equation}\label{triple10}0\to\mathcal U\to E\to \mathbb D(\mathcal I_{L,H})(-2)\to 0,
		\end{equation}
		\begin{equation}\label{triple11}0\to \mathbb D(\mathcal I_{L,H})(-2)\to E\to\mathcal U\to 0,
		\end{equation}
		\begin{equation}\label{triple12}0\to\mathcal O_X(-1)^{\oplus 6}\to E\to\mathcal U(-1)^{\oplus 2}[1]\to 0.
		\end{equation}
	\end{lemma}
	
	\begin{proof}Assume that $e\ge\frac 45$. Consider the case in which $E$ is destabilized at a semicircular wall intersecting the ray $\beta=-1$ and induced by a subobject or quotient $F$ with $\mathrm{ch}(F)=s+xH+yH^2+e'H^3,s>1$. As in Remark \ref{inequalities}, $x=1-s$.
		
		If $s=2$, then $x=-1$ and $y\in\frac 12+\frac 15\mathbb Z$. Conditions (\ref{F positive}), (\ref{wall at beta=-1}) imply that $y=\frac{1}{10}$. By \cite[Lemma 3.6]{Vass} $\frac{\mathrm{ch}_3(F)}{H^3}\le\frac{1}{30}$, and in case of equality $F\cong\mathcal U$. We have $\mathrm{ch}(E/F)=H-\frac{13}{10}H^2+e''H^3$, so $\mathrm{ch}((E/F)(1))=H-\frac{3}{10}H^2+(e''-\frac 45)H^3$. The object $(E/F)(1)$ is tilt-semistable at the wall $W((E/F)(1),\mathcal U(1))$, which is located above $W((E/F)(1),\mathcal O_X)$. Lemma \ref{0 1 -3/10} implies that $e''\le\frac{23}{30}$ and in case of equality $(E/F)(1)\cong \mathbb D(\mathcal I_{L,H})(-1)$. So, the maximal $e$ in this case is $\frac 45$, and for $e=\frac 45$ we obtain exact triples (\ref{triple10}), (\ref{triple11}).
		
		If $s=3$, then $x=-2$ and $y\in\frac 15\mathbb Z$. Conditions (\ref{F positive}), (\ref{wall at beta=-1}) imply that $y=\frac 35$, but there are no such tilt-semistable objects $F$ by Lemma \ref{3 -2 3/5}.
		
		If $s\ge 4$ is even, then $x$ is odd and $y\in\frac 12+\frac 15\mathbb Z$. Conditions (\ref{F positive}), (\ref{wall at beta=-1}) imply that $y\in(\frac s2-\frac{11}{10},\frac s2-\frac 65+\frac{1}{2(s-2)}]$. However, $\frac s2-\frac{11}{10}\in\frac 12+\frac 15\mathbb Z$ and length of this semi-interval is equal to $\frac{1}{2(s-2)}-\frac{1}{10}<\frac 15$, hence it does not contain a point from $\frac 12+\frac 15\mathbb Z$. Similarly, if $s\ge 5$ is odd, then $y\in\frac 15\mathbb Z$, and there is no such a point in the semi-interval given above.
		
		Consider now the case in which $E$ is not destabilized at a wall intersecting the ray $\beta=-1$. Since $W(E,\mathcal O_X(-3)[1])$ intersects this ray, it follows that $\mathrm{Ext}^2(\mathcal O_X(-1),E)\cong\Hom(E,\mathcal O_X(-3)[1])^\vee=0$. Riemann--Roch gives $\dim \Hom(\mathcal O_X(-1),E)\ge\chi(E(1))\ge 6$ (for $e\ge\frac 45$). Therefore, $E$ is unstable below $W(E,\mathcal O_X(-1))$. Since any wall lying above $W(E,\mathcal O_X(-1))$ intersects the ray $\beta=-1$, the object $E$ has to be tilt-semistable at $W(E,\mathcal O_X(-1))$ in this case. Lemma \ref{subobject} gives an exact triple
		\begin{equation}\label{triple13}0\to\mathcal O_X(-1)^{\oplus 6}\to E\to G\to 0
		\end{equation}
		with a tilt-semistable object $G$. Note that $\nu_{\alpha,\beta}(\mathcal U(1))>\nu_{\alpha,\beta}(E)$ for all $\beta<0,\alpha>0$ ($W(E,\mathcal U(1))$ lies on the wrong side of the vertical wall), hence $\mathrm{Ext}^3(E,\mathcal U(-1))\cong\Hom(\mathcal U(1),E)^\vee=0$. Riemann--Roch and Serre duality give $\dim \Hom(E,\mathcal U(-1)[1])\ge-\chi(E,\mathcal U(-1))=\chi(\mathcal U(1),E)\ge 2$. In particular, there is a nonzero morphism $f:E\to\mathcal U(-1)[1]$. Since $\mathrm{Ext}^1(\mathcal O_X(-1)^{\oplus 6},\mathcal U(-1))=0$, the exact triple (\ref{triple13}) implies that $f$ is induced by some $g:G\to\mathcal U(-1)[1]$. Note that $\mathcal U(-1)[1]$ is tilt-stable for all $\beta>-\frac 32,\alpha>0$. Indeed, applying Proposition \ref{duality} and \cite[Lemma 3.6]{Vass}, we get that $\mathcal U(-1)[1]$ has the maximal $\mathrm{ch}_3$, which is possible for tilt-semistable objects with the same $\mathrm{ch}_{\le 2}$. Then, applying Proposition \ref{duality} to $E\cong\mathcal U(2)$ and using Lemma \ref{tilt-stable}, we get what we need. Since $\nu_{\alpha,\beta}(G)=\nu_{\alpha,\beta}(\mathcal U(-1)[1])$ for all $(\beta,\alpha)\in\mathbb H$ and $\mathcal U(-1)[1]$ is tilt-stable, it follows that $g$ should be an epimorphism. We find $\mathrm{ch}(\ker g)=-2+3H-\frac{21}{10}H^2+(e+\frac{1}{10})H^3$. Applying Proposition \ref{duality} and \cite[Lemma 3.6]{Vass}, we get that $e\le\frac 45$, and in case of equality $\ker g\cong\mathcal U(-1)[1]$. Therefore, $G\cong\mathcal U(-1)^{\oplus 2}[1]$, and we obtain the exact triple (\ref{triple12}).
	\end{proof}
	
	\begin{lemma}\label{2 -1 -11/10} If $E\in\mathrm{Coh}^\beta(X)$ is tilt-semistable and $\mathrm{ch}(E)=2-H-\frac{11}{10}H^2+eH^3$, then $e\le\frac{11}{6}$. In case of equality $E$ is destabilized by an exact triple
		\begin{equation}\label{triple14}0\to\mathcal O_X(-1)^{\oplus 2}\to E\to \mathbb D(\mathcal I_{C,H})(-3)\to 0.
		\end{equation}
	\end{lemma}
	
	\begin{proof} Assume that $e\ge\frac{11}{6}$. A calculation gives $\rho_Q(E)\ge\frac{31185}{72900}$ ($\rho_Q$ grows with growing $e$, so it suffices to consider the case $e=\frac{11}{6}$, in which $Q_{\alpha,\beta}(E)=135\alpha^2+135\beta^2+495\beta+396$ and $\rho_Q(E)=\frac{31185}{72900}$). In particular, $\rho_Q(E)>0$, hence $E$ is destabilized at some semicircular wall by a subobject or quotient $F$ with $\mathrm{ch}(F)=s+xH+yH^2+e'H^3,s>0,\mu(F)<\mu(E)$. The wall $W(E,F)$ cannot lie inside the semidisk bounded by $W_Q(E)$, since there are no tilt-semistable objects below $W_Q(E)$. So $\rho(E,F)\ge\rho_Q(E)\ge\frac{31185}{72900}>\frac{\overline{\Delta}_H(E)}{4(H^3)^2\cdot 4\cdot (4-2)}=\frac{135}{800}$, and Proposition \ref{actual}~(1) implies that $s<4$. Note also that $Q_{0,-\frac 32}(E)<0$, hence $W(E,F)$ intersects the ray $\beta=-\frac 32$.
		
		Suppose that $s=1$. We should have $0<H^2\cdot\mathrm{ch}_1^{-\frac 32}(F)<H^2\cdot\mathrm{ch}_1^{-\frac 32}(E)$, that is, $x+\frac 32\in(0,2)$. Since $\mu(F)<\mu(E)$, we have $x=-1$, and $y\in\frac 12+\frac 15\mathbb Z$. The condition $\overline{\Delta}_H(F)\ge 0$ implies that $y\le\frac 12$. Since $W(E,F)$ lies above $W_Q(E)$, we have $s(E,F)\le s_Q(E)$. By a direct calculation this can be rewritten as $y\ge\frac{11}{30}$, hence only the case $y=\frac 12$ is possible. In this case $\mathrm{ch}_3(F)$ is maximized by $F\cong\mathcal O_X(-1)$. We find $\mathrm{ch}_{\le 2}(E/F)=1-\frac 85H^2$, and by the proof of \cite[Proposition 3.2]{MS18} in the case of maximal $\mathrm{ch}_3(E/F)$ there is a monomorphism $F'=\mathcal O_X(-1)\to E/F$ ($\mathrm{ch}_3((E/F)/F')$ is maximized by $\mathbb D(\mathcal I_{C,H})(-3)$). Overall, we get a monomorphism $\mathcal O_X(-1)^{\oplus 2}\to E$, reducing to the case $s=2$.
		
		Now consider the case $s=2$. The condition $0<H^2\cdot\mathrm{ch}_1^{-\frac 32}(F)<H^2\cdot\mathrm{ch}_1^{-\frac 32}(E)$ can be rewritten as $x+3\in(0,2)$, hence $x=-2$ and $y\in\frac 15\mathbb Z$. The condition $\overline{\Delta}_H(F)\ge 0$ implies that $y\le 1$. The condition $s(E,F)\le s_Q(E)$ reads as $y\ge\frac{11}{15}$, and we get $y\in\{\frac 45,1\}$. In the case $y=\frac 45$ we have $\mathrm{ch}(F(1))=2-\frac{H^2}{5}+(e'+\frac{2}{15})H^3$. By \cite[Lemma 3.7]{Vass} $e'\le-\frac 13$. Also $\mathrm{ch}_{\le 2}((E/F)(1))=H-\frac{9}{10}H^2$, and we can apply Lemma \ref{0 1 -9/10}, obtaining in this case a bound $e\le\frac{43}{30}<\frac{11}{6}$. In the case $y=1$ $\mathrm{ch}_3(F)$ is maximized by $F\cong\mathcal O_X(-1)^{\oplus 2}$. Also $\mathrm{ch}_{\le 2}((E/F)(2))=H-\frac{H^2}{10}$, and by Lemma \ref{0 1 -1/10} $\mathrm{ch}_3(E/F)$ is maximized by $E/F\cong \mathbb D(\mathcal I_{C,H})(-3)$. In this case the maximal value of $e$ equals $\frac{11}{6}$. Also one can check that $\mathrm{Ext}^1(\mathcal O_X(-1)^{\oplus 2},\mathbb D(\mathcal I_{C,H})(-3))=0$, hence $\mathcal O_X(-1)^{\oplus 2}$ is indeed a subobject and $\mathbb D(\mathcal I_{C,H})(-3)$ is the quotient.
		
		Finally, in the case $s=3$ we have $x+\frac 92\in(0,2)$, hence $x\in\{-4,-3\}$. If $x=-4$, then the ray $\beta=-\frac 43$ is the vertical wall for $F$, which cannot intersect a semicircular wall $W(E,F)$. However, $Q_{0,-\frac 43}(E)<0$, and we get a contradiction. And if $x=-3$, then the condition $s(E,F)\le s_Q(E)$ reads as $y\ge\frac{11}{10}$. However, the condition $\overline{\Delta}_H(E/F)\ge 0$ is $y\le\frac{9}{10}$, so this case is also impossible.
	\end{proof}
	
	Let us give a preparatory lemma.
	
	\begin{lemma}\label{3 -2 2/5} If $E\in\mathrm{Coh}^\beta(X)$ is tilt-semistable and $\mathrm{ch}(E)=3-2H+\frac 25H^2+eH^3$, then $e\le \frac{1}{15}$. If $e=\frac{1}{15}$ and $E$ is tilt-semistable above $W(E,\mathcal O_X(-1))$, then $E\cong\mathcal Q(-1)$.
	\end{lemma}
	
	\begin{proof} Assume that $e\ge\frac{1}{15}$. Note that $H^2\cdot\mathrm{ch}_1^{-1}(E)$ has the least positive value, hence there are no walls at $\beta=-1$. Since $W(E,\mathcal O_X(-3)[1])$ intersects the ray $\beta=-1$, we have $\mathrm{Ext}^2(\mathcal O_X(-1),E)\cong\Hom(E,\mathcal O_X(-3)[1])^\vee=0$. Therefore $\dim \Hom(\mathcal O_X(-1),E)\ge\chi(\mathcal O_X(-1),E)\ge 5$. By Lemma \ref{subobject} we obtain an exact triple
		\begin{equation}\label{triple17}0\to\mathcal O_X(-1)^{\oplus 5}\to E\to G\to 0.
		\end{equation}
		This exact triple destabilizes $E$ below $W(E,\mathcal O_X(-1))$, and any wall lying above $W(E,\mathcal O_X(-1))$ would intersect the ray $\beta=-1$. Hence $E$ it tilt-semistable at $W(E,\mathcal O_X(-1))$, and $G$ is tilt-semistable. Proposition \ref{duality} gives an exact triangle $\widetilde{G}\to\mathbb D(G)\to T[-1]\to\widetilde{G}[1]$ with a tilt-semistable $\widetilde{G}$ and a sheaf $T$ of length $t\ge 0$. We have $\mathrm{ch}(\widetilde{G}(-2))=2-H+\frac{H^2}{10}+(t+e-\frac{1}{30})H^3$. \cite[Lemma 3.6]{Vass} implies that $e\le\frac{1}{15}$; in case of equality $T=0$ and $\widetilde{G}(-2)\cong\mathcal U$, so $G\cong\mathcal U(-1)[1]$. In other words, $E$ is a cone of a morphism $f:\mathcal U(-1)\to\mathcal O_X(-1)^{\oplus 5}$. Assume that $E$ is tilt-semistable above the (unique) wall $W(E,\mathcal O_X(-1))$, then $E$ is a sheaf and $f$ is a monomorphism. Moreover, $E$ is slope stable, hence $f$ corresponds to five linearly independent morphisms $\mathcal U(-1)\to\mathcal O_X(-1)$ (otherwise there would exist a nonzero morphism $E\to\mathcal O_X(-1)$). The space $\Hom(\mathcal U(-1),\mathcal O_X(-1))$ is five-dimensional, so $f$ is uniquely determined up to an action of $\mathrm{GL}(5)$. The exact triple (\ref{tautological}), twisted by $\mathcal O_X(-1)$, implies that $E\cong\mathcal Q(-1)$.
	\end{proof}
	
	\begin{lemma}\label{2 0 -7/5} If $E\in\mathrm{Coh}^\beta(X)$ is tilt-semistable and $\mathrm{ch}(E)=2-\frac 75 H^2+eH^3$, then $e\le\frac 65$. In case of equality $E$ is destabilized by an exact triple
		\begin{equation}\label{triple16}0\to\mathcal U\to E\to \mathcal O_H(-1)\to 0.
		\end{equation}
	\end{lemma}
	
	\begin{proof}Let $e\ge\frac 65$. We have $Q_{0,-1}(E)=336-300e<0$, hence $E$ is destabilized at a semicircular wall intersecting the ray $\beta=-1$ and induced by a subobject or quotient $F$ with $\mathrm{ch}(F)=s+xH+yH^2+e'H^3,s>1$. Moreover, we can compute $\rho_Q(E)\ge\frac{1984}{7840}>\frac{140}{800}=\frac{\overline{\Delta}_H(E)}{4(H^3)^2\cdot 4(4-2)}$, and Proposition \ref{actual}~(1) implies that $s<4$. As in Remark \ref{inequalities}, we have $x+s=1$.
		
	If $s=2$, then $x=-1,y\in\frac 12+\frac 15\mathbb Z$. Conditions (\ref{F positive}) and (\ref{wall at beta=-1}) imply that $y\in\{-\frac{1}{10},\frac{1}{10}\}$. Suppose that $y=-\frac{1}{10}$, then by \cite[Lemma 3.8]{Vass} $e'\le\frac{7}{30}$. Also, in this case $\mathrm{ch}_{\le 2}(E/F)=H-\frac{13}{10}H^2+e''H^3$, and, as we have seen in Lemma \ref{2 0 -6/5}, $e''\le\frac{23}{30}$. So, in this case $e\le 1<\frac 65$. Suppose now that $y=\frac{1}{10}$. In this case $\mathrm{ch}_3(F)$ is maximized by $F\cong \mathcal U$. Also we have $\mathrm{ch}(E/F)=1-\frac 32 H^2+e''H^3$, and, applying Corollary \ref{0 1}, we get a bound $e''\le\frac 76$; in the case of equality $E/F\cong\mathcal O_H(-1)$ (since $W(E,\mathcal U)$ is located above $W(E/F,\mathcal O_X(-1))$). It can be seen that $\mathrm{Ext}^1(\mathcal U,\mathcal O_H(-1))=0$, hence $\mathcal U$ is indeed a subobject and $\mathcal O_H(-1)$ is the quotient. So, we get an exact triple of the form (\ref{triple16}), and $e=\frac 65$ in this case.
		
	If $s=3$, then $x=-2,y\in\frac 15\mathbb Z$. Conditions (\ref{F positive}) and (\ref{wall at beta=-1}) imply that $y\in\{\frac 25,\frac 35\}$, but the case $y=\frac 35$ is prohibited by Lemma \ref{3 -2 3/5}. Let $y=\frac 25$, then Lemma \ref{3 -2 2/5} implies that $e'\le\frac{1}{15}$. We find $\mathrm{ch}(E/F)=-1+2H-\frac 95H^2+e''H^3$. Proposition \ref{duality} gives an exact triangle $\widetilde{G}\to\mathbb D(E/F)\to T[-1]\to\widetilde{G}[1]$ with a tilt-semistable $\widetilde{G}$ and a sheaf $T$ of length $t\ge 0$. We calculate $\mathrm{ch}(\widetilde{G}(-2))=1-\frac{H^2}{5}+(t+e''-\frac{14}{15})$, and Lemma \ref{line} gives a bound $e''\le\frac{14}{15}$. So, in this case $e\le 1<\frac 65$.\end{proof}
	
	\begin{lemma}\label{2 0 -8/5} If $E\in\mathrm{Coh}^\beta(X)$ is tilt-semistable and $\mathrm{ch}(E)=2-\frac 85 H^2+eH^3$, then $e\le\frac 75$. In case of equality $E$ is destabilized by one of the following exact triples:
		\begin{equation}\label{triple18}0\to F\to E\to \mathcal O_H(-1)\to 0,\end{equation}
		\begin{equation}\label{triple19}0\to \mathcal O_H(-1)\to E\to F\to 0,\end{equation}
		\begin{equation}\label{triple20}0\to\mathcal U\to E\to \mathcal I_{L,H}(-1)\to 0,\end{equation}
		\begin{equation}\label{triple23}0\to\mathcal I_{L,H}(-1)\to E\to\mathcal U\to 0,\end{equation}
		\begin{equation}\label{triple21}0\to\mathcal Q(-1)\to E\to \mathcal O_X(-2)[1]\to 0,\end{equation}
		where $F$ is the cokernel of a monomorphism of sheaves $\mathcal U(-1)\hookrightarrow\mathcal O_X(-1)^{\oplus 4}$.
	\end{lemma}
	
	\begin{proof} Assume that $e\ge\frac 75$. We have $Q_{0,-1}(E)<0$, so there is a semicircular wall for $E$ intersecting the ray $\beta=-1$ and induced by a subobject or quotient $F$ with $\mathrm{ch}(F)=s+xH+yH^2+e'H^3,s>1,x=1-s$.
		
	If $s=2$, then $x=-1$ and $y\in\frac 12+\frac 15\mathbb Z$. Conditions (\ref{wall at beta=-1}) and (\ref{F positive}) imply that $y\in\{-\frac{1}{10},\frac{1}{10}\}$. Let $y=-\frac{1}{10}$. By \cite[Lemma 3.8]{Vass} $e'\le\frac{7}{30}$, and in case of equality $F$ is the cokernel of a monomorphism of sheaves $\mathcal U(-1)\hookrightarrow\mathcal O_X(-1)^{\oplus 4}$. Also, we have $\mathrm{ch}(E/F)=H-\frac 32 H^2+e''H^3$. Corollary \ref{0 1} implies that $e''\le\frac 76$ and in case of equality $E/F\cong\mathcal O_H(-1)$. So, in this case $e\le\frac 75$, and for $e=\frac 75$ we obtain exact triples (\ref{triple18}), (\ref{triple19}). In the case $y=\frac{1}{10}$ \cite[Lemma 3.6]{Vass} implies that $e'\le\frac{1}{30}$ and in case of equality $F\cong \mathcal U$. Also, here $\mathrm{ch}(E/F)=H-\frac{17}{10} H^2+e''H^3$; applying Corollary \ref{0 1} we get a bound $e''\le\frac{41}{30}$ and in case of equality $E/F\cong\mathcal I_{L,H}(-1)$. In both cases $E/F$ is tilt-stable along $W(E,E/F)$. We obtain (\ref{triple20}) and (\ref{triple23}).
		
	If $s=3$, then $x=-2$ and $y\in\frac 15\mathbb Z$. Conditions (\ref{wall at beta=-1}) and (\ref{E/F positive}) imply that $y=\frac 25$. In this case Lemma \ref{3 -2 2/5} gives $e'\le\frac{1}{15}$, and for $e'=\frac{1}{15}$ $F\cong\mathcal Q(-1)$. We have $\mathrm{ch}(E/F)=-1+2H-2H^2+e''H^3$; \cite[Corollary 3.3]{Vass} implies that $e''\le\frac 43$ and in case of equality $E/F\cong\mathcal O_X(-2)[1]$. Since $\mathcal Q$ is an ACM sheaf, we have $\mathrm{Ext}^1(\mathcal Q(-1),\mathcal O_X(-2)[1])\cong\mathrm{Ext}^2(\mathcal Q(-1),\mathcal O_X(-2))=0$, hence $\mathcal Q(-1)$ is indeed a subobject, $\mathcal O_X(-2)[1]$ is the quotient and we obtain the exact triple (\ref{triple21}).
		
	Finally, if $s\ge 4$, then conditions (\ref{wall at beta=-1}) and (\ref{E/F positive}) give $y\in(\frac s2-\frac{13}{10},\frac s2-\frac 85+\frac{1}{2(s-2)}]$, but this semi-interval is empty.
	\end{proof}
	
	\begin{lemma}\label{2 0 -9/5} If $E\in\mathrm{Coh}^\beta(X)$ is tilt-semistable and $\mathrm{ch}(E)=2-\frac 95 H^2+eH^3$, then $e\le\frac 95$. In case of equality $E$ is destabilized by an exact triple
		\begin{equation}\label{triple22}0\to\mathcal U\to E\to \mathcal I_{C,H}(-1)\to 0.\end{equation}
	\end{lemma}
	
	\begin{proof}Assume that $e\ge\frac 95$. A calculation shows that $\rho_Q^2\ge\frac{9}{20}>\frac{180}{800}=\frac{\overline{\Delta}_H(E)}{4(H^3)^2\cdot 4(4-2)}$, so Proposition \ref{actual}~(1) implies that $E$ is destabilized at a semicircular wall induced by a subobject or quotient $F$ with $\mathrm{ch}(F)=s+xH+yH^2+e'H^3,1<s<4$. Besides, $W_Q(E)$ intersects the ray $\beta=-1$ at a point with $\alpha^2\ge\frac 15$, hence $W(E,F)$ also has to intersect this ray at a point with $\alpha^2\ge\frac 15$ in order to be an actual wall. Also, $x=1-s$ by Remark \ref{inequalities}.
		
		If $s=2$, then $x=-1$ and $y\in\frac 12+\frac 15\mathbb Z$. Conditions (\ref{F positive}) and (\ref{wall at beta=-1}) imply that $y\in\{-\frac{3}{10},-\frac{1}{10},\frac{1}{10}\}$. If $y=-\frac{3}{10}$, then by Lemma \ref{2 -1 -3/10} $e'\le\frac{13}{30}$. Also, $\mathrm{ch}(E/F)=H-\frac 32H^2+e''H^3$ and we obtain $e''\le\frac 76$ using Lemma \ref{0 1 -1/2}. So, in this case $e\le\frac{13}{30}+\frac 76=\frac 85<\frac 95$. If $y=-\frac{1}{10}$, then by \cite[Lemma 3.8]{Vass} $e'\le\frac{7}{30}$. We have $\mathrm{ch}(E/F)=H-\frac{17}{10}H^2+e''H^3$, and Lemma \ref{0 1 -7/10} implies that $e''\le\frac{41}{30}$. In this case also $e\le\frac 85<\frac 95$. If $y=\frac{1}{10}$, then by \cite[Lemma 3.6]{Vass} $e'\le\frac{1}{30}$ and in case of equality $F\cong\mathcal U$. We have $\mathrm{ch}(E/F)=H-\frac{19}{10}H^2+e''H^3$, and Corollary \ref{0 1} implies that $e''\le\frac{53}{30}$; in case of equality $E/F\cong\mathcal I_{C,H}(-1)$. Using the exact triple (\ref{triple15}) we find $\mathrm{Ext}^1(\mathcal U,\mathcal I_{C,H}(-1))=0$, hence in this case $\mathcal U$ is indeed a subobject and we get the exact triple (\ref{triple22}).
		
		Finally, if $s=3$, then $x=-2$ and $y\in\frac 15\mathbb Z$. Conditions (\ref{wall at beta=-1}) and (\ref{E/F positive}) imply that $y=\frac 15$, but in this case $W(E,F)$ intersects the ray $\beta=-1$ at a point with $\alpha^2=\frac{1}{10}<\frac 15$, hence $W(E,F)$ is not an actual wall.
	\end{proof}
	
	Now we come to the general case.
	
	\begin{lemma}\label{0 1 2} Let $E\in\mathrm{Coh}^\beta(X)$ be a tilt-semistable object with $\mathrm{ch}(E)=2+cH+dH^2+eH^3$. Suppose that either
		\begin{enumerate}
			\item $c=-1,d\le-\frac{13}{10},e\ge\frac{d^2}{2}-d+\frac{5}{24}-\frac{3}{25}$, or
			\item $c=0,d\le -2,e\ge\frac{d^2}{2}-\frac{d}{10}-\frac{1}{25}$.
		\end{enumerate}
		Then $E$ is destabilized at a semicircular wall by an exact triple, in which the subobject and the quotient have rank at most two.
	\end{lemma}
	
	\begin{proof}\begin{enumerate}
	\item A straightforward calculation gives $$\rho_Q(E)^2\ge\frac{22500d^4-5000d^3-7050d^2+7050d-\frac{2491}{4}}{4(25-100d)^2}.$$ 
	Using a partial fraction decomposition, we get
	$$\rho_Q(E)^2-\frac{\overline{\Delta}_H(E)}{300}\ge\frac{9}{16}d\left(d+\frac{47}{54}\right)-\frac{4159}{19200}-\frac{639}{1600(1-4d)}+\frac{45369}{160000(1-4d)^2}>0,$$
		where the last inequality is obtained by bounding each of the four summands for $d\le-\frac{13}{10}$. Proposition \ref{actual}~(1) implies that the rank of a destabilizing subobject or quotient cannot be bigger than or equal to three.
		\item We find $Q_{0,-1}(E)\le-50d^2-70d+12<0$ (for $d\le-2$), hence any semicircular wall $W(E,F)$ has to intersect the ray $\beta=-1$. As in Remark \ref{inequalities}, we have $\mathrm{ch}_{\le 2}(F)=s+xH+yH^2$, where $x=1-s$. If $s\ge 3$, then (\ref{wall at beta=-1}) and (\ref{E/F positive}) imply $\frac s2+\frac{d-1}{2}<y\le\frac s2+d+\frac{1}{2(s-2)}$, but for $d\le-2$ these conditions define an empty set.
		\end{enumerate}
	\end{proof}
	
	\begin{lemma}\label{one}Let $E\in\mathrm{Coh}^\beta(X)$ be a tilt-semistable object with $\mathrm{ch}(E)=2+cH+dH^2+eH^3$.
		\begin{enumerate}
			\item If $c=-1,d\le-\frac{13}{10}$ and $E$ is destabilized by a subobject of rank one, then
			$$e\le\frac{d^2}{2}-d+\frac{5}{24}-\varepsilon(d).$$
			In case of equality either the subobject or the quotient is isomorphic to $\mathcal O_X(-1)$.
			\item If $c=0,d\le -2$ and $E$ is destabilized at a semicircular wall by a subobject of rank one, then
			$$e<\frac{d^2}{2}-\frac{d}{10}+\frac{2}{25}-\varepsilon(d).$$
		\end{enumerate}
	\end{lemma}
	
	\begin{proof}
		\begin{enumerate}
			\item Assume that $e\ge\frac{d^2}{2}-d+\frac{5}{24}-\frac{3}{25}=\frac{d^2}{2}-d+\frac{53}{600}$, then a direct calculation shows that $Q_{0,-2}(E)\le -125d^2-50d+\frac{241}{4}<0$ (for $d\le -\frac{13}{10}$). Hence any semicircular wall $W$ has to intersect the ray $\beta=-2$. Suppose that $W$ is induced by a subobject or quotient $F$ of rank one. Since $\mathrm{ch}_1^{-2}(E)=3H$, we should have $\mathrm{ch}_1^{-2}(F)\in\{H,2H\}$, so $\mathrm{ch}_1(F)\in\{-H,0\}$. Without loss of generality (replacing $F$ by $E/F$ if necessary) we may assume that $\mathrm{ch}_1(F)=-H$. Then $\mathrm{ch}(F)=(1-yH^2+zH^3)\cdot\mathrm{ch}(\mathcal O_X(-1))$ and $y\ge 0$ (by Bogomolov--Gieseker inequality). We find
			$$s_Q(E)=\frac{d+6e}{4d-1}\le\frac{3d^2-5d+\frac{53}{100}}{4d-1},$$
			$$s(E,F)=d+2y-1.$$
			Since $s(E,F)\le s_Q(E)$, we get
			$$y\le\frac{d^2+\frac{47}{100}}{2-8d}<-\frac d2-\frac 14.$$
			Using the bound $z\le\frac{y(y+1)}{2}$ from \cite[Lemma 4.2]{Vass} and an analogous bound for $\mathrm{ch}_3(E/F)$, we obtain that
			$$e\le y^2+\left(d+\frac 12\right)y+\frac{d^2}{2}-d+\frac{5}{24}.$$
			On the right-hand side we have a quadratic function of $y$ with a minimum at $y_0=-\frac d2-\frac 14$. If $y\ge\frac 15$, then the maximum is reached at $y=\frac 15$ and is equal to $\frac{d^2}{2}-\frac 45d+\frac{209}{600}$. Assume for the moment that $d<-\frac{13}{10}$, then this maximum will be smaller than $\frac{d^2}{2}-d+\frac{53}{600}$. Therefore, we must have $y=0$. Then $\mathrm{ch}_3(F)$ is maximized when $z=0$ and $F\cong\mathcal O_X(-1)$. Also we have $\mathrm{ch}_{\le 2}(E/F)=1+(d+y-\frac 12)H^2$ and $d+y-\frac 12<-1$. Arguments from the proof of \cite[Proposition 3.2]{MS18} imply that in the case of maximal $\mathrm{ch}_3(E/F)$ there is a monomorphism $\mathcal O_X(-1)\to E/F$. Applying Lemma \ref{error} to the quotient $(E/F)/(\mathcal O_X(-1))$, we get the bound on $e$ that we need.
			
			Now consider the case $d=-\frac{13}{10}$. In this case the maximal value of $e$ for $y=0$ is equal to $\frac{67}{30}$. Also $y<-\frac d2-\frac 14=\frac 25$, hence it suffices to eliminate the case $y=\frac 15$. In this case $z\le 0$ by Lemma \ref{line}, hence $\frac{\mathrm{ch}_3(F)}{H^3}\le\frac{1}{30}$. Also $\mathrm{ch}_{\le 2}(E/F)=1-\frac 85H^2+e''H^3$, and \cite[Lemma 4.2]{Vass} gives a bound $e''\le\frac{104}{50}$. This, combined with the condition $\chi(E/F)\in\mathbb Z$, gives a bound $e''\le 2$. Overall, for $y=\frac 15$ $e\le\frac{61}{30}<\frac{67}{30}$, hence $e$ is maximal for $y=0$ and the arguments above apply.
			\item Assume for the contrary that $e\ge\frac{d^2}{2}-\frac{d}{10}-\frac{1}{25}$. As in Lemma \ref{0 1 2}~(2), we have $Q_{0,-1}(E)<0$, hence any semicircular wall $W$ has to intersect the ray $\beta=-1$. If $W$ is induced by a subobject $F$ of rank one, then, as in Remark \ref{inequalities}, we have $\mathrm{ch}_1(F)=0$, but then $W$ is not semicircular.
		\end{enumerate}
	\end{proof}
	
	\textit{Proof of Theorem \ref{main}.} As in \cite{Sch18} and \cite{Fano}, the proof is by induction on $\overline{\Delta}_H(E)$. The base of induction is provided by \cite[Theorem 3.4]{Vass} and Lemmas \ref{2 -1 -3/10}--\ref{2 0 -9/5}.
	
	\begin{enumerate}
		\item Assume that $c=-1,d\le-\frac{13}{10}$ and $e\ge\frac{d^2}{2}-d+\frac{53}{600}$. As we have shown in Lemma \ref{0 1 2}, the numerical wall $W_Q(E)$ is non-empty, and $E$ is destabilized at a semicircular wall by a subobject or quotient $F$ of rank $s\in\{0,1,2\}$ and $\overline{\Delta}_H(F)<\overline{\Delta}_H(E)$. The case $s=1$ is considered in Lemma \ref{one}, it suffices to consider now the case $s=2$. We find $Q_{0,-\frac 32}(E)<0$ and $Q_{0,-2}(E)<0$, hence $W(E,F)$ intersects rays $\beta=-\frac 32$ and $\beta=-2$. Conditions $\mathrm{ch}_1^{-\frac 32}(F)>0$ and $\mathrm{ch}_1^{-2}(F)<\mathrm{ch}_1^{-2}(E)$ imply that $\mathrm{ch}_1(F)=-2H$. So, we get $\mathrm{ch}(F)=(2+yH^2+zH^3)\cdot\mathrm{ch}(\mathcal O_X(-1))$ with $y\le 0$. By induction we have $z\le\frac{y^2}{2}-\frac{y}{10}+\frac{2}{25}$. Assume that $y\le-\frac 15$ to get a contradiction. We have
		$$s_Q(E)=\frac{d+6e}{4d-1}\le\frac{3d^2-5d+\frac{53}{100}}{4d-1},$$
		$$s(E,F)=d-y-1.$$
		Since $s(E,F)\le s_Q(E)$, we obtain
		$$y\ge\frac{d^2+\frac{47}{100}}{4d-1}>\frac d2-\frac{9}{20}.$$
		Applying Corollary \ref{0 1} to $E/F$ and using the bound on $z$ from above, we get a bound
		$$e\le y^2+\left(-d+\frac{9}{10}\right)y+\frac{d^2}{2}-d+\frac{29}{50}.$$
		On the right-hand side we have a quadratic function of $y$ with a minimum at $y_0=\frac d2-\frac{9}{20}$. Therefore, the maximum is attained at $y=-\frac 15$, where we get $e\le\frac{d^2}{2}-\frac 45d+\frac{11}{25}$. For $d\le-\frac{19}{10}$ this bound contradicts the assumption $e\ge\frac{d^2}{2}-d+\frac{53}{600}$, hence $y=0$ in these cases. We need to consider separately the cases $d\in\{-\frac{13}{10},-\frac 32,-\frac{17}{10}\}$. For $y\le-\frac 25$ we analogously get a bound $e\le\frac{d^2}{2}-\frac 35d+\frac{19}{50}<\frac{d^2}{2}-d+\frac{53}{600}$, so we only need to exclude the case $y=-\frac 15$. In this case \cite[Lemma 3.7]{Vass} implies that the maximal value of $\mathrm{ch}_3(F)$ equals $-\frac{H^3}{3}$. Let $d=-\frac{13}{10}$, then for $y=0$ the maximal value of $e$ equals $\frac{67}{30}$, while for $y=-\frac 15$ we have $e\le\frac{11}{6}<\frac{67}{30}$. Respectively, for $d=-\frac 32$ and $y=0$ the maximal $e$ is $\frac{17}{6}$, while for $y=-\frac 15$ $e\le\frac{67}{30}<\frac{17}{6}$. And for $d=-\frac{17}{10}$ the maximal $e$ is $\frac{97}{30}$, while for $y=-\frac 15$ $e\le\frac{17}{6}<\frac{97}{30}$. We have proved that $e$ is maximal for $y=0$, hence for $F\cong\mathcal O_X(-1)^{\oplus 2}$. Corollary \ref{0 1} in this case gives the claimed bound $e\le\frac{d^2}{2}-d+\frac{5}{24}-\varepsilon(d)$.
		
		We need to check that $F\cong\mathcal O_X(-1)^{\oplus 2}$ is indeed a subobject of $E$. We have proved that $W(E,\mathcal O_X(-1))$ is the unique wall for $E$. Since $E$ is not tilt-semistable below $W_Q(E)$, it cannot be tilt-semistable below $W(E,\mathcal O_X(-1))$. The only subobjects that can destabilize $E$ there are $\mathcal O_X(-1)$ or $\mathcal O_X(-1)^{\oplus 2}$. Therefore, $\mathcal O_X(-1)$ is a subobject of $E$. As in the proof of \cite[Proposition 3.2]{MS18}, we get a monomorphism $\mathcal O_X(-1)\hookrightarrow E/(\mathcal O_X(-1))$, which allows us to obtain a monomorphism $\mathcal O_X(-1)^{\oplus 2}\hookrightarrow E$. 
		
		The object $G:=E/(\mathcal O_X(-1)^{\oplus 2})$ is described by Lemmas \ref{0 1 -1/2}--\ref{0 1 -3/10}. One can check that $W(E,\mathcal O_X(-1))$ is located above the unique wall for $G$, so Lemmas \ref{0 1 -1/2}--\ref{0 1 -3/10} give the description of $G$ from the formulation of the present theorem.
		
		\item Assume that $c=0,d\le-2$ and $e\ge\frac{d^2}{2}-\frac{d}{10}-\frac{1}{25}$. Lemmas \ref{0 1 2} and \ref{one} imply that $E$ is destabilized at a semicircular wall, induced by a subobject or quotient $F$ of rank two. Since $Q_{0,-1}(E)<0$ and $\mathrm{ch}_1^{-1}(E)=2H$, we should have $\mathrm{ch}_1^{-1}(F)=H$, that is, $\mathrm{ch}_1(F)=-H$. Then $\mathrm{ch}(F)=2-H+yH^2+zH^3$ for some $y\le\frac{1}{10}$. Suppose that $y\le-\frac{1}{10}$. By induction assumption we have $z\le\frac{y^2}{2}-y+\frac{5}{24}$. Also we have
		$$s_Q(E)=\frac{3e}{2d}\ge\frac 34d-\frac{3}{20}-\frac{3}{50d},$$
		$$s(E,F)=d-y.$$
		Since $s(E,F)\le s_Q(E)$, we get
		$$y\ge\frac d4+\frac{3}{20}+\frac{3}{50d}>\frac d2+\frac 12.$$
		We apply Corollary \ref{0 1} to $E/F$ and obtain a bound
		\begin{equation}\label{e bound}e\le y^2+(-d-1)y+\frac{d^2}{2}+\frac 14.\end{equation}
		On the right-hand side we have a quadratic function of $y$ with a minimum at $y_0=\frac d2+\frac 12$, hence the maximum is attained at $y=-\frac{1}{10}$, where we get $e\le\frac{d^2}{2}+\frac{d}{10}+\frac{9}{25}$. For $d<-2$ this bound contradicts our assumption $e\ge\frac{d^2}{2}-\frac{d}{10}-\frac{1}{25}$, so we should have $y=\frac{1}{10}$. For $d=-2$ and $y=\frac{1}{10}$ the maximal value of $e$ (by Corollary \ref{0 1} and \cite[Lemma 3.6]{Vass}) equals $\frac{11}{5}$, while for $y=-\frac{1}{10}$ we get $e\le 2<\frac{11}{5}$ (by Corollary \ref{0 1} and \cite[Lemma 3.8]{Vass})) and for $y\le-\frac{3}{10}$ the bound (\ref{e bound}) gives $e\le\frac{51}{25}<\frac{11}{5}$. We proved that $e$ is maximal when $y=\frac{1}{10}$, hence when $F\cong\mathcal U$. In this case Corollary \ref{0 1} gives the bound $e\le\frac{d^2}{2}-\frac{d}{10}+\frac{2}{25}-\varepsilon(d)$. Again $E$ cannot be tilt-semistable below the unique wall $W(E,\mathcal U)$, hence $\mathcal U$ is indeed a subobject. The object $G:=E/\mathcal U$ is described by Lemmas \ref{0 1 -1/2}--\ref{0 1 -3/10}. This object $G$ is tilt-semistable along the wall $W(E,\mathcal U)$, which is located above the unique wall for $G$ (in fact, the unique wall for $G$ does not intersect the ray $\beta=-1$). \qed
	\end{enumerate} 
	
	\textit{Proof of Theorem \ref{intro}.} By Proposition \ref{2-stability} any Gieseker semistable rank two sheaf $E$ corresponds to an object, which is $\nu_{\alpha,\beta}$-semistable for $\beta<\mu(E)<+\infty,\alpha\gg0$. Hence, we can apply Theorem \ref{main}. Chern classes of $E$ are related to the Chern character by the formula
	\begin{equation}\label{ch}\mathrm{ch}(E)=\mathrm{rk}(E)+c_1H+\left(\frac{c_1^2}{2}-\frac{c_2}{5}\right)H^2+\left(\frac{c_1^3}{6}-\frac{c_1c_2}{10}+\frac{c_3}{10}\right)H^3.\end{equation}
	In those cases with maximal $c_3$, in which exact triples in $\mathrm{Coh}^\beta(X)$, in which $E$ is included, contain a sheaf shifted by 1, we rotate the corresponding distinguished triangle in order to obtain an exact triple in $\mathrm{Coh}(X)$. Also, in cases (2.6) and (2.7) we exclude exact triples defining objects $E$, which are tilt-semistable only below a semicircular wall. Indeed, these objects $E$ are rank two sheaves with torsion, which are not Gieseker semistable. \qed
	
	\section{Moduli Spaces of Semistable Sheaves}
	\subsection{Special cases.}\label{special} Now we deal with the question of existence of semistable sheaves with maximal $c_3$ from Theorem \ref{intro}. Firstly we consider special cases of low $c_2$. Let us prove a preparatory lemma.
	
	\begin{lemma}\label{reflexive}Suppose that $F$ and $G$ are locally free sheaves on $X$, such that the sheaf $\mathcal Hom(F,G)$ is globally generated and $\mathrm{rk}(G)=\mathrm{rk}(F)+2$, then there exists a monomorphism $f:F\to G$ such that the sheaf $E:=\coker(f)$ is reflexive.
	\end{lemma}
	
	\begin{proof}
	We firstly prove that there exist an $f:F\to G$ such that the singularity set $\mathrm{Sing}(f):=\{x\in X\ |\ f_x\ \text{is not injective}\}$ is either empty or zero-dimensional. Consider an incidence variety $Z\subset X\times\Hom(F,G)$ consisting of those pairs $(x,f)$ for which $x\in\mathrm{Sing}(f)$, with canonical projections $p:Z\to X,q:Z\to\Hom(F,G)$. Take an $x\in X$, then the fiber $p^{-1}(x)$ consists of those $f:F\to G$ for which the map on fibers $f_x:F_x\to G_x$ is not injective. We can consider a linear map $\pi_x:\Hom(F,G)\to\Hom(F_x,G_x)$, which is surjective, since $\mathcal Hom(F,G)$ is globally generated. In the linear space $\Hom(F_x,G_x)$ the codimension of the subvariety of non-injective maps is equal to 3 (see e.~g. \cite{Hirschowitz}), hence $\dim p^{-1}(x)=\dim \Hom(F,G)-3$, and $\dim Z=\dim \Hom(F,G)$. Therefore, for a general $f\in\Hom(F,G)$ we have $\dim q^{-1}(f)\le 0$, which is what we need. In particular, such an $f$ is a monomorphism (otherwise $\ker f$ would be a sheaf of positive rank and $\dim q^{-1}(f)$ would be equal to 3). Now, dualizing an exact triple
	\begin{equation}\label{F G E}0\to F\overset{f}{\to} G\to E\to 0,\end{equation} 
	we obtain an exact sequence of the form 
	\begin{equation}\label{dual}0\to E^\vee\overset{g}{\to} G^\vee\to F^\vee\to T\to 0,\end{equation} 
	where $T$ is a sheaf with at most zero-dimensional support. Denoting $F':=\coker g$ (so that (\ref{dual}) splits into two exact triples) and using that $\mathcal Ext^i(T,\mathcal O_X)=0$ for $i\le 2$, we obtain $F'^\vee\cong F^{\vee\vee},\mathcal Ext^1(F',\mathcal O_X)=0$, hence there is an exact triple
	\begin{equation}\label{E vee vee}0\to F\to G\to E^{\vee\vee}\to 0.\end{equation}
	Moreover, the morphisms in (\ref{E vee vee}) are double duals of morphisms in (\ref{F G E}), so the natural transformation $\id\to(-)^{\vee\vee}$ induces a commutative diagram with rows (\ref{F G E}) and (\ref{E vee vee}), and reflexivity of $E$ follows from the Snake Lemma.
	\end{proof}
	
	\begin{proposition}\label{existence}There exist rank two Gieseker semistable sheaves $E$ with maximal $c_3(E)$ permitted by Theorem \ref{intro} in special cases (1.1), (1.2), (1.4) and (2.1)--(2.7).
	\end{proposition}
	
	\begin{proof}In the case (1.1) the vector bundle $\mathcal U$ is slope stable, since $h^0(\mathcal U)=0$ (arguments from \cite[Lemma 1.2.5]{OSS} work for $X=X_5$ too).
		
	In the case (1.2) we have an exact sequence
	\begin{equation}\label{1.2}0\to\mathcal U(-1)\to\mathcal O_X(-1)^{\oplus 4}\to E\to 0.\end{equation}
	Since a sheaf $\mathcal Hom(\mathcal U(-1),\mathcal O_X(-1)^{\oplus 4})\cong(\mathcal U(1))^{\oplus 4}$ is globally generated (this can be shown by dualizing the exact triple (\ref{tautological})), by Lemma \ref{reflexive} a general $E$ from (\ref{1.2}) is reflexive. Since $h^0(E)=0$, such an $E$ is slope stable.
	
	In the case (1.4) we have an exact sequence
	$$0\to\mathcal O_X(-2)\to\mathcal O_X(-1)^{\oplus 3}\to E\to 0.$$
	Again by Lemma \ref{reflexive} a general such $E$ is reflexive, and $h^0(E)=0$, so $E$ is slope stable.
	
	In the case (2.1) the sheaf $\mathcal O_X^{\oplus 2}$ is of course Gieseker semistable as a direct sum of two isomorphic stable sheaves.
	
	In the case (2.2) the sheaf $E$ is uniquely determined up to isomorphism as the kernel of the composition $\mathcal O_X^{\oplus 2}\twoheadrightarrow\mathcal O_X^{\oplus 2}|_L\twoheadrightarrow\mathcal O_L(1)$. It was proven in \cite[Lemma 3.7]{Vass} that $E$ is $\nu_{\alpha,\beta}$-stable for all $\beta<0,\alpha>0$, hence $E$ is a 2-stable sheaf, and therefore it is Gieseker stable.
	
	In cases (2.3) and (2.4) the existence of slope stable sheaves $E$ is argued for in \cite[Remark 3.13]{Vass} and in Remark \ref{instantons} of the present work, respectively.
	
	In the case (2.5) sheaves $E$ are included into exact triples of the form
	$$0\to\mathcal O_X(-1)^{\oplus 2}\to\mathcal U^{\oplus 2}\to E\to 0.$$
	By Lemma \ref{reflexive} a general such $E$ is reflexive, and hence slope stable, since $h^0(E)=0$.	
	
	In the case (2.6) we can take those sheaves $E$, which are included into exact triples
	$$0\to\mathcal U(-1)^{\oplus 2}\to\mathcal O_X(-1)^{\oplus 6}\to E\to 0.$$
	Respectively, in the case (2.7) we can take sheaves $E$, which are included into exact triples
	$$0\to\mathcal O_X(-2)\to\mathcal Q(-1)\to E\to 0.$$
	In both cases we can apply Lemma \ref{reflexive} (using the exact triple (\ref{tautological})) to show that general such $E$ are reflexive, hence slope stable, because $h^0(E)=0$.	
\end{proof}

In the general case the existence of Gieseker semistable sheaves $E$ from Theorem \ref{intro} follows from the classification of subobjects and quotients that can destabilize $E$ in tilt-stability. For completeness we also include such classification for objects from Lemmas \ref{2 -1 -1/2} and \ref{2 -1 -7/10}.

\begin{lemma}\label{Jordan-Holder}Let $E\in\mathrm{Coh}^\beta(X)$ be a tilt-semistable object, which can be included into one of exact triples
	\begin{equation}\label{third}0\to \mathcal O_X(-1)^{\oplus 3}\to E\to\mathcal O_X(-2)[1]\to 0,\end{equation}
		\begin{equation}\label{first}0\to \mathcal O_X(-1)^{\oplus 2}\to E\to\mathcal I_{L,H}(-1)\to 0,\end{equation}
		\begin{equation}\label{second}0\to \mathcal O_X(-1)^{\oplus 2}\to E\to\mathcal I_{C,H}(-1)\to 0.\end{equation}
If $E$ is destabilized at a semicircular wall by a subobject or quotient $F$ with $\mu(F)<\mu(E)$, then $F$ is isomorphic to $\mathcal O_X(-1)$ or $\mathcal O_X(-1)^{\oplus 2}$, or $\mathcal O_X(-1)^{\oplus 3}$ in the case (\ref{third}).
\end{lemma}

\begin{proof}
	Let us consider the case of the exact triple (\ref{third}). The proof of Lemma \ref{2 -1 -1/2} implies that the wall $W(E,\mathcal O_X(-1))=W(\mathcal O_X(-1),\mathcal O_X(-2)[1])$ is the unique semicircular wall for $E$. Therefore, $W(E,F)$ coincides with this wall. By \cite[Theorem 2.9 (iii)]{Sch18} this means that $\mathrm{ch}_{\le 2}(F),\mathrm{ch}_{\le 2}(\mathcal O_X(-1))$ and $\mathrm{ch}_{\le 2}(\mathcal O_X(-2)[1])$ are linearly dependent. Hence, we have $\mathrm{ch}_{\le 2}(F)=a\cdot\mathrm{ch}_{\le 2}(\mathcal O_X(-1))+b\cdot\mathrm{ch}_{\le 2}(\mathcal O_X(-2)[1])=a-b+(2b-a)H+(\frac a2-2b)H^2$ for some $a,b\in\mathbb R$. The condition $\mathrm{ch}_{\le 2}(F)\in\Lambda$ implies that $a,b\in\mathbb Z$. We find $\overline{\Delta}_H(F)=25ab,\overline{\Delta}_H(E/F)=25(3-a)(1-b)$. Suppose that $a\le -1$, then the condition $\overline{\Delta}_H(F)\ge 0$ implies that $b\le 0$. But in this case we have $\overline{\Delta}_H(E/F)\ge 100>\overline{\Delta}_H(E)=75$, contradicting Proposition \ref{actual} (2). So, $a\ge 0$, and analogously one shows that $b\ge 0$. Switching the roles of $F$ and $E/F$, we get that $a\le 3,b\le 1$, hence $a\in\{0,1,2,3\},b\in\{0,1\}$. Among these pairs $(a,b)$ only $(1,0),(2,0)$ and $(3,0)$ satisfy $\mu(F)<\mu(E)$. Maximizing $\mathrm{ch}_3(F)$ for these $\mathrm{ch}_{\le 2}(F)$ (using \cite[Proposition 3.1]{Vass}) we get that $F\cong\mathcal O_X(-1)$, or $F\cong\mathcal O_X(-1)^{\oplus 2}$, or $F\cong\mathcal O_X(-1)^{\oplus 3}$.
	
	Now we give an argument for the exact triple (\ref{first}). The proof of Lemma \ref{2 -1 -7/10} implies that the wall $W(E,\mathcal O_X(-1))=W(\mathcal O_X(-1),\mathcal I_{L,H}(-1))$ is the unique semicircular wall for $E$. Therefore, $W(E,F)$ coincides with this wall. Again this means that $\mathrm{ch}_{\le 2}(F)=a\cdot\mathrm{ch}_{\le 2}(\mathcal O_X(-1))+b\cdot\mathrm{ch}_{\le 2}(\mathcal I_{L,H}(-1))=a+(b-a)H+(\frac a2-\frac{17}{10}b)H^2$ for some $a,b\in\mathbb Z$. Firstly, we have a condition
		\begin{equation}\label{F1}\overline{\Delta}_H(F)=b(35a+25b)\ge 0.
		\end{equation}
		Secondly,
		\begin{equation}\label{E/F}\overline{\Delta}_H(E/F)=b(35a+25b)-35a-120b+95\ge 0.
		\end{equation}
		And also $\overline{\Delta}_H(F)+\overline{\Delta}_H(E/F)<\overline{\Delta}_H(E)=95.$ Using (\ref{F1}) and (\ref{E/F}) we see that this condition implies
		\begin{equation}\label{F+E/F}35a+120b>0.
		\end{equation}
	Now, if $b<0$, then by (\ref{F1}) $35a+25b\le 0$, hence $35a+120b<0$ and we get a contradiction with (\ref{F+E/F}). Therefore, $b\ge 0$. Switching the roles of $F$ and $E/F$ we analogously get that $b\le 1$. If $b=0$, then (\ref{F+E/F}) implies that $a>0$, while (\ref{E/F}) gives $a\le 2$. And if $b=1$, then we again switch the roles of $F$ and $E/F$ and get $0\le a\le 2$. Since $\mathrm{ch}_{\le 2}(F)$ and $\mathrm{ch}_{\le 2}(E/F)$ have to be nonzero in order to induce a semicircular wall, we have $(a,b)\in\{(1,0),(2,0),(0,1),(1,1)\}$. Among these pairs only $(1,0)$ and $(2,0)$ satisfy $\mu(F)<\mu(E)$. Maximizing $\mathrm{ch}_3(F)$ for these $\mathrm{ch}_{\le 2}(F)$ we get that $F\cong\mathcal O_X(-1)$ or $F\cong\mathcal O_X(-1)^{\oplus 2}$. Finally, the argument for the exact triple (\ref{second}) is analogous to that of (\ref{first}).
\end{proof}

Now we can prove that the bound from Theorem \ref{intro} is exact in the most nontrivial case (1.3).

\begin{proposition}\label{existence2}In the special case (1.3) of Theorem \ref{intro} there exist Gieseker semistable sheaves $E$ with maximal $c_3(E)$ permitted by this theorem.
\end{proposition}

\begin{proof}Recall that the sheaves $E$ should be included into an exact sequence of the form
\begin{equation}\label{sequence3}0\to\mathcal O_X(-2)\overset{f}{\to}\mathcal O_X(-1)^{\oplus 3}\to E\to\mathcal O_L(-2)\to 0.\end{equation}
Denote by $F$ the cokernel of the morphism $f$ from (\ref{sequence3}), so that we have an exact triple
	\begin{equation}\label{F L}0\to F\to E\to \mathcal O_L(-2)\to 0.\end{equation}
	If three morphisms $\mathcal O_X(-2)\to\mathcal O_X(-1)$ defining $f$ are linearly dependent, then $F$ splits as a direct sum $\mathcal O_X(-1)\oplus F'$ with $\mu(F')=0$, hence $E$ is not 2-semistable (since $\mu(E)=-\frac 12$). Therefore, we assume that $f$ corresponds to three linearly independent morphisms $\mathcal O_X(-2)\to\mathcal O_X(-1)$. In this case $F$ is tilt-semistable everywhere above $W(F,\mathcal O_X(-1))$, hence it is 2-semistable, and hence slope stable ($\rk(F)$ and $c_1(F)$ are coprime). Indeed, if $F$ were not tilt-semistable above $W(F,\mathcal O_X(-1))$, then Lemma \ref{Jordan-Holder} would imply that there is a nonzero morphism $F\to\mathcal O_X(-1)$. However, in this case three components of $f$ are linearly dependent, as can be shown using the long exact sequence of Ext groups.
	
	Note that $\mathrm{Ext}^1(\mathcal O_L(-2),\mathcal O_X(-1)^{\oplus 3})\cong H^0(\mathcal Ext^1(\mathcal O_L(-2),\mathcal O_X(-1)^{\oplus 3}))$, and this $\mathcal Ext^1$ sheaf vanishes, since the codimension of the support of $\mathcal O_L(-2)$ is bigger than 1 and $\mathcal O_X(-1)^{\oplus 3}$ is locally free. Therefore, we have an exact sequence
	\begin{equation}\label{Ext sequence}0\to\mathrm{Ext}^1(\mathcal O_L(-2),F)\to\mathrm{Ext}^2(\mathcal O_L(-2),\mathcal O_X(-2))\overset{g}{\to}\mathrm{Ext}^2(\mathcal O_L(-2),\mathcal O_X(-1)^{\oplus 3}).\end{equation}
	The spectral sequence of local and global Ext implies that the morphism $g$ from (\ref{Ext sequence}) corresponds to a morphism $H^0(\mathcal Ext^2(\mathcal O_L(-2),\mathcal O_X(-2)))\to H^0(\mathcal Ext^2(\mathcal O_L(-2),\mathcal O_X(-1)^{\oplus 3}))$ induced by $f$. By the arguments from the proof of Lemma \ref{2 -1 -3/10} $\mathcal Ext^2(\mathcal O_L(-2),\mathcal O_X(-2))\cong\mathcal O_L$ and \linebreak$\mathcal Ext^2(\mathcal O_L(-2),\mathcal O_X(-1)^{\oplus 3})\cong\mathcal O_L(1)^{\oplus 3}$. 
	If the restriction of the morphism $f$ to $L$ vanishes, we get that $\mathrm{Ext}^1(\mathcal O_L(-2),F)\neq 0$. In this case there exists a nontrivial extension (\ref{F L}), and the sheaf $E$ from it is slope stable. Indeed, if there is a subsheaf $0\neq G\subset E$ with $\mu(G)\ge\mu(E),\rk(G)<2$, then $G$ would be included into an exact triple $0\to G'\to G\to G''\to 0,G'\subset F,G''\subset\mathcal O_L(-2)$. Since $F$ is slope stable and $\mathrm{ch}_{\le 1}(\mathcal O_L(-2))=0$, only the case $G'=0,G''\cong\mathcal O_L(a),a\le -2$ remains possible. Since the extension (\ref{F L}) is nontrivial, we get $a<-2$. Existence of a lift $\mathcal O_L(a)\to E$ and non-triviality of the extension gives $\mathrm{Ext}^1(\mathcal O_L(-2)/\mathcal O_L(a),E)\neq 0$, which contradicts the fact that $E$ has maximal $\mathrm{ch}_3$ among tilt-semistable objects with the same $\mathrm{ch}_{\le 2}$.
\end{proof}

\subsection{Moduli spaces of torsion sheaves.}\label{torsion}
In this section we prove Theorem \ref{torsion thm} about moduli spaces of torsion sheaves $\mathcal I_{C,H},\mathcal I_{L,H},\mathcal O_H,\mathbb D(\mathcal I_{L,H})$ and $\mathbb D(\mathcal I_{C,H})$, as a preparation for construction of moduli spaces of rank two sheaves on $X$ with maximal $c_3$ and $c_2$ big enough.

\textit{Proof of Theorem \ref{torsion thm}.}
Denote by $p_i:X\times B_i\to X,q_i:X\times B_i\to B_i$ the canonical projections.

\begin{enumerate}
\item  
By Lemma \ref{0 1 -9/10} sheaves $\mathcal I_{C,H}$ can be included into exact triples
\begin{equation}\label{triple26}0\to\mathcal O_X(-1)^{\oplus 2}\overset{f_1}{\to}\mathcal U\to \mathcal I_{C,H}\to 0.\end{equation}
The morphism $f_1$ in (\ref{triple26}) is given by two elements of $\Hom(\mathcal O_X(-1),\mathcal U)\cong\mathbb C^5$. Since $f_1$ is a monomorphism, these elements have to be linearly independent. Note that the proof of Lemma \ref{0 1 -9/10} implies that the cokernel of any monomorphism $\mathcal O_X(-1)^{\oplus 2}\hookrightarrow \mathcal U$ is isomorphic to $\mathcal I_{C,H}$ for some conic $C\subset H$. Conversely, if $f_1:\mathcal O_X(-1)^{\oplus 2}\to\mathcal U$ is not a monomorphism, then $\ker f_1\subset\mathcal O_X(-1)^{\oplus 2}$ has to be reflexive, hence it is a line bundle $\mathcal O_X(a)$ for some $a\le -1$ (we assume that $f_1\neq 0$). But if $a\le -2$, then $\mathcal O_X(-1)^{\oplus 2}/\mathcal O_X(a)$ cannot be a subsheaf of $\mathcal U$ by stability of $\mathcal U$. Hence, $a=-1$ and two corresponding elements of $\mathbb C^5$ are linearly dependent.

Suppose that there is an isomorphism $i:\mathcal I_{C,H}\to\mathcal I_{C',H'}$, where $\mathcal I_{C',H'}$ is the cokernel of a monomorphism $f_1':O_X(-1)^{\oplus 2}\to\mathcal U$. Then (\ref{triple26}) and the vanishing of $\mathrm{Ext}^1(\mathcal U,\mathcal O_X(-1)^{\oplus 2})$ implies that $i$ is induced by a (nonzero) element of $\mathrm{End}\ \mathcal U\cong\mathbb C$. Therefore, morphisms $f_1$ and $f_1'$ are conjugate with respect to the natural action of $\Aut \mathcal O_X(-1)^{\oplus 2}\times\Aut \mathcal U\cong\mathrm{GL}(2)\times\mathbb C^*$. It follows that isomorphism classes of sheaves $\mathcal I_{C,H}$ correspond bijectively to two-dimensional subspaces of $\mathbb C^5$. 

Now we can construct a family of the sheaves $\mathcal I_{C,H}$ that we need.  
On $\mathrm{Gr}(2,5)=\mathrm{Gr}(2,\Hom(\mathcal O_X(-1),\mathcal U))$ there is a tautological rank 2 subbundle $\widetilde{\mathcal U}$ with a canonical morphism $u:\widetilde{\mathcal U}\to\Hom(\mathcal O_X(-1),\mathcal U)\otimes\mathcal O_{\mathrm{Gr}(2,5)}$. Also, there is a canonical evaluation morphism $e_1:\Hom(\mathcal O_X(-1),\mathcal U)\otimes\mathcal O_X(-1)\to\mathcal U$. Consider the composition
\begin{equation}\label{comp}p_1^*(\mathcal O_X(-1))\otimes q_1^*(\widetilde{\mathcal U})\overset{\id\otimes q_1^*(u)}{\to}p_1^*(\mathcal O_X(-1))\otimes\Hom(\mathcal O_X(-1),\mathcal U)\overset{p_1^*(e_1)}{\to}p_1^*(\mathcal U).
\end{equation}
Define $\mathbb I_1$ as the cokernel of composition (\ref{comp}). The family $\mathbb I_1$ induces a bijective modular morphism $\Phi_1:\mathrm{Gr}(2,5)\to\mathcal M_X(v)$ for $v=\mathrm{ch}(\mathcal I_{C,H})$ (since restrictions of (\ref{comp}) to fibers $X\times\{p\},p\in\mathrm{Gr}(2,5)$ bijectively correspond to exact triples (\ref{triple26}) up to isomorphism). In order to conclude, it suffices to prove that $\mathcal M_X(v)$ is smooth. Indeed, since we work in characteristic 0, the bijective map $\Phi_1$ is birational, and if it not biregular, then we get a contradiction with \cite[Ch. 2, \S 4.4, Theorem 2.16]{Sh}.

Note that $\dim \mathcal M_X(v)=6$, because $\Phi_1$ is bijective. Since $\mathcal I_{C,H}$ is Gieseker stable (being a pure sheaf of rank 1 on $H$), it suffices to prove that $\mathrm{ext}^1(\mathcal I_{C,H},\mathcal I_{C,H})=6$ for any $[\mathcal I_{C,H}]\in\mathcal M_X(v)$. This follows from a direct calculation using (\ref{triple15}).

\item 
By the proof of Lemma \ref{0 1 -7/10} there are exact triples
\begin{equation}\label{triple27}0\to\mathcal O_X(-1)\overset{f_2}{\to}\mathcal I_L\to \mathcal I_{L,H}\to 0,\end{equation}
\begin{equation}\label{triple28}0\to\mathcal U\overset{g}{\to}\mathcal Q^\vee\to \mathcal I_L\to 0.\end{equation}
Since $\mathcal I_L$ is torsion-free, any nonzero $f_2:\mathcal O_X(-1)\to\mathcal I_L$ is a monomorphism. The argument in Lemma \ref{0 1 -7/10} implies that the cokernel of such $f_2$ is isomorphic to $\mathcal I_{L,H}$ for some $L\subset H\subset X$. Moreover, by \cite[Lemma 4.2]{K} any nonzero $g:\mathcal U\to\mathcal Q^\vee$ is a monomorphism and its cokernel is isomorphic to $\mathcal I_L$ for some line $L\subset X$. The Fano scheme of lines on $X$ is isomorphic to $\mathbb P^2=\mathbb P(\Hom(\mathcal U,\mathcal Q^\vee))$.

Suppose that there is an isomorphism $i:\mathcal I_{L,H}\to\mathcal I_{L',H'}$, where $\mathcal I_{L',H'}$ is the cokernel of a monomorphism $f_2':\mathcal O_X(-1)\to\mathcal I_L$. The exact triple (\ref{triple27}) and the vanishing of $\mathrm{Ext}^1(\mathcal I_L,\mathcal O_X(-1))$ imply that $i$ is induced by a (nonzero) morphism $j:\mathcal I_L\to\mathcal I_{L'}$. Since $\mathcal I_L$ and $\mathcal I_{L'}$ are stable sheaves with equal Hilbert polynomials, $j$ is an isomorphism. Therefore, $f_2$ and $f_2'$ are conjugate by an action of $\Aut \mathcal O_X(-1)\times\Aut\mathcal I_L$, that is, they are proportional. It follows that giving the isomorphism class of $\mathcal I_{L,H}$ is equivalent to giving a point of $\mathbb P^2=\mathbb P(\Hom(\mathcal U,\mathcal Q^\vee))$ corresponding to some line $L$ and a point of $\mathbb P(H^0(\mathcal I_L(1)))$.

Consider the product $X\times\mathbb P(\Hom(\mathcal U,\mathcal Q^\vee))$ with projections $p_l,q_l$ on the first and on the second factor, respectively. There are a canonical morphism $i_1:\mathcal O_{\mathbb P^2}(-1)\to\Hom(\mathcal U,\mathcal Q^\vee)\otimes\mathcal O_{\mathbb P^2}$ (from the Euler exact sequence) and an evaluation morphism $e_2:\Hom(\mathcal U,\mathcal Q^\vee)\otimes\mathcal U\to\mathcal Q^\vee$. Define a sheaf $\mathbb I_l$ as the cokernel of the composition
\begin{equation}\label{I_l}p_l^*(\mathcal U)\otimes q_l^*(\mathcal O_X(-1))\overset{\id\otimes q_l^*(i_1)}{\to}p_l^*(\mathcal U)\otimes\Hom(\mathcal O_X(-1),\mathcal U)\overset{p_l^*(e_2)}{\to}p_l^*(\mathcal Q^\vee).
\end{equation}
The sheaf $\mathbb I_l$ is a (flat) family of ideal sheaves $\mathcal I_L$. By (\ref{triple28}) we find that $h^0(\mathcal I_L(1))=5$ for all $L$. Proper base change theorem implies that $\mathcal F:=q_{l,*}(\mathbb I_l\otimes p_l^*(\mathcal O_X(1)))$ is rank 5 vector bundle on $\mathbb P^2$. Define 
\begin{equation}\label{bbF}\mathbb F:=\mathbb P(\mathcal F^\vee).\end{equation} 
Points of $\mathbb F$ correspond bijectively to pairs $(L,H)$ of a line $L\subset X$ and a hyperplane section $H\subset X$ containing $L$.
 
Denote by $\pi:\mathbb F\to\mathbb P^2$ the canonical projection and by $\widetilde{\pi}$ the map $(\id,\pi):X\times\mathbb F\to X\times\mathbb P^2$. On $\mathbb F=\mathbb P(\mathcal F^\vee)$ there is a (dualized) Grothendieck sheaf $\mathcal O_{\mathbb F}(-1)$ with a canonical map $i_2:\mathcal O_{\mathbb F}(-1)\to\pi^*(\mathcal F)$. Denote by $e_3$ the evaluation morphism ${q_l}^*q_{l,*}(\mathbb I_l\otimes p_l^*(\mathcal O_X(1))\to\mathbb I_l\otimes p_l^*(\mathcal O_X(1))$. Consider the composition
\begin{multline}\label{comp3} p_2^*\mathcal O_X(-1)\otimes q_2^*\mathcal O_{\mathbb F}(-1)\overset{\id\otimes q_2^*(i_2)}{\to}p_2^*\mathcal O_X(-1)\otimes q_2^*\pi^*(\mathcal F)\overset{\id\otimes\widetilde{\pi}^*(e_3)}{\to}\\
\overset{\id\otimes\widetilde{\pi}^*(e_3)}{\to}p_2^*\mathcal O_X(-1)\otimes\widetilde{\pi}^*(\mathbb I_l\otimes p_l^*\mathcal O_X(1))\cong\widetilde{\pi}^*(\mathbb I_l).\end{multline}
Denote by $\mathbb I_2$ the cokernel of the composition (\ref{comp3}). The family $\mathbb I_2$ induces a bijective modular morphism $\Phi_2:\mathbb F\to\mathcal M_X(\mathrm{ch}(\mathcal I_{L,H}))$, which is an isomorphism, since $\mathrm{ext}^1(\mathcal I_{L,H},\mathcal I_{L,H})=6$ (in view of (\ref{triple4.1})).

\item 
In this case we can consider the family $\mathbb I_3$ of sheaves $\mathcal O_H$ with the base $\mathbb P^6=\mathbb P(H^0(\mathcal O_X(1)))$, which induces an isomorphism $\Phi_3:\mathbb P^6\to\mathcal M_X(\mathrm{ch}(\mathcal O_H))$ (again we have $\mathrm{ext}^1(\mathcal O_H,\mathcal O_H)=6$).

\item 
By Lemma \ref{0 1 -3/10} sheaves $\mathbb D(\mathcal I_{L,H})$ can be included into exact triples
\begin{equation}\label{triple29}0\to\mathcal Q\to\mathcal O_X(1)\oplus\mathcal U(1)\to \mathbb D(\mathcal I_{L,H})\to 0.
\end{equation}
By \cite[proof of Lemma 4.2]{K} for any nonzero $a:\mathcal Q\to\mathcal U(1)$ there is an exact sequence of the form
\begin{equation}\label{sequence1}0\to\mathcal O_X\to\mathcal Q\overset{a}{\to}\mathcal U(1)\to\mathcal O_L\to 0.
\end{equation}
Fix a monomorphism $f:\mathcal O_X(1)\hookrightarrow\mathcal O_X(1)\oplus\mathcal U(1)$ (it is unique up to proportionality). We obtain a commutative diagram
$$\label{diagram}
\xymatrix{
		0 \ar[r] & 0 \ar[r]\ar[d] & \mathcal O_X(1)\ar@{=}[r]\ar@{^{(}->}[d]^f & \mathcal O_X(1)\ar[r]\ar[d]^g & 0 \\
		0\ar[r] & \mathcal Q\ar[r] & \mathcal O_X(1)\oplus\mathcal U(1)\ar[r] & \mathbb D(\mathcal I_{L,H})\ar[r] & 0 }
$$
Snake Lemma and the exact sequence (\ref{sequence1}) imply that $\ker g\cong\mathcal O_X,\coker g\cong\mathcal O_{L'}$. Therefore, we get an exact triple
\begin{equation}\label{extension}0\to\mathcal O_{H'}(1)\to\mathbb D(\mathcal I_{L,H})\to\mathcal O_{L'}\to 0
\end{equation}
for a hyperplane section $H'\subset X$ and a line $L'\subset X$. Note that $H=\mathrm{supp}(\mathcal I_{L,H})=\mathrm{supp}(\mathbb D(\mathcal I_{L,H}))=H'$. Also, the line $L$ can be described as the locus of points $x\in X$, for which the rank of the map $\mathcal U(x)\to\mathcal Q^\vee(x)$ induced by (\ref{triple28}) is smaller than $2$. Meanwhile, the line $L'$ is the locus of points, where the rank of the map $\mathcal Q(x)\to\mathcal U(1)(x)$ induced by (\ref{triple29}) is smaller than $2$. Since these maps are dual to each other by construction, we have $L=L'$.

Since $\mathbb D(\mathcal I_{L,H})$ is tilt-semistable, the extension (\ref{extension}) should be nontrivial. The long exact sequence of $\mathrm{Ext}$ groups associated with an exact triple $0\to\mathcal O_X\to\mathcal O_X(1)\to\mathcal O_H(1)\to 0$ implies that $\mathrm{Ext}^1(\mathcal O_L,\mathcal O_H(1))$ is isomorphic to the kernel of the induced morphism $h:\mathrm{Ext}^2(\mathcal O_L,\mathcal O_X)\to\mathrm{Ext}^2(\mathcal O_L,\mathcal O_X(1))$. The spectral sequence of local and global Ext's together with the ``fundamental local isomorphism'' \cite[Theorem 4.5]{AK} implies that $\mathrm{Ext}^2(\mathcal O_L,\mathcal O_X)\cong H^0(\mathcal Ext^2(\mathcal O_L,\mathcal O_X))\cong H^0(\Lambda^2\mathcal N_{L/X})$. 
As we saw in the proof of Lemma \ref{2 -1 -3/10}, $\Lambda^2\mathcal N_{L/X}\cong\mathcal O_L$. We also get $\mathrm{Ext}^2(\mathcal O_L,\mathcal O_X(1))\cong H^0(\mathcal Ext^2(\mathcal O_L,\mathcal O_X(1)))\cong H^0(\mathcal O_L(1))$. Note that the morphism $H^0(\mathcal O_L)\cong\mathrm{Ext}^2(\mathcal O_L,\mathcal O_X)\overset{h}{\to}\mathrm{Ext}^2(\mathcal O_L,\mathcal O_X(1))\cong H^0(\mathcal O_L(1))$ is given by multiplication with the equation of $H$. We obtain that $\mathrm{ext}^1(\mathcal O_L,\mathcal O_H(1))=1$ if $L\subset H$, and equals $0$ otherwise. So, we get a bijection between isomorphism classes of the sheaves $\mathbb D(\mathcal I_{L,H})$ and the points of $\mathbb F$ via exact triples (\ref{extension}). 

In order to construct a family of the sheaves $\mathbb D(\mathcal I_{L,H})$ with the base $\mathbb F$, we firstly construct families of the sheaves $\mathcal O_L$ and $\mathcal O_H$ with the same base. Starting from the sheaf $\mathbb I_l$ on $X\times\mathbb P^2$ constructed above, we can consider the sheaf $\mathbb I_l^{\vee\vee}$, which is a reflexive sheaf of rank one, hence it is a line bundle. The sheaf $\mathbb I_l$ is the cokernel of the composition morphism (\ref{I_l}), which is a monomorphism, and we calculate $c_1(\mathbb I_l)=(0,2)$. Since $\mathbb I_l$ is locally free in the complement of a closed subset of codimension 2, we have $c_1(\mathbb I_l^{\vee\vee})=c_1(\mathbb I_l)$, so $\mathbb I_l^{\vee\vee}\cong q^*(\mathcal O_{\mathbb P^2}(2))$. In particular, $\mathbb I_l^{\vee\vee}$ is trivial on fibers $X\times\{p\}$ for $p\in\mathbb P^2$. So, we define $\mathbb I_{o,l}$ as the cokernel of the natural map $\mathbb I_l\to\mathbb I_l^{\vee\vee}$. Then, we put $\mathbb I'_{o,l}:=\widetilde{\pi}^*(\mathbb I_{o,l})$, it is the family of the sheaves $\mathcal O_L$ that we need. Also, we define $\mathbb I_h$ as the cokernel of the composition morphism (\ref{comp3}), composed with the natural map $\widetilde{\pi}^*(\mathbb I_l)\to\widetilde{\pi}^*(\mathbb I_l^{\vee\vee})$, it the family of $\mathcal O_H$ with the base $\mathbb F$ that we need.

Since the variety $\mathbb F$ is smooth and $\mathrm{ext}^1(\mathcal O_L,\mathcal O_H(1))=1$ independently of the point $(L,H)\in\mathbb F$, by \cite[Satz 3.(ii)]{BPS}, \cite[Theorem 1.4]{Lan} there is a line bundle
$$\mathcal A_{l,h}=\mathcal Ext_{q_4}^1(\mathbb I'_{o,l},\mathbb I_h\otimes p_4^*(\mathcal O_X(1)))$$
on $\mathbb F$, such that the fiber of $\mathcal A_{l,h}$ over the point $(L,H)$ is naturally isomorphic to $\mathrm{Ext}^1(\mathcal O_L,\mathcal O_H(1))$. Also, we have $\mathcal Ext_{q_4}^0(\mathbb I'_{o,l},\mathbb I_h\otimes p_4^*(\mathcal O_X(1)))=0$. Now \cite[Corollary 4.5]{Lan} implies that there is an universal extension $\mathbb I_4$ on $\mathbb P(\mathcal A_{l,h}^\vee)\times X\times \mathbb F\cong X\times \mathbb F$
$$0\to\mathbb I_h\otimes p_4^*(\mathcal O_X(1))\to\mathbb I_4\to \mathbb I'_{o,l}\to 0,$$
such that the fiber of $\mathbb I_4$ over the point $(L,H)\in\mathbb F$ is a nontrivial extension of $\mathcal O_L$ by $\mathcal O_H(1)$, which is isomorphic to $\mathbb D(\mathcal I_{L,H})$. The family $\mathbb I_4$ induces a bijective morphism $\Phi_4:\mathbb F\to\mathcal M_X(v))$ (here $v=\mathrm{ch}(\mathbb D(\mathcal I_{L,H})$), which is an isomorphism, since $\mathrm{ext}^1(\mathbb D(\mathcal I_{L,H}),\mathbb D(\mathcal I_{L,H}))=6$ (by duality).

\item 
By Lemma \ref{0 1 -1/10} sheaves $\mathbb D(\mathcal I_{C,H})$ can be included into exact triples
\begin{equation}\label{triple30}0\to\mathcal U(1)\overset{f_1^\vee}{\to}\mathcal O_X(1)^{\oplus 2}\to \mathbb D(\mathcal I_{C,H})\to 0,\end{equation}
where $f_1^\vee$ is dual to the morphism $f_1$ from (\ref{triple26}). In this case we can simply dualize the composition (\ref{comp}) to obtain a morphism
\begin{equation}\label{comp4}p_1^*(\mathcal U(1))\overset{(p_1^*(e_1))^\vee}{\to}p_1^*(\mathcal O_X(1))\otimes\Hom(\mathcal O_X(-1),\mathcal U)^\vee\overset{(\id\otimes q_1^*(u))^\vee}{\to}p_1^*(\mathcal O_X(1))\otimes q_1^*(\widetilde{\mathcal U}^\vee). 
\end{equation}
We define $\mathbb I_5$ as the cokernel of the composition (\ref{comp4}). Since the source and the target of the composition (\ref{comp4}) are locally free, its restrictions to fibers $X\times\{p\}$ over points $p\in\mathrm{Gr}(2,5)$ are dual to restrictions of (\ref{comp}). Therefore, $\mathbb I_5$ is a family of sheaves $\mathbb D(\mathcal I_{C,H})$, which induces a bijective modular morphism $\Phi_5:\mathrm{Gr}(2,5)\to\mathcal M_X(v)$ ($v=\mathrm{ch}(\mathbb D(\mathcal I_{C,H}))$). Since by duality $\mathrm{ext}^1(\mathbb D(\mathcal I_{C,H}),\mathbb D(\mathcal I_{C,H}))=6$, $\Phi_5$ is an isomorphism.
\end{enumerate}\qed

\subsection{General case with $c_1=-1$}\label{section -1} In this section we describe moduli spaces of Gieseker semistable rank 2 sheaves $E$ on $X$ with $c_1=-1,c_2\ge 6$ and maximal $c_3=\frac{c_2^2}{5}-10\varepsilon(\frac 12-\frac{c_2}{5})$. By Theorem \ref{intro} any such sheaf can be included into an exact sequence (\ref{extension2}).

\begin{lemma}\label{Gr}Choosing an isomorphism class $[E]$ of a slope stable sheaf $E$ that can be written as an extension (\ref{extension2}) is equivalent to choosing an isomorphism class of a sheaf $G_i$ and a two-dimensional subspace of $\mathrm{Ext}^1(G_i,\mathcal O_X(-1))$.
\end{lemma}

\begin{proof}We have $\mathrm{Hom}(\mathcal O_X(-1),G_i)=0$, hence a monomorphism $\mathcal O_X(-1)^{\oplus 2}\hookrightarrow E$ is uniquely determined up to an action of $\mathrm{GL}(2)$, and $G_i$ is uniquely determined by $[E]$ up to isomorphism. Extensions (\ref{extension2}) are classified by a pair of elements of $\mathrm{Ext}^1(G_i,\mathcal O_X(-1))$. If these elements are linearly dependent, then one can (e.~g. using the axiom (TR 3) for distinguished triangles $G_i\to\mathcal O_X(-1)^{\oplus 2}[1]\to E[1]\to G_i[1],0\to\mathcal O_X(-1)[1]\to\mathcal O_X(-1)[1]\to 0$) construct a nonzero morphism $E\to\mathcal O_X(-1)$, contradicting slope stability of $E$. So, the two elements are linearly independent, and, factoring by $\mathrm{GL}(2)$, we get a two-dimensional subspace of $\mathrm{Ext}^1(G_i,\mathcal O_X(-1))$.
	
Conversely, given a two-dimensional subspace of $\mathrm{Ext}^1(G_i,\mathcal O_X(-1))$, we can choose its basis and obtain an extension (\ref{extension2}). Its middle term $E$ is tilt-semistable, at least at $W(E,\mathcal O_X(-1))$. If $E$ is not semistable above this wall, then the classification of destabilizing subobjects and quotients from the proof of Theorem \ref{main} and Lemma \ref{Jordan-Holder} implies that there is a quotient $E\twoheadrightarrow\mathcal O_X(-1)$ or $E\twoheadrightarrow\mathcal O_X(-1)^{\oplus 2}$ (if a quotient $E\twoheadrightarrow F$ destabilizes $E$ above $W(E,F)$, then $\mu(F)<\mu(E)$). In any case we get a nonzero morphism $E\to\mathcal O_X(-1)$, which induces a nonzero morphism $\mathcal O_X(-1)^{\oplus 2}\to\mathcal O_X(-1)$ (since there are no nonzero morphisms from a torsion sheaf $G_i$ to $\mathcal O_X(-1)$). In this situation we again can apply the axiom (TR 3) to get a contradiction with linear independence of two elements of $\mathrm{Ext}^1(G_i,\mathcal O_X(-1))$ corresponding to (\ref{extension2}). Since $W(E,\mathcal O_X(-1))$ is the unique wall for $E$, we get that $E$ should be $\nu_{\alpha,\beta}$-semistable for $\beta<-\frac 12,\alpha\gg 0$, hence $E$ is a 2-semistable sheaf by Proposition \ref{2-stability}, hence $E$ is slope semistable. Since $\rk(E)$ and $c_1(E)$ are coprime, $E$ is in fact slope stable.
\end{proof}

A calculation using long exact sequences of Ext groups associated with exact triples (\ref{triple26}, \ref{triple27}, \ref{triple28}, \ref{triple29}, \ref{triple30}) (twisted with $\mathcal O_X(m)$) shows that $\mathrm{ext}^1(G_i,\mathcal O_X(-1))$ is independent of $G_i$ (for fixed $i$ and $m$). In this calculation we use Grothendieck--Riemann--Roch theorem, Serre duality and the fact that $\mathcal O_X,\mathcal U$ and $\mathcal Q$ are ACM sheaves. Also, 
one can check that $\mathrm{ext}^2(G_i,\mathcal O_X(-1))=0$.

In the notation of Section \ref{torsion} let 
\begin{equation}\label{A}\mathcal A_i:=\mathcal Ext_{q_i}^1(\mathbb I_i\otimes p_1^*(\mathcal O_X(m)),p_1^*(\mathcal O_X(1))),\quad 1\le i\le 5.\end{equation}  
By \cite[Satz 3.(ii)]{BPS}, \cite[Theorem 1.4]{Lan} sheaves $\mathcal A_i$ are locally free and the fiber of $\mathcal A_i$ over an arbitrary point $p\in B_i$ is canonically isomorphic to the space $\mathrm{Ext}^1(G_i,\mathcal O_X(-1))$ for the sheaf $G_i$ corresponding to the point $p$. Now we can prove Theorem \ref{c=-1}.

\textit{Proof of Theorem \ref{c=-1}}. Consider varieties $P=\mathbb P(\mathcal A_i^\vee),X_P=X\times P$ and canonical projections $p_P:X_P\to P,\rho_P:X_P\to X\times B_i$. By \cite[Corollary 4.5]{Lan} there is a universal family of extensions on $X_P$:
\begin{equation}0\to p_P^*(\mathcal O_P(1))\otimes\rho_P^*(p_i^*(\mathcal O_X(-1)))\to\mathcal V\to\rho_P^*(\mathbb I_i\otimes p_i^*(\mathcal O_X(m)))\to 0.
\end{equation}
Applying the functor $\rho_{P,*}$ to this triple, we get a universal extension on $X\times B_i$:
\begin{equation}\label{W}0\to q_i^*(\mathcal A^\vee)\otimes p_i^*(\mathcal O_X(-1))\to\mathcal W\to\mathbb I_i\otimes p_i^*(\mathcal O_X(m))\to 0.
\end{equation}
A sheaf $E$ from (\ref{extension2}) is given by an element $$\xi\in\mathrm{Ext}^1(G_i,\mathcal O_X(-1)^{\oplus 2})\cong\Hom(\mathrm{Ext}^1(G_i,\mathcal O_X(-1))^\vee,\mathbb C^2).$$ By construction the extension (\ref{W}), restricted to a fiber $X\times\{p\},p\in B_i$, together with the extension (\ref{extension2}), where $E=E_\xi$, can be included into a commutative pushout diagram
\begin{equation}\label{comm}
	\xymatrix{
		0 \ar[r] & \mathrm{Ext}^1(G_i,\mathcal O_X(-1))^\vee\otimes\mathcal O_X(-1) \ar[r]\ar[d]^\xi & \mathcal W|_{X\times\{p\}} \ar[r]\ar[d] & G_i\ar@{=}[d]\ar[r] & 0 \\
		0 \ar[r] & \mathcal O_X(-1)^{\oplus 2} \ar[r] & E_\xi \ar[r] & G_i\ar[r] & 0}
\end{equation}

Lemma \ref{Gr} implies that stable sheaves $E_\xi$ correspond to surjective maps $\xi$ up to an action of $\mathrm{GL}(2)$, that is, to points of Grassmanization $\mathcal Gr(\mathcal A_i^\vee,2)$ of two-dimensional quotients of fibers of $\mathcal A_i^\vee$. We globalize the diagram (\ref{comm}) as follows. Let $\widetilde{p}:X\times \mathcal Gr(\mathcal A_i^\vee,2)\to X,\widetilde{q}:X\times \mathcal Gr(\mathcal A_i^\vee,2)\to\mathcal Gr(\mathcal A_i^\vee,2)$ be canonical projections. Recall that there is also a canonical projection $\Pi:\mathcal Gr(\mathcal A_i^\vee,2)\to B_i$. Let $\Pi^*(\mathcal A_i^\vee)\twoheadrightarrow\widetilde{\mathcal Q}$ be the tautological quotient bundle on $\mathcal Gr(\mathcal A_i^\vee,2)$. Then (\ref{W}) induces a commutative pushout diagram with exact rows
\begin{equation}\label{U}\xymatrix{
		 \widetilde{q}^*\Pi^*(\mathcal A_i^\vee)\otimes\widetilde{p}^*(\mathcal O_X(-1)) \ar[d]\ar@{^{(}->}[r] & (\id\times\Pi)^*(\mathcal W) \ar@{>>}[r]\ar[d] & (\id\times\Pi)^*(\mathbb I_i\otimes p_i^*(\mathcal O_X(m)))\ar@{=}[d]  \\
		 \widetilde{q}^*(\widetilde{\mathcal Q}) \ar@{^{(}->}[r] & \mathcal E \ar@{>>}[r] & (\id\times\Pi)^*(\mathbb I_i\otimes p_i^*(\mathcal O_X(m))) }
\end{equation}
Here $\mathcal E$ is a family of sheaves from $\mathcal M_X(v)$ over the base $\mathcal Gr(\mathcal A_i^\vee,2)$, which induces a bijective modular morphism
\begin{equation}\Phi:\mathcal Gr(\mathcal A_i^\vee,2)\to \mathcal M_X(v),
\end{equation}
\begin{equation}\dim \mathcal M_X(v)=\dim \mathcal Gr(\mathcal A_i^\vee,2)=2(\mathrm{ext}^1(G_i,\mathcal O_X(-1))-2)+6.
\end{equation}
In order to conclude, it suffices to prove that $\mathcal M_X(v)$ is smooth, as in Section \ref{torsion}. 
For this it suffices to prove that $\mathrm{ext}^1(E,E)=\dim \mathcal M_X(v)$ for any $[E]\in\mathcal M_X(v)$.
Stability of $E$ implies that $\mathrm{Hom}(E,\mathcal O_X(-1)^{\oplus 2})=0,\mathrm{Hom}(E,E)\cong\mathbb C$. 
The arguments from Section \ref{torsion} imply that $\mathcal I_{C,H}$ and $\mathcal I_{L,H}$ are simple sheaves, hence $\mathbb D(\mathcal I_{C,H})$ and $\mathbb D(\mathcal I_{L,H})$ are simple, and $\mathcal O_H$ is simple too, so $\mathrm{Hom}(G_i,G_i)\cong\mathbb C$. 
Also $\mathrm{Hom}(\mathcal O_X(-1)^{\oplus 2},G_i)=0$, hence $\Hom(E,G_i)\cong\mathbb C$. And $\mathrm{Ext}^2(E,\mathcal O_X(-1)^{\oplus 2})=0$, since $\mathrm{Ext}^2(G_i,\mathcal O_X(-1)^{\oplus 2})=0$. It follows that the long exact sequence of $\mathrm{Ext}$ groups associated with (\ref{extension2}) induces a short exact sequence
\begin{equation}\label{ext1}0\to\mathrm{Ext}^1(E,\mathcal O_X(-1)^{\oplus 2})\to\mathrm{Ext}^1(E,E)\to\mathrm{Ext}^1(E,G_i)\to 0.
\end{equation}
Analogously we get a short exact sequence
$$0\to\mathrm{Hom}(\mathcal O_X(-1)^{\oplus 2},\mathcal O_X(-1)^{\oplus 2})\to\mathrm{Ext}^1(G_i,\mathcal O_X(-1)^{\oplus 2})\to\mathrm{Ext}^1(E,\mathcal O_X(-1)^{\oplus 2})\to 0,$$
which implies that $\mathrm{ext}^1(E,\mathcal O_X(-1)^{\oplus 2})=2(\mathrm{ext}^1(G_i,\mathcal O_X(-1))-2)$. 

We noted that $\Hom(\mathcal O_X(-1)^{\oplus 2},G_i)=0$ and one can check that $\mathrm{Ext}^1(\mathcal O_X(-1)^{\oplus 2},G_i)=0$, hence $\mathrm{ext}^1(E,G_i)=\mathrm{ext}^1(G_i,G_i)=6$ (about the last equality see Section \ref{torsion}).  
Now (\ref{ext1}) implies that $\mathrm{ext}^1(E,E)=\dim \mathcal Gr(\mathcal A_i^\vee,2)=\dim \mathcal M_X(v)$. This dimension may be calculated as a function of $m$ using the calculation of $\mathrm{ext}^1(G_i,\mathcal O_X(-1))$, which was mentioned before this proof. \qed

\subsection{General case with $c_1=0$}\label{section 0} In this section we describe moduli spaces of Gieseker semistable rank 2 sheaves $E$ on $X$ with $c_1=0,c_2\in\{5, 7\}\cup\mathbb Z_{\ge 9}$ and maximal $c_3=\frac{c_2^2+c_2+4}{5}-10\varepsilon(-\frac{c_2}{5})$. 
Again Proposition \ref{2-stability} and Theorem \ref{main} imply that any such sheaf $E$ can be included into an exact triple of the form (\ref{extension3}).

\begin{lemma}\label{P}Choosing an isomorphism class $[E]$ of a Gieseker semistable sheaf $E$ that can be written as an extension (\ref{extension3}) is equivalent to choosing an isomorphism class of a sheaf $G'_i$ and a one-dimensional subspace of $\mathrm{Ext}^1(G'_i,\mathcal U)$.
\end{lemma}

\begin{proof}We have $\mathrm{Hom}(\mathcal U,G'_i)=0$, hence a monomorphism $\mathcal U\hookrightarrow E$ is uniquely determined up to proportionality, and $[G'_i]$ is uniquely determined by $[E]$. An extension (\ref{extension3}) is given by an element of $\mathrm{Ext}^1(G'_i,\mathcal U)$. Since a rank two semistable sheaf has to be torsion-free, this element is nonzero, and we get a one-dimensional subspace of $\mathrm{Ext}^1(G'_i,\mathcal U)$.

Conversely, given a one-dimensional subspace of $\mathrm{Ext}^1(G'_i,\mathcal U)$, we obtain an extension of the form (\ref{extension3}). Its middle term $E$ is tilt-semistable at $W(E,\mathcal U)$. If $E$ is not semistable above this wall, then the description of destabilizing subobjects and quotients from Theorem \ref{main} implies that there is a quotient $E\twoheadrightarrow\mathcal U$. Since $\mathrm{Hom}(G'_i,\mathcal U)=0$ and $\mathrm{Hom}(\mathcal U,\mathcal U)\cong\mathbb C$, in this case the extension (\ref{extension3}) splits. Therefore, $E$ has to be tilt-semistable above the unique wall $W(E,\mathcal U)$. If $E$ is not tilt-stable above this wall, then in view of Lemma \ref{tilt-stable} $E$ is destabilized by a subobject $F$ with $\mathrm{ch}_{\le 2}(F)=\frac 12\cdot\mathrm{ch}_{\le 2}(E)$. As in Theorem \ref{main}, denote $\mathrm{ch}(E)=1+dH^2+eH^3$, then $\mathrm{ch}_{\le 2}(F)=\mathrm{ch}_{\le 2}(E/F)=1+\frac d2H^2$. In this case \cite[Lemma 4.2]{Vass} gives a bound $e\le\frac{d^2}{4}-\frac d2$, but $\frac{d^2}{4}-\frac d2<\frac{d^2}{2}-\frac{d}{10}-\frac{1}{25}$ for $d\le-\frac{9}{10}$.  
Hence, $E$ is tilt-stable above $W(E,\mathcal U)$, so $E$ is a 2-stable sheaf by Proposition \ref{2-stability}, hence it is Gieseker stable.
\end{proof}

As in the case $c_1=-1$, a calculation with exact triples (\ref{triple26}, \ref{triple27}, \ref{triple28}, \ref{triple29}, \ref{triple30}) shows that $\mathrm{ext}^1(G'_i,\mathcal U)$ is independent of $G'_i$ (for fixed $i$ and $m$). In the cases $i\in\{1, 2, 3, 5\}$ the calculation uses additionally the fact that $\mathcal U\otimes\mathcal U$ is an ACM sheaf (\cite[proof of Proposition 5.6]{Faenzi}). In the case $i=4$ the calculation of this Ext group is non-trivial and uses the theory of stability conditions, so we include it here.

The exact triple (\ref{triple29}), twisted by $\mathcal O_X(m)$, induces an exact sequence
\begin{equation}\label{ext}0\to\mathrm{Hom}(\mathcal O_X(m+1)\oplus\mathcal U(m+1),\mathcal U)\to\mathrm{Hom}(\mathcal Q(m),\mathcal U)\to\mathrm{Ext}^1(\mathbb D(\mathcal I_{L,H})(m),\mathcal U)\to 0.\end{equation}
The dimension of the Hom group from the left equals $h^0(\mathcal U(-m-1))+h^0(\mathcal U\otimes\mathcal U(-m))=\chi(\mathcal U(-m-1))+\chi(\mathcal U\otimes\mathcal U(-m))$ (we use the ACM property of $\mathcal U$ and $\mathcal U\otimes\mathcal U$ together with Serre duality, stability of $\mathcal U$ and the exact triple (\ref{tautological}), tensored with $\mathcal U(m)$). It remains to find $h^0(\mathcal Q^\vee\otimes\mathcal U(-m))$.

The exact triple (\ref{tautological}), dualized and tensored with $\mathcal U(-m)$, induces an exact sequence
\begin{equation}\label{cohomology2}0\to H^0(\mathcal Q^\vee\otimes\mathcal U(-m))\to H^0(\mathcal U(-m)^{\oplus 5})\to H^0(\mathcal U\otimes\mathcal U(1-m))\to H^1(\mathcal Q^\vee\otimes\mathcal U(-m))\to 0.
\end{equation}
Serre duality implies that $h^1(\mathcal Q^\vee\otimes\mathcal U(-m))=h^2(\mathcal Q\otimes\mathcal U^\vee(m-2))$. Note that $H^2(\mathcal Q\otimes\mathcal U^\vee(m-2))\cong\mathrm{Ext}^2(\mathcal U,\mathcal Q(m-2))$, and (\ref{triple29}) induces an isomorphism $\mathrm{Ext}^2(\mathcal U,\mathcal Q(m-2))\cong\mathrm{Ext}^1(\mathcal U,\mathbb D(\mathcal I_{L,H})(m-2))$. However, in the proof of Theorem \ref{main} we established that this Ext group vanishes. Therefore, (\ref{cohomology2}) implies that $h^0(\mathcal Q^\vee\otimes\mathcal U(-m))=h^0(\mathcal U(-m)^{\oplus 5})-h^0(\mathcal U\otimes\mathcal U(1-m))=\chi(\mathcal U(-m)^{\oplus 5})-\chi(\mathcal U\otimes\mathcal U(1-m))$. Now (\ref{ext}) and Grothendieck--Riemann--Roch allow to compute $\mathrm{ext}^1(\mathbb D(\mathcal I_{L,H})(m),\mathcal U)=5m^2+2m-1$.

Let 
\begin{equation}\label{tilde A}\widetilde{\mathcal A}_i:=\mathcal Ext_{q_i}^1(\mathbb I_i\otimes p_1^*(\mathcal O_X(m)),p_1^*(\mathcal U)),\quad 1\le i\le 5.\end{equation}  
By \cite[Satz 3.(ii)]{BPS} or \cite[Theorem 1.4]{Lan} sheaves $\widetilde{\mathcal A}_i$ are locally free and the fiber of $\widetilde{\mathcal A}_i$ over an arbitrary point $p\in B_i$ is canonically isomorphic to the space $\mathrm{Ext}^1(G'_i,\mathcal U)$ for the sheaf $G'_i$ corresponding to the point $p$ (\cite[Theorem 1.4]{Lan} can be applied in this situation, because $\mathrm{Ext}^2(G'_i,\mathcal U)=0$). Now we can prove Theorem \ref{c=0}.

\textit{Proof of Theorem \ref{c=0}.} Consider varieties $\widetilde{P}=\mathbb P(\widetilde{\mathcal A}_i^\vee),X_{\widetilde{P}}=X\times \widetilde{P}$ and canonical projections $p_{\widetilde{P}}:X_{\widetilde{P}}\to \widetilde{P},\rho_{\widetilde{P}}:X_{\widetilde{P}}\to X\times B_i$. By \cite[Corollary 4.5]{Lan} (which can be applied, since $\mathrm{Hom}(G'_i,\mathcal U)=0$) there is a universal family of extensions on $X_{\widetilde{P}}$:
	\begin{equation}0\to p_{\widetilde{P}}^*(\mathcal O_{\widetilde{P}}(1))\otimes\rho_{\widetilde{P}}^*(p_i^*(\mathcal O_X(-1)))\to\widetilde{\mathcal V}\to\rho_{\widetilde{P}}^*(\mathbb I_i\otimes p_i^*(\mathcal O_X(m)))\to 0.
	\end{equation}
	
This family induces a bijective modular morphism
\begin{equation}\widetilde{\Phi}:\mathbb P(\widetilde{\mathcal A}_i^\vee)\to \mathcal M_X(v),
\end{equation}
\begin{equation}\dim \mathcal M_X(v)=\dim \mathbb P(\widetilde{\mathcal A}_i^\vee)=\mathrm{ext}^1(G'_i,\mathcal O_X(-1))+5.
\end{equation}
Again it suffices to prove that $\mathcal M_X(v)$ is smooth. We need to check that $\mathrm{ext}^1(E,E)=\dim \mathcal M_X(v)$ for any $[E]\in\mathcal M_X(v)$.
Stability of $E$ implies that $\mathrm{Hom}(E,\mathcal U)=0,\mathrm{Hom}(E,E)\cong\mathbb C$. 
Also $\mathrm{Hom}(\mathcal U,G'_i)=0$, hence $\Hom(E,G'_i)\cong\mathbb C$. We have $\mathrm{Ext}^2(E,\mathcal U)=0$, since $\mathrm{Ext}^2(G'_i,\mathcal U)=0$. It follows that the long exact sequence of $\mathrm{Ext}$ groups associated with (\ref{extension3}) induces a short exact sequence
\begin{equation}\label{ext2}0\to\mathrm{Ext}^1(E,\mathcal U)\to\mathrm{Ext}^1(E,E)\to\mathrm{Ext}^1(E,G'_i)\to 0.
\end{equation}
Analogously we get a short exact sequence
$$0\to\mathrm{Hom}(\mathcal U,\mathcal U)\to\mathrm{Ext}^1(G'_i,\mathcal U)\to\mathrm{Ext}^1(E,\mathcal U)\to 0,
$$
which implies that $\mathrm{ext}^1(E,\mathcal U)=\mathrm{ext}^1(G'_i,\mathcal O_X(-1))-1$. 

We proved in Theorem \ref{main} that $\mathrm{Ext}^1(\mathcal U,G'_i)=0$, hence $\mathrm{Ext}^1(E,G'_i)\cong\mathrm{Ext}^1(G'_i,G'_i)=6$. Now (\ref{ext2}) implies that $\mathrm{ext}^1(E,E)=\mathrm{ext}^1(G'_i,\mathcal O_X(-1))+5=\dim \mathbb P(\widetilde{\mathcal A}_i^\vee)=\dim \mathcal M_X(v)$. We showed how to calculate this number in the case $i=4$, and in other cases it can be calculated directly. \qed

Calculations performed in Theorems \ref{c=-1} and \ref{c=0} imply that $\mathrm{ext}^1(G_i,\mathcal U)>0$ and $\mathrm{ext}^1(G'_i,\mathcal O_X(-1))>1$. From this fact and from Propositions \ref{existence} and \ref{existence2} the following corollary follows.

\begin{corollary} The bounds of Theorem \ref{intro} are exact, that is, in all cases there exist Gieseker semistable rank two sheaves $E$ with maximal $c_3(E)$, which is permitted by Theorem \textit{loc. cit.}
\end{corollary}

\subsection{Proof of a conjecture}\label{conjecture} Now we can prove a conjecture that we posed in \cite[Conjecture 4.5]{Vass}. Namely, we have the following result.

\begin{proposition}\label{conj proof}Suppose that $E\in\mathrm{Coh}(X)$ is a Gieseker semistable sheaf on $X$ with $\mathrm{ch}(E)=2+cH+dH^2+eH^3$. If $c=-1, d\le-\frac{7}{10}$ and $e$ is maximal, then a general such sheaf may be included into an exact triple $0\to\mathcal O_X(-1)^{\oplus 2}\to E\to\mathcal O_S(D)\to 0,$ where $S\subset X$ is smooth hyperplane section and $D$ is a divisor on $S$ defined in \cite[Proposition 4.4]{Vass}. Respectively, for $c=0, d\le-\frac 95$ and maximal $e$ a general such sheaf $E$ may be included into an exact triple $0\to\mathcal U\to E\to\mathcal O_S(D)\to 0$ for the corresponding smooth hyperplane section $S$ and a divisor $D$ on $S$.
\end{proposition}

\begin{proof}According to Theorem \ref{main} any such sheaf can be included into an exact triple (\ref{extension2}) for $c=-1$, resp., (\ref{extension3}) for $c=0$. Firstly, we prove that for a general $E$ the support of $G_i$ (resp., of $G'_i$) is smooth. Indeed, the support $Z_i$ of the sheaf $\mathbb I_i$ is a Cartier divisor in $X\times B_i$, which has some bidegree $(n_i,m_i)\in\mathrm{Pic}(X\times B_i)\cong\mathbb Z\oplus\mathbb Z$. Since the intersection of $Z_i$ with the fiber over any $p\in B_i$ is set-theoretically a hyperplane section $H\subset X$, the fiber of $Z_i$ over $p$ is given by the equation of $H$ raised to the power of $n_i$. In any case $Z_i$ is flat over $B_i$, and we get a regular surjective map $f_i:B_i\to\mathbb P^6$, where $\mathbb P^6$ is the Hilbert scheme of hyperplane sections of $X$. By Bertini's Theorem there is a dense open subset $U\subset\mathbb P^6$ corresponding to smooth hyperplane sections $S\subset X$. It follows that for any $[E]$ from $\Pi^{-1}(f_i^{-1}(U))$ (resp. $\widetilde{\Pi}^{-1}(f_i^{-1}(U))$) the support of $G_i$ (resp. $G'_i$) is smooth, hence it is a del Pezzo surface of degree 5.
	
It remains to prove that $G_i$ (or, equivalently, $G'_i$) are (direct images of) line bundles on $S$. Let $i:S=\mathrm{supp}(G_i)\to X$ be the canonical injection. Since $G_i$ are pure (by Lemma \ref{minimal}) and $c_1(G_i)=H$, it follows that canonical morphisms $g_i:G_i\to i_*i^*(G_i)$ are isomorphisms. Indeed, $i_*i^*(G_i)\cong G_i\otimes\mathcal O_S$ and $g_i$ are epimorphisms (it follows from $i$ being a closed embedding and $G_i$ being quasi-coherent sheaves). Hence, the support of $i_*i^*(G_i)$ also equals $S$, and $c_1(i_*i^*(G_i))=H$ (since $c_1(\ker g_i)$ cannot be negative, and since $c_1(i_*i^*(G_i))$ cannot vanish in view of Asymptotic Riemann--Roch \cite[Theorem 1.1.24]{Laz}) on $S$ and Riemann--Roch on $X$). So, $\ker g_i$  has vanishing rank and $c_1$, and for $\ker g_i\neq 0$ this contradicts the purity of $G_i$. We obtain that $G_i=i_*(\widetilde{G}_i)$ for a torsion-free rank one sheaf $\widetilde{G}_i$ on $S$. If $\widetilde{G}_i$ is not a line bundle, then it is not reflexive (see e.~g. \cite[p.~76]{OSS}) and we can consider an exact triple
\begin{equation}\label{non-refl}0\to\widetilde{G}_i\to\widetilde{G}_i^{\vee\vee}\to T\to 0,\ \dim T=0.\end{equation}
After applying the functor $i_*$ the sequence (\ref{non-refl}) remains exact. It follows that $\nu_{\alpha,\beta}(i_*(\widetilde{G}_i^{\vee\vee}))=\nu_{\alpha,\beta}(i_*(\widetilde{G}_i))$ for all $(\beta,\alpha)\in\mathbb H$. If $i_*(\widetilde{G}_i^{\vee\vee})$ is not tilt-semistable, then there is a quotient $i_*(\widetilde{G}_i^{\vee\vee})\twoheadrightarrow Q$, where $Q$ is a tilt-semistable object with $\nu_{\alpha,\beta}(Q)<\nu_{\alpha,\beta}(i_*(\widetilde{G}_i^{\vee\vee}))$. However, then there is a nonzero induced map $G_i\cong i_*(\widetilde{G}_i)\to Q$ (it is nonzero, since $\nu_{\alpha,\beta}(Q)$ is finite, hence $\Hom(T,Q)=0$), contradicting tilt-semistability of $G_i$. So, $i_*(\widetilde{G}_i^{\vee\vee})$ is tilt-semistable, but $\mathrm{ch}_{\le 2}(i_*(\widetilde{G}_i^{\vee\vee}))=\mathrm{ch}_{\le 2}(G_i)$ and $\mathrm{ch}_3(i_*(\widetilde{G}_i^{\vee\vee}))>\mathrm{ch}_3(G_i)$, which contradicts Lemmas \ref{0 1 -1/2}--\ref{0 1 -3/10}. Hence $G_i\cong\mathcal O_S(D)$. The divisor $D$ may be determined using the calculation of Chern characters from \cite[Proposition 4.4]{Vass}.
\end{proof}

	\vspace{5mm}
	
	\noindent
	Danil A.~Vassiliev\\
	National Research University Higher School of Economics, Russian Federation\\
	AG Laboratory, HSE, 6 Usacheva str., Moscow, Russia, 119048\\
	Steklov Mathematical Institute of Russian Academy of Sciences, Moscow, Russia\\
	8 Gubkina St., Moscow 119333, Russia\\
	\textit{E-mail}:{\ danneks@yandex.ru}

\end{document}